\documentclass{amsart}

\usepackage{graphicx}
\usepackage{amssymb}
\usepackage[latin1]{inputenc} 
\newtheorem{theorem}{Theorem}[section]

\newtheorem{lemma}[theorem]{Lemma}

\theoremstyle{definition}
\newtheorem{definition}[theorem]{Definition}

\theoremstyle{remark}
\newtheorem{remark}[theorem]{Remark}
\theoremstyle{proposition}

\theoremstyle{corollary}
\newtheorem{corollary}{Corollary}

\numberwithin{equation}{section}

\begin{document}
	
\title{$M\backslash L$ near $3$}


\author[D. Lima]{Davi Lima}

\address{Davi Lima: Instituto de Matem\'atica, UFAL, Av. Lourival Melo Mota s/n, Maceio, Alagoas, Brazil}

\email{davimat@impa.br}
\thanks{The first author was partially supported by Fondation Louis D.}


\author[C. Matheus]{Carlos Matheus}
\address{Carlos Matheus: CMLS, \'Ecole Polytechnique, CNRS (UMR 7640),
91128, Palaiseau, France}
\email{matheus.cmss@gmail.com}

\author[C. G. Moreira]{Carlos Gustavo Moreira}
\address{Carlos Gustavo Moreira: IMPA, Estrada Dona Castorina, 110. Rio de Janeiro, Rio de Janeiro-Brazil.}
\email{gugu@impa.br}
\thanks{The third author was partially supported by CNPq and Faperj.}

\author[S. Vieira]{Sandoel Vieira}
\address{Sandoel Vieira: IMPA, Estrada Dona Castorina, 110. Rio de Janeiro, Rio de Janeiro-Brazil.}
\email{sandoelpi@gmail.br}
\thanks{The fourth author was partially supported by CNPq}

\date{\today}




\begin{abstract}
We construct four new elements $3.11>m_1>m_2>m_3>m_4$ of $M\backslash L$ lying in distinct connected components of $\mathbb{R}\setminus L$, where $M$ is the Markov spectrum and $L$ is the Lagrange spectrum. 

These elements are part of a decreasing sequence $(m_k)_{k\in\mathbb{N}}$ of elements in $M$ converging to $3$ and we give some evidence towards the possibility that $m_k\in M\setminus L$ for all $k\geq 1$. In particular, this indicates that $3$ might belong to the closure of $M\setminus L$, so that the answer to Bousch's question about the closedness of $M\setminus L$ might be negative.   
\end{abstract}

\maketitle

\section{Introduction}

The Lagrange and Markov spectra are closed subsets of the real line consisting of the best constants of Diophantine approximations of certain irrational numbers and indefinite binary quadratic forms. 

After the foundational works of A. Markov from 1880, a vast literature describing many aspects of these spectra was developed: see Cusick--Flahive book \cite{CF} for a nice review of the literature on this topic until the mid-eighties. 

Nevertheless, the structure of the complement $M\setminus L$ of the Lagrange spectrum in the Markov spectrum remained a particularly challenging subject. Indeed, the works of Freiman \cite{Fr68}, \cite{Fr73} and Flahive \cite{Fl77} between 1968 and 1977 showed that $M\setminus L$ contains two explicit countable subsets near $3.11$ and $3.29$, and this was essentially all known information about $M\setminus L$ until 2017. 

In their recent works \cite{MM1}, \cite{MM2} and \cite{MM3}, the second and third authors proved that $M\setminus L$ has a rich fractal structure: more concretely, there are three explicit open intervals $I_1, I_2, I_3$ nearby $3.11$, $3.29$ and $3.7$ whose boundaries are included in the Lagrange spectrum $L$ such that $(M\setminus L)\cap I_j = M\cap I_j$, $j=1, 2, 3$ (resp.), are explicit Cantor sets of Hausdorff dimensions at least $0.26$, $0.35$ and $0.53$ (resp.). 

In particular, the articles mentioned in the previous paragraph show that the known portions $(M\setminus L)\cap I_j$, $j=1,2,3$, of $M\setminus L$ are closed subsets. This led T. Bousch to ask the second author whether $M\setminus L$ is a closed subset of $\mathbb{R}$. 

The present paper provides some evidence in favor of the possibility that $M\setminus L$ is not closed. Before stating our main result, we need to introduce some definitions and notations. 

\subsection{Some classical facts about continued fractions} 

The continued fraction expansion of an irrational number $\alpha$ is denoted by 
$$\alpha=[a_0;a_1,a_2,\dots] = a_0 + \frac{1}{a_1+\frac{1}{a_2+\frac{1}{\ddots}}},$$ 
so that the Gauss map $g:(0,1)\to[0,1)$, $g(x)=\dfrac{1}{x}-\left\lfloor \dfrac{1}{x}\right\rfloor$ acts on continued fraction expansions by $g([0;a_1,a_2,\dots]) = [0;a_2,\dots]$. 

Given $\alpha=[a_0;a_1,\dots, a_n, a_{n+1},\dots]$ and $\tilde{\alpha}=[a_0;a_1,\dots, a_n, b_{n+1},\dots]$ with $a_{n+1}\neq b_{n+1}$, recall that $\alpha>\tilde{\alpha}$ if and only if $(-1)^{n+1}(a_{n+1}-b_{n+1})>0$. 

For an irrational number $\alpha=\alpha_0$, the continued fraction expansion $\alpha=[a_0;a_1,\dots]$ is recursively obtained by setting $a_n=\lfloor\alpha_n\rfloor$ and $\alpha_{n+1} = \frac{1}{\alpha_n-a_n} = \frac{1}{g^n(\alpha_0)}$. The rational approximations  
$$\frac{p_n}{q_n}:=[a_0;a_1,\dots,a_n]\in\mathbb{Q}$$ 
of $\alpha$ satisfy the recurrence relations $p_n=a_n p_{n-1}+p_{n-2}$ and $q_n=a_n q_{n-1}+q_{n-2}$ (with the convention that $p_{-2}=q_{-1}=0$ and $p_{-1}=q_{-2}=1$). Moreover, $p_{n+1}q_n-p_nq_{n+1}=(-1)^n$ and $\alpha=\frac{\alpha_n p_{n-1}+p_{n-2}}{\alpha_n q_{n-1}+q_{n-2}}$. In particular, given $\alpha=[a_0;a_1,\dots, a_n, a_{n+1},\dots]$ and $\tilde{\alpha}=[a_0;a_1,\dots,a_n,b_{n+1},\dots]$, we have 
$$\alpha-\tilde{\alpha}=(-1)^n\frac{\tilde{\alpha}_{n+1}-\alpha_{n+1}}{q_n^2(\beta_n+\alpha_{n+1})(\beta_n+\tilde{\alpha}_{n+1})}$$ 
where $\beta_n:=\frac{q_{n-1}}{q_n}=[0;a_n,\dots,a_1]$. 

In general, given a finite string $(a_1,\dots, a_l)\in(\mathbb{N}^*)^l$, we write 
$$[0;a_1,\dots,a_l] = \frac{p(a_1\dots a_l)}{q(a_1\dots a_l)}.$$ 
By Euler's rule, $q(a_1\dots a_l) = q(a_1\dots a_m) q(a_{m+1}\dots a_l) + q(a_1\dots a_{m-1}) q(a_{m+2}\dots a_l)$ for $1\leq m<l$, and $q(a_1\dots a_l) = q(a_l\dots a_1)$. In particular, if $(a_1,\dots, a_l)$ is a palindrome, then $p(a_1\dots a_l) = q(a_l,\dots,a_1)$. 

\subsection{Markov and Lagrange spectra}

Given a bi-infinite sequence $\theta=(\theta_n)_{n\in\mathbb{Z}}\in(\mathbb{N}^*)^{\mathbb{Z}}$, let 
$$\lambda_i(\theta):=[a_i;a_{i+1},a_{i+2},\dots]+[0;a_{i-1}, a_{i-2},\dots].$$
The Markov value $m(\theta)$ of $\theta$ is $m(\theta):=\sup\limits_{i\in\mathbb{Z}} \lambda_i(\theta)$. 

The Markov spectrum is the set $M:=\{m(\theta)<\infty: \theta\in(\mathbb{N}^*)^{\mathbb{Z}}\}$ and the Lagrange spectrum $L$ is the closure of the set of Markov values of periodic words in $(\mathbb{N}^*)^{\mathbb{Z}}$. 

In this article, we study exclusively the portion of $M$ below $\sqrt{12}$ and, for this reason, we assume that all sequences appearing in the sequel consist of $1$ and $2$ (i.e., all sequences in this paper belong to $\{1,2\}^{\mathbb{Z}}$ by default). 

Furthermore, we use subscripts to indicate the repetition of a certain character: in particular, $1_2 2_4$ is the string $112222$. Also, $\overline{a_1,\dots, a_l}$ is the periodic word obtained by infinite concatenation of the string $(a_1,\dots, a_l)$. Moreover, unless explicitly stated otherwise, we indicate the zeroth position $a_0$ of a string $(a_{-m}, \dots, a_{-1}, a_0^*, a_1,\dots, a_n)$ by an asterisk.  

\subsection{Statement of the main result}

For each $k\in\mathbb{N}^{\ast}$, consider the finite string $\underline{\omega}_k:=(2_{2k},1_2,2_{2k+1},1_2,2_{2k+2},1_2)$ and the bi-infinite word $\gamma^1_k:=(\overline{\underline{\omega}}_k\underline{\omega}^{\ast}_k\underline{\omega}_k\bar{2})$ where the asterisk indicates that the $(2k+2)$-th position occurs in the first $2$ in substring $2_{2k+1}$ of $\underline{\omega}_k$. In this context, our goal is to prove the following theorem:

\begin{theorem}\label{t.A} The Markov values $m_k=m(\gamma_k^1)$ form a decreasing sequence converging to $3$ whose first four elements belong to $M\setminus L$. Moreover, these four elements belong to distinct connected components of $\mathbb{R}\setminus L$. 
\end{theorem} 

\begin{remark}\label{r.4Cantors} Even though we will not pursue this direction here, the technique used in \cite{MM1}, \cite{MM2}, \cite{MM3} suggests that it might be possible to extend our discussion below to show that, for each $k\in\{1,2,3,4\}$, the connected component of $\mathbb{R}\setminus L$ containing $m_k$ intersects $M\setminus L$ in a Cantor set of positive Hausdorff dimension. 
\end{remark}

\subsection{Why $M\setminus L$ might not be closed?}

Our construction of the new elements $m_j$, $j\in\{1,2,3,4\}$, of $M\setminus L$ follows the ideas of  Freiman \cite{Fr68}, \cite{Fr73}, Flahive \cite{Fl77} and the second and third authors \cite{MM1}, \cite{MM2}, \cite{MM3}, namely, we selected non-semi-symmetric strings $\underline{\omega}_k$ of odd lengths with some hope of getting \emph{local uniqueness} and \emph{self-replication}  properties. In a nutshell, these properties are: 
\begin{itemize}
\item the local uniqueness asks that any word $\theta\in\{1,2\}^{\mathbb{Z}}$ with Markov value $m(\theta)=\lambda_0(\theta)$ sufficiently close to $m_k$ has the form $$\theta=\dots 2_{2k}1_22_{2k+2}1_22_{2k}1_22^*2_{2k}1_22_{2k+2}1_22_{2k}1\dots$$ 
(up to transposition)
\item the self-replication requires that any word $\theta\in\{1,2\}^{\mathbb{Z}}$ of the form $\theta=\dots 2_{2k}1_22_{2k+2}1_22_{2k}1_22^*2_{2k}1_22_{2k+2}1_22_{2k}1\dots$ whose Markov value $m(\theta)$ is sufficiently close to $m_k$ extends as 
$$\theta= \overline{2_{2k}1_22_{2k+1}1_22_{2k+2}1_2}2_{2k}1_22^*2_{2k}1_22_{2k+2}1_22_{2k}1_22_{2k+1}1_22_{2k+2}1_22_{2k} \dots$$
\end{itemize}

It is not hard to see that these properties imply that $m_k\in M\setminus L$ because they would say that a periodic word $\theta$ with Markov value $m(\theta)$ sufficiently close to $m_k$ must coincide with the periodic word $\theta(\underline{\omega}_k)$ determined by $\underline{\omega}_k$, a contradiction with the fact that $m_k\neq m(\theta(\underline{\omega}_k))$. 

As it turns out, we establish in Section \ref{s.replication} below that the self-replication property holds for every $k\in\mathbb{N}$, but unfortunately\footnote{Our failure is related to the fact (explained in Bombieri's excellent survey \cite{Bo}) that the combinatorics of the words in $\{1,2\}^{\mathbb{Z}}$ with Markov value $3$ is quite intricate. Nevertheless, there is still some hope to get the local uniqueness property for $m_k$ because Proposition 1 in \cite{Mo} seems to indicate that the function $\{1,2\}^{\mathbb{Z}}\ni\theta\mapsto m(\theta)\in\mathbb{R}$ could be injective on $m^{-1}((3,3.0056))$.} we could not find a systematic argument allowing to obtain the local uniqueness property for every $k\in\mathbb{N}$. For this reason, we are forced to content ourselves with a proof of the local uniqueness property for $k\in\{1,2,3,4\}$ in Sections \ref{s.k1}, \ref{s.k2}, \ref{s.k3}, \ref{s.k4}, and a proof of an ``almost uniqueness'' property for all $k\in \mathbb{N}$ in Section \ref{s.almost-uniqueness}. 

In summary, the previous two paragraphs give some support to the possibility that $m_k\in M\setminus L$ for every $k\in\mathbb{N}$ (and, thus, $M\setminus L$ is not closed). 

\subsection{Organization of the paper} In Section \ref{s.preliminaries}, we show that $m_k=m(\gamma_k^1)$ is a decreasing sequence converging to $3$. After that, we obtain the relevant self-replication and local uniqueness properties in Sections \ref{s.replication}, \ref{s.k1}, \ref{s.k2}, \ref{s.k3}, \ref{s.k4}. Then, we complete the proof of Theorem \ref{t.A} in Section \ref{s.tA}. Finally, we establish the almost uniqueness property in Section \ref{s.almost-uniqueness}. 

\section{Preliminaries}\label{s.preliminaries}

Recall that $\underline{\omega}_k:=(2_{2k},1_2,2_{2k+1},1_2,2_{2k+2},1_2)$ is a finite string determining a periodic word $\theta(\underline{\omega}_k)$ and a bi-infinite word $\gamma^1_k:=(\overline{\underline{\omega}}_k\underline{\omega}^{\ast}_k\underline{\omega}_k\bar{2})$ where the asterisk indicates the $(2k+2)$-position occurs at the first $2$ in $2_{2k+1}$ in $\underline{\omega}_k$

\subsection{The Markov value of $\theta(\underline{\omega}_k)$} 

\begin{lemma}\label{l.1}
If $\theta=(a_n)_{n\in \mathbb{Z}}$ contains $(a_n)_{i-1\le n\le i+1}=(222)$ then $\lambda_i(\theta)<2.85$.
\end{lemma}
\begin{proof}
	In fact, $\lambda_i(\theta)=[2;2,...]+[0;2,...]\leq 2+2[0;2,\overline{2,1}]<2.85.$
	\end{proof}
	
\begin{lemma}\label{l.2}
	The Markov value of $\theta(\underline{\omega}_k)$ is attained at the position $2k+2$. In particular, $m(\theta(\underline{\omega}_k))$ is a decreasing sequence converging to $3$. 
\end{lemma}
\begin{proof}
	First, by Lemma \ref{l.1}, $\lambda_i(\theta(\underline{\omega}_k))<2.85$ for $i\in \{1, \dots,2k-2, 2k+3, \dots,4k+1,4k+6, \dots,6k+5\}$. Moreover, if $\alpha_k:=[2_{2k},1_2, 2_{2k+2},...]$ and $\beta_k:=[2_{2k-1},1_2,...]$, then $\beta_k>\alpha_k$. Thus, Lemma 3 in \cite[Chapter 1]{CF} implies that 
	$$\lambda_{2k+2}(\theta(\underline{\omega}_k))=[2,1_2,\alpha_k]+[0;2,\beta_k]>3.$$
Therefore, $m(\theta(\underline{\omega}_k)) = \lambda_i(\theta(\underline{\omega}_k))$ for some $i\in\{0, 2k-1, 2k+2, 4k+2, 4k+5, 6k+6\}$. Since we also have that $\lambda_{2k+2}(\theta(\underline{\omega}_k))>\lambda_{2k+4}(\theta(\underline{\omega}_{k+1}))$ and $\lim_{k\to +\infty}\lambda_{2k+2}(\theta(\underline{\omega}_k))=3$ (because $\lim_{k\to +\infty}\alpha_k=\lim_{k\to +\infty}\beta_k=[2;\bar{2}]$), our task is reduced to show that 
$$\lambda_i(\theta(\underline{\omega}_k)) \leq \lambda_{2k+2}(\theta(\underline{\omega}_k))$$
for each $i\in\{0, 2k-1, 4k+2, 4k+5, 6k+6\}$.

In this direction, note that  
	$$\lambda_0(\theta(\underline{\omega}_k))=[2;2_{2k-1},1_2,2_{2k+1},1_2,2_{2k+2},1_2,2_{2k}...]+[0;1_2,2_{2k+2},1_2,2_{2k+1},1_2,2_{2k},...],$$
	$$\lambda_{2k-1}(\theta(\underline{\omega}_k))=[2;1_2,2_{2k+1},1_2,2_{2k+2},1_2,2_{2k},...]+[0;2_{2k-1},1_2,2_{2k+2},1_2,2_{2k+1},1_2,...],$$
	$$\lambda_{2k+2}(\theta(\underline{\omega}_k))=[2;2_{2k},1_2,2_{2k+2},1_2,2_{2k},1_2,...]+[0;1_2,2_{2k},1_2,2_{2k+2},1_2,2_{2k+1},...],$$
	$$\lambda_{4k+2}(\theta(\underline{\omega}_k))=[2;1_2,2_{2k+2},1_2,2_{2k},1_2,2_{2k+1},1_2,...]+[0;2_{2k},1_2,2_{2k},1_2,2_{2k+2},1_2,...],$$
	$$\lambda_{4k+5}(\theta(\underline{\omega}_k))=[2;2_{2k+1},1_2,2_{2k},1_2,2_{2k+1},1_2,...]+[0;1_2,2_{2k+1},1_2,2_{2k},1_2,...],$$
	$$\lambda_{6k+6}(\theta(\underline{\omega}_k))=[2;1_2,2_{2k},1_2,2_{2k+1},1_2,2_{2k+2},1_2,...]+[0;2_{2k+1},1_2,2_{2k+1},1_2,2_{2k},...]$$
		A direct inspection of these formulas reveals that $\lambda_{2k+2}(\theta(\underline{\omega}_k)) > \lambda_i(\theta(\underline{\omega}_k))$ for each $i\in \{0, 2k-1, 4k+5, 6k+6\}$. Thus, it suffices to prove that 
		$$\lambda_{2k+2}(\theta(\underline{\omega}_k))>\lambda_{4k+2}(\theta(\underline{\omega}_k)).$$
		
		For this sake, let us write 
		$$\lambda_{2k+2}(\theta(\underline{\omega}_k))-\lambda_{4k+2}(\theta(\underline{\omega}_k))=A_k-D_k+B_k-C_k,$$
		where
		$$A_k=[0;2_{2k},1_2,2_{2k+2},1_2,\overline{\underline{\omega}}_k], \quad D_k=[0;2_{2k},1_2,2_{2k},1_2,2_{2k+2},1_2,2_{2k+1},1_2,2_{2k},\overline{\underline{\omega}^t_k}],$$
		$$B_k=[0;1_2,2_{2k},1_2,2_{2k+2},1_2,2_{2k+1},1_2,2_{2k},\overline{\underline{\omega}^t_k}], \quad C_k=[0;1_2,2_{2k+2},1_2,\overline{\underline{\omega}}_k],$$
		and $\underline{\omega}^t_k$ is the transpose of $\underline{\omega}_k$. 
		
		Observe that the Gauss map $g$ acts by $g^{4k+2}(A_k)=g^{2k+2}(C_k):=1/x$ and $g^{4k+2}(D_k)=g^{2k+2}(B_k):=1/y.$
		Since
		$$A_k-D_k=\dfrac{(g^{4k+2}(A_k))^{-1}-(g^{4k+2}(D_k))^{-1}}{(xq_{4k+2}+q_{4k+1})(yq_{4k+2}+q_{4k+1})}$$
		and $$B_k-C_k=\dfrac{(g^{2k+2}(B_k))^{-1}-(g^{2k+2}(C_k))^{-1}}{(x\tilde{q}_{2k+2}+\tilde{q}_{2k+1})(y\tilde{q}_{2k+2}+\tilde{q}_{2k+1})},$$
		$\dfrac{p_{2k+2+j}}{q_{2k+2+j}}=[0;2_{2k},1_2,2_j]$ and $\dfrac{\tilde{p}_{j+2}}{\tilde{q}_{j+2}}=[0;1_2,2_{j}].$
		We have that
		$A_k-D_k+B_k-C_k$ equals to
		\begin{eqnarray*}
		(x-y)\left [\dfrac{1}{(x\tilde{q}_{2k+2}+\tilde{q}_{2k+1})(y\tilde{q}_{2k+2}+\tilde{q}_{2k+1})}- \dfrac{1}{(xq_{4k+2}+q_{4k+1})(yq_{4k+2}+q_{4k+1})}\right]>0,
		\end{eqnarray*}
	because $x-y>0$ and $q_{4k+i} > \tilde{q}_{2k+i}$ for $i=1,2$. 
	\end{proof}
	
\subsection{The Markov value of $\gamma_k^1$} 

\begin{lemma}\label{l.3}
	The Markov value of $\gamma_k^1$ is attained at the position $2k+2$. In particular, $m(\theta(\underline{\omega}_k))<m(\gamma_k^1)<m(\theta(\underline{\omega}_{k-1})), \ k>1$.
\end{lemma}

\begin{proof} First, let $i$ be the position such that 
$$\lambda_i(\gamma^1_k)=[2;\overline{2}]+[0;1_2,2_{2k+2},1_2,...]=[2;\overline{2}]+[0;\overline{\omega_k^t}].$$
Since $[2;\overline{2}]<[2;2_{2k},1_2,...]$ and $[0;1_2,2_{2k+2},1_2,...]<[0;1_2,2_{2k},1_2,...]$ we have that
$$\lambda_i(\gamma^1_k)<\lambda_{2k+2}(\gamma^1_k).$$
Then, like above, it suffices to prove that $\lambda_{2k+2}(\gamma^1_k)>\lambda_{4k+2}(\gamma^1_k).$ For this sake remember that
	$$\lambda_{2k+2}(\gamma^1_k)=[2;2_{2k},1_2,2_{2k+2},1_2,\underline{\omega}_k,\overline{2}]+[0,1_2,2_{2k},\overline{\underline{\omega}_k^t}].$$
while
	$$\lambda_{4k+2}(\gamma^1_k)=[2;2_{2k},1_2,2_{2k},\overline{\omega_k^t}]+[0;1_2,2_{2k+2},1_2,\underline{\omega}
	_k,\overline{2}].$$
Then, $\lambda_{2k+2}(\gamma^1_k)-\lambda_{4k+2}(\gamma^1_k)=A_k-C_k+B_k-D_k$, where 
	$$A_k=[0;2_{2k},1_2,2_{2k+2},1_2,\underline{\omega}_k,\overline{2}], \quad B_k=[0,1_2,2_{2k},1_2,2_{2k+2},1_2,2_{2k+1},1_2,2_{2k},\overline{\underline{\omega_k^t}}]$$
and 
	$$C_k=[0;2_{2k},1_2,2_{2k},\overline{\omega_k^t}], \quad D_k=[0;1_2,2_{2k+2},1_2,\underline{\omega}_k,\overline{2}].$$
	Again, we observe that $g^{4k+2}(A_k)=g^{2k+2}(D_k):=\dfrac{1}{x}$ and $g^{2k+2}(B_k)=g^{4k+2}(C_k):=\dfrac{1}{y}.$
	We have
	$$A_k-C_k=\dfrac{y-x}{(yq_{4k+2}+q_{4k+1})(xq_{4k+2}+q_{4k+1})}$$
and
	$$B_k-D_k=\dfrac{x-y}{(x\tilde{q}_{2k+2}+\tilde{q}_{2k+1})(y\tilde{q}_{2k+2}+\tilde{q}_{2k+1})}.$$
	This implies $A_k-C_k+B_k-D_k$ is equals to
	$$(x-y)\left [\dfrac{1}{(x\tilde{q}_{2k+2}+\tilde{q}_{2k+1})(y\tilde{q}_{2k+2}+\tilde{q}_{2k+1})}- \dfrac{1}{(xq_{4k+2}+q_{4k+1})(yq_{4k+2}+q_{4k+1})}\right]>0.$$
	Then $\lambda_{2k+2}(\gamma^1_k)>\lambda_{4k+2}(\gamma^1_k).$
	Finally, note that
	\begin{equation}\label{EQ0}
	\lambda_{2k+2}(\gamma^1_k)>\lambda_{2k+2}(\theta(\underline{\omega}_k)),
	\end{equation}
	because since $|\underline{\omega}_k|$ is odd we have
	\begin{equation}\label{EQ1}
	 [2;2_{2k},1_2,2_{2k+2},1_2,\underline{\omega}_k,\overline{2}]>[2;2_{2k},1_2,2_{2k+2},1_2,\overline{\omega}_k,2_{2k},1...]
	 \end{equation}
	and
	\begin{equation}\label{EQ2}
	[0;1_2,2_{2k},\overline{\underline{\omega^t}}_k]=[0;1_2,2_{2k},\overline{\underline{\omega^t}}_k]
	\end{equation}
	adding (\ref{EQ1}) and (\ref{EQ2}) we have (\ref{EQ0}). It is easy to see $m(\theta(\underline{\omega}_{k-1}))>m(\gamma^1_k), \ k>1.$

\end{proof}

\subsection{Prohibited and avoided strings} Given a finite string $\underline{u}=(a_i)_{i=-m}^{n}$, we define $$\lambda^{-}_i(\underline{u}):=\min \{[a_i;a_{i+1},...,a_n,\theta_1]+[0;a_{i-1},...,a_{-m},\theta_2]: \theta_1,\theta_2\in \{1,2\}^{\mathbb{N}}\},$$
 and
 $$\lambda^{+}_i(\underline{u}):=\max \{[a_i;a_{i+1},...,a_n,\theta_1]+[0;a_{i-1},...,a_{-m},\theta_2];\theta_1,\theta_2\in \{1,2\}^{\mathbb{N}}\}.$$
 
 \begin{definition}
 	A finite string $\underline{u}=(a_i)_{i=-m}^{n}$ is:
	\begin{itemize} 
	\item $k$-prohibited if $\lambda^{-}_i(\underline{u})>m(\gamma^1_k)$, for some $-m\leq i\leq n$.     	\item $k$-avoided if $\lambda^+_0(\underline{u})<m(\theta(\underline{\omega}_k))$.
	\end{itemize}
	A word $\theta\in\{1,2\}^{\mathbb{Z}}$ is $(k,\lambda)$-admissible whenever $m(\theta(\underline{\omega}_k))<m(\theta)=\lambda_0(\theta)<\lambda$.
 \end{definition}

These notions are crucial in the study of the self-replication and local uniqueness properties. Indeed, the self-replication is based on the construction of an appropriate finite set of prohibited strings, the local uniqueness relies on the identification of an adequate finite set of prohibited and avoided strings, and the self-replication and local uniqueness properties imply that the Markov value of any $(k,\lambda_k)$-admissible word belongs to $M\setminus L$ whenever $\lambda_k$ is close to $m_k = m(\gamma_k^1)$. 
\begin{remark}\label{r.1}
	 By Lemmas \ref{l.2} and \ref{l.3}, if $\underline{u}$ is $(k-1)$-prohibited, resp. $(k+1)$-avoided, then it is also $k$-prohibited, resp. $k$-avoided. Also, by definition, a $k$-avoided string can not appear in the center of a $(k,\lambda)$-admissible word. 
	\end{remark}
In the sequel, we give basic examples of prohibited, avoided and admissible words. 

\begin{lemma}\label{l.4}
	 The strings $(12^*1)$, $(2^*12)$, $(1112^*22)$, $(21_32^*211)$ (and their transpositions) are $k$-prohibited for all $k\in \mathbb{N}$. 
	
\end{lemma}

\begin{proof} In fact, we have
	
	\begin{enumerate}
	
 \item [(1)]	$\lambda^{-}_0(12^*1)= [2;1,\overline{1,2}]+[0;1,\overline{1,2}]>3.15$;
	
	\smallskip
	
	  \item[(2)]	$\lambda^{-}_0(2^*12)= [2;1,2,\overline{2,1}]+[0;\overline{2,1}]>3.06$;
	
	\smallskip
	
	\item[(3)] $\lambda^{-}_0(1_32^*2_2)= [2;2_2,\overline{2,1}]+[0;1_3,\overline{1,2}]>3.02$; 
	
\smallskip

\item[(4)] $\lambda^{-}_0(21_32^*21_2)= [2;2,1,1,\overline{1,2}]+[0;1,1,1,2,\overline{2,1}]>3.009$.
	\end{enumerate}
Since $m(\gamma_1^1)=3.00558731248699779947\dots$, it follows from Remark \ref{r.1} that the proof of the lemma is complete. 
\end{proof} 

\begin{remark}\label{r.2}
	Let $\theta$ be a $(k,3.009)$-admissible word. It follows from the proof of Lemma \ref{l.4} that:
	\begin{itemize} 
	\item if $\theta=...12...$, then $\theta=...12_2...$;
	\item if $\theta=...21..$, then $\theta=...21_2...$; 
	\item if $\theta=...1_32_2...$, then $\theta=...1_32_21_2...$, and 
	\item if $\theta=...1_32_21_2...$, then $\theta=...1_42_21_2....$. 
	\end{itemize} 
We use this remark systematically in what follows. 
\end{remark}


\begin{corollary}\label{c.1}
Given $k\geq 1$, if $\theta$ is $(k,3.009)$-admissible, then, up to transposition, $\theta=(...1_22^*21_2...)$ or $(...1_22^*2_2...)$.
\end{corollary}

\begin{proof}
	Note that $222$ is $k$-avoided (cf. Lemma \ref{l.1}). Thus, by Remark \ref{r.2}, it follows that, up to transposition, a $(k,3.009)$-admissible word $\theta$ is $\theta=(...1_22^*21_2...)$ or $\theta=(...1_22^*2_2...)$.
\end{proof}

\begin{lemma}\label{l.5}
	The string $21_22^*21_2$ is $k$-avoided for any $k\in \mathbb{N}$.  
\end{lemma}

\begin{proof}
In fact, $\lambda^+_{0}(\theta)= [2;2,1_2,\overline{2,1}]+[0;1_2,2,\overline{2,1}]<2.98.$	
\end{proof}

\begin{corollary}\label{c.2}
   	Given $k\geq 1$, if $\theta$ is $(k,3.009)$-admissible, then, up to transposition, either $\theta=(...1_42^*21_2...)$ or $(...2_21_22^*2_2...)$.
\end{corollary}

\begin{proof}
	By Corollary \ref{c.1}, $\theta=(...1_22^*21_2...)$ or $\theta=(...1_22^*2_2...)$. In the first case, by Lemma \ref{l.5}, $\theta$ extends as $\theta=(...1_32^*21_2...)$. So, by Remark \ref{r.2}, it follows that $\theta$ extends as $(...1_42^*21_2...)$ or $(...2_21_22^*2_2...)$.
	\end{proof}

\section{Replication mechanism for $\gamma_k^1$}\label{s.replication}

In this section, we investigate the extensions of a word $\theta$ containing the string  
\begin{equation}\label{e.replication-0}
\theta_k^0:=2_{2k}1_22_{2k+2}1_22_{2k}1_22^*2_{2k}1_22_{2k+2}1_22_{2k}1
\end{equation}

We write $C(2)=\{x=[0;a_1,a_2,...];a_i\le 2, \ \forall i\ge 1\}.$ Below we need the following lemma. 

\begin{lemma}\label{Le1}
The minimum value of $\dfrac{(n_1+x_1+x_2)(n_2+x_3+x_4)}{(n_3+x_5+x_6)(n_4+x_7+x_8)}$, where $n_i\in \{1,2\}$ for $i=1,2,3,4$ and $x_i\in C(2), i=1,...,8$ is $\dfrac{1}{4}.$ In the same way the minimum value of $\dfrac{n_1+x_1-x_2}{n_2+x_3-x_4}$, where $n_i\in \{1,2\}$
and $x_i\in C(2)$ is greater than 0.2679.
\end{lemma}

\begin{proof}
We minimize the numerator with $(1+[0;\overline{2,1}]+[0;\overline{2,1}])(1+[0;\overline{2,1}]+[0;\overline{2,1}])=(1+2[0;\overline{2,1}])^2=3$. We maximize the denominator with $(2+2[0;\overline{1,2}])^2=12.$ Then
$$\dfrac{(n_1+x_1+x_2)(n_2+x_3+x_4)}{(n_3+x_5+x_6)(n_6+x_7+x_8)}\ge \dfrac{1}{4}.$$ The other estimated follows the same way as the former.
\end{proof}
Below we write $p_j=p(2_j)$ and $q_j=q(2_j)$. Moreover, $\tilde{p}_j=p(1_22_j)$ and $\tilde{q}_j=q(1_22_j)$. 
Note that
$$\dfrac{\tilde{p}_{s+2}}{\tilde{q}_{s+2}}=\dfrac{1}{1+\dfrac{1}{1+\dfrac{p_s}{q_s}}}=\dfrac{p_s+q_s}{p_s+2q_s}.$$
 Since gcd$(p_s+q_s,p_s+2q_s)=1$ we have $\tilde{q}_{s+2}=p_s+2q_s.$ On the other hand, 
 $$\dfrac{p_s}{q_s}=\dfrac{1}{2+\dfrac{p_{s-1}}{q_{s-1}}}=\dfrac{q_{s-1}}{2q_{s-1}+p_{s-1}}.$$
 Since, gcd$(q_{s-1},2q_{s-1}+p_{s-1})=1$ we must have $p_s=q_{s-1}.$ Therefore
 $$\tilde{q}_{s+2}=2q_s+q_{s-1}=q_{s+1}.$$
 If we write $\tilde{\beta}_{s+2}=[0;2_s,1_2]$ then $\tilde{\beta}_{s+2}=\dfrac{\tilde{q}_{s+1}}{\tilde{q}_{s+2}}=\dfrac{q_{s}}{q_{s+1}}=\beta_{s+1}$.
\subsection{Extension from $\theta_k^0$ to $2_{2k}1_22\theta_k^012_2$}

\begin{lemma}\label{l.6} A $(k,3.0055873128)$-admissible word $\theta$ containing  \eqref{e.replication-0} extends as 
$$\theta=...\theta_k^012_2...=...2_{2k}1_22_{2k+2}1_22_{2k}1_22^*2_{2k}1_22_{2k+2}1_22_{2k}1_22_2...$$
\end{lemma}

\begin{proof} If $k=1$, the desired result follows from Remark \ref{r.2} and the fact that 
$$\lambda_0^-(2_{2}1_22_{4}1_22_{2}1_22^*2_{2}1_22_{4}1_22_{2}1_4)>3.0055873128>m(\gamma_1^1).$$ 
If $k\geq 2$, this is an immediate consequence of Remark \ref{r.2}.  
\end{proof}

\begin{lemma}\label{l.7} If $j<k$, then $\lambda^-_0(1_22_{2j}1_22^*2_{2k})>m(\gamma^1_k)$. 
\end{lemma} 

\begin{proof} We remember that $m(\gamma^1_k)=[2;2_{2k},1_2,2_{2k+2},1_2,\underline{\omega}_k,\overline{2}]+[0,1_2,2_{2k},\overline{\underline{\omega}_k^t}].$
Let $C_k:=[2;2_{2k},\overline{1,2}]$ and $D_k:=[0,1_2,2_{2k},\overline{1,2}]$. By definition, $m(\gamma^1_k)< C_k+D_k$. 

Note that 
	$$\lambda_0^-(1_22_{2j}1_22^*2_{2k})\geq \lambda_0^-(1_22_{2k-2}1_22^*2_{2k}) = [2;2_{2k},1_2,2_2,\overline{2,1}]+[0;1_2,2_{2k-2},1_2,\overline{2,1}]:=A_k+B_k$$
	 for each $j<k$. Thus, our task is reduced to prove that $B_k-D_k>C_k-A_k$. 

In order to establish this estimate, we observe that 
$$C_k-A_k=\dfrac{[2;\overline{1,2}]-[1;\overline{2,1}]}{q^2_{2k}([2;\overline{1,2}]+\beta_{2k})([1;\overline{2,1}]+\beta_{2k})}$$
and 
$$B_k -D_k=\dfrac{[2;\overline{2,1}]-[1;\overline{1,2}]}{\tilde{q}^2_{2k}([2;\overline{2,1}]+\tilde{\beta}_{2k})([1;\overline{1,2}]+\tilde{\beta}_{2k})}.$$ 
Thus,
$$\dfrac{B_k-D_k}{C_k-A_k}=\dfrac{q^2_{2k}}{q^2_{2k-1}}\cdot X\cdot Y$$
where 
	$$X=\dfrac{[2;\overline{2,1}]-[1;\overline{1,2}]}{[2;\overline{1,2}]-[1;\overline{2,1}]}>0.464$$
and
	$$Y=\dfrac{([2;\overline{1,2}]+\beta_{2k})([1;\overline{2,1}]+\beta_{2k})}{([2;\overline{2,1}]+\beta_{2k-1})([1;\overline{1,2}]+\beta_{2k-1})}\ge \dfrac{([2;\overline{1,2}]+0.4)([1;\overline{2,1}]+0.4)}{([2;\overline{2,1}]+0.5)([1;\overline{1,2}]+0.5)}>0.864,$$
	because $\beta_{2k}\ge \beta_2=0.4$ and $\beta_{2k-1}\le \beta_1=0.5.$
Using this we have
$$\dfrac{B_k -D_k}{C_k -A_k}>\dfrac{q^2_{2k-1}}{q^2_{2k}}\cdot 0.464\cdot 0.864>1$$
because $q_{2k}=2q_{2k-1}+q_{2k-2}$ and then $\dfrac{q_{2k}}{q_{2k-1}}=2+\beta_{2k}>2.4$.
\end{proof}

\begin{lemma}\label{l.16} If $j<k$ then $\lambda^-_0(1_22_{2j}1_22^*2_{2k+1})>m(\gamma^1_k)$.
\end{lemma}

\begin{proof} Since $\lambda^-_0(1_22_{2j}1_22^*2_{2k+1})=\lambda^-_0(1_22_{2j}1_22^*2_{2k})$, the desired result follows from Lemma \ref{l.7}. 
\end{proof}

\begin{lemma}\label{l.10} If $m<k$, then $\lambda^-_0(2_{2k}1_22^*2_{2m}1_22_2)>m(\gamma^1_k)$.
\end{lemma} 

\begin{proof}
We note that it is suffices to show the case $m=k-1$ because $\lambda^-_0(2_{2k}1_22^*2_{2m}1_22_2)$ increases when $m$ decreases.
For this sake, let us write $$m(\gamma^1_k)<[2;2_{2k},1_2,2_2,\overline{1,2}]+[0;1_2,2_{2k},\overline{1,2}]=C_k +D_k.$$ Then, we shall show that $\lambda^-_0(2_{2k}1_22^*2_{2m}1_22_2)=[2;2_{2k-2},1_2,2_2,\overline{2,1}]+[0;1_2,2_{2k},\overline{2,1}]:=A_k+B_k>C_k+D_k.$ In fact,
	$A_k-C_k=[0;2_{2k-2},1_2,2_2,\overline{2,1}]-[0;2_{2k},1_2,2_{2},\overline{1,2}]$. 
That is,
	$$A_k-C_k=\dfrac{[2;2,1_2,2_2,\overline{1,2}]-[1;1,2_2,\overline{2,1}]}{q^2_{2k-2}([2;2,1_2,2_2,\overline{1,2}]+\beta_{2k-2})([1;1,2_2,\overline{2,1}]+\beta_{2k-2})}.$$
Moreover, $D_k-B_k=[0;1_2,2_{2k},\overline{1,2}]-[0;1_2,2_{2k},\overline{2,1}]$, therefore using that $\tilde{q}_{2k+2}=q_{2k+1}$ and $\tilde{\beta}_{2k+1}$ we have 
	$$D_k-B_k=\dfrac{[2;\overline{1,2}]-[1;\overline{2,1}]}{q^2_{2k+1}([2;\overline{1,2}]+\beta_{2k+1})([1;\overline{2,1}]+\beta_{2k+1})}.$$
Therefore by Lemma \ref{Le1}
	$$\dfrac{A_k-C_k}{D_k-B_k}=\dfrac{q^2_{2k+1}}{q^2_{2k-2}}\cdot X\cdot Y>64 \cdot 0.26\cdot \dfrac{1}{4}>1,$$
	because $q_{2k+1}=2q_{2k}+q_{2k-1}>2(2q_{2k-1}+q_{2k-2})+2q_{2k-2}>8q_{2k-2}$,
where 
	$$X=\dfrac{[2;2,1_2,2_2,\overline{1,2}]-[1;1,2_2,\overline{2,1}]}{[2;\overline{1,2}]-[1;\overline{2,1}]}$$
	and
	$$Y=\dfrac{([2;\overline{1,2}]+\beta_{2k+1})([1;\overline{2,1}]+\beta_{2k+1})}{([2;2,1_2,2_2,\overline{1,2}]+\beta_{2k-2})([1;1,2_2,\overline{2,1}]+\beta_{2k-2})}.$$
	Then, $A_k+B_k>C_k+D_k$.
 \end{proof}

\begin{lemma}\label{l.8} If $1\le m<k$, then $\lambda^-_0(2_21_22_{2k+1}1_22^*2_{2m}1_22_2)>m(\gamma^1_k).$
\end{lemma}

\begin{proof} Since $\lambda^-_0(2_21_22_{2k+1}1_22^*2_{2m}1_22_2) > \lambda^-_0(2_{2k}1_22^*2_{2m}1_22_2)$, the desired result follows from Lemma \ref{l.10}. 
\end{proof}

\begin{lemma}\label{l.15} If $k>m$, $\lambda^-_0(1_22_{2k+2}1_22^*2_{2m}1_22_2)>m(\gamma^1_k).$
\end{lemma}

\begin{proof} Since $\lambda^-_0(2_21_22_{2k+1}1_22^*2_{2m}1_22_2) > \lambda^-_0(2_{2k}1_22^*2_{2m}1_22_2)$, the desired result follows from Lemma \ref{l.8}.
\end{proof}




\begin{lemma}\label{l.9} If $k\geq 2$, then $\lambda_0^-(1_2\theta_k^012_2)>\lambda_0^-(2_2\theta_k^012_2)>m(\gamma_k^1)$, where $\theta_k^0$ is the string in \eqref{e.replication-0}. Also, 
$$\lambda_0^-(1_2\theta_1^012_2)>3.005587313>m(\gamma_1^1),$$ 
and 
$$\lambda_0^-(2_{2}\theta_1^012_21_2)>\lambda_0^-(2_{2}\theta_1^012_3)>3.0055873125>m(\gamma_1^1).$$
\end{lemma}
\begin{proof} The inequalities $\lambda_0^-(1_2\theta_k^012_2)>\lambda_0^-(2_2\theta_k^012_2)$, $\lambda_0^-(1_2\theta_1^012_2)>3.005587313>m(\gamma_1^1)$, and 
$$\lambda_0^-(2_{2}\theta_1^012_21_2)>\lambda_0^-(2_{2}\theta_1^012_3)>3.0055873125>m(\gamma_1^1)$$ are clear. Hence, it remains only to prove that $\lambda_0^-(2_2\theta_k^012_2)>m(\gamma_k^1)$ for all $k\geq 2$. 

For this sake, let us show that $A_k+B_k > C_k+D_k$, where 
\begin{eqnarray*}\lambda^{-}_0(2_2\theta_k^012_2)&=&[2;2_{2k},1_2,2_{2k+2},1_2,2_{2k},1_2,2_2,\overline{2,1}]+[0;1_2,2_{2k},1_2,2_{2k+2},1_2,2_{2k+2},\overline{2,1}] \\ 
&=:& A_k+B_k
\end{eqnarray*} 
and 
\begin{eqnarray*}
m(\gamma^1_k)&\leq& [2;2_{2k},1_2,2_{2k+2},1_2,2_{2k},1_2,2_{2k+1},1_2,2_2,\overline{2,1}]+[0;1_2,2_{2k},1_2,2_{2k+2},1_2,2_{2k+1},1_2,2_{2},\overline{2,1}] \\ &:=& C_k+D_k 
\end{eqnarray*}

Note that 
$$C_k-A_k=\dfrac{[2;2_{2k-3},1_2,2_2,\overline{2,1}]-[1;\overline{2,1}]}{q_{6k+11}^2([2;2_{2k-3},1_2,2_2,\overline{2,1}]+\beta_{6k+11})([1;\overline{2,1}]+\beta_{6k+11})}$$
and
$$B_k-D_k=\dfrac{[2;1_2,2_2,\overline{2,1}]-[2;2,\overline{2,1}]}{\tilde{q}_{6k+8}^2([2;2,\overline{2,1}]+\tilde{\beta}_{6k+8})([2;1_2,2_2,\overline{2,1}]+\tilde{\beta}_{6k+8})}.$$
Thus,
$$\dfrac{C_k-A_k}{B_k-D_k}=\dfrac{[2;2_{2k-3},1_2,2_2,\overline{2,1}]-[1;\overline{2,1}]}{[2;1_2,2_2,\overline{2,1}]-[2;2,\overline{2,1}]}\cdot X \cdot \dfrac{\tilde{q}^2_{6k+8}}{q^2_{6k+11}},$$
where
$$X=\dfrac{([2;2,\overline{2,1}]+\tilde{\beta}_{6k+8})([2;1_2,2_2,\overline{2,1}]+\tilde{\beta}_{6k+8})}{([2;2_{2k-3},1_2,2_2,\overline{2,1}]+\beta_{6k+11})([1;\overline{2,1}]+\beta_{6k+11})}.$$

We have
$$\dfrac{[2;2_{2k-3},1_2,2_2,\overline{2,1}]-[1;\overline{2,1}]}{[2;1_2,2_2,\overline{2,1}]-[1;1,\overline{2,1}]}\leq\dfrac{[2;\overline{2}]-[1;\overline{2,1}]}{[2;1_2,2_2,\overline{2,1}]-[2;2,\overline{2,1}]}<6.44.$$
Furthermore, by Euler's rule, $q_{6k+11}> q(2_{2k}1_22_{2k+2}1_22_{2k})q(1_22_3)=29q_{6k+6}$ and $\tilde{q}_{6k+8}=q(1_22_{2k}1_22_{2k+2}1_22_{2k})=p_{6k+6}+2q_{6k+6}$. Thus,
$$\dfrac{\tilde{q}_{6k+8}}{q_{6k+11}}<\dfrac{2}{29}\cdot\dfrac{p_{6k+6}}{q_{6k+6}}+\dfrac{1}{29}<\dfrac{3}{29}.$$
 
By Lemma \ref{Le1}, $X<4$ and therefore,  
$$\dfrac{C_k-A_k}{B_k-D_k}\leq 6.44 \cdot 4 \cdot \left(\dfrac{3}{29}\right)^2<1.$$
\end{proof}

\begin{corollary}\label{c.3} Consider the following parameters  
$$\lambda_k^{(1)}:=\left\{\begin{array}{cl}\min\{\lambda^-_0(1_22_{2k-2}1_22^*2_{2k}1_22_2), \lambda^-_0(2_21_22_{2k+1}1_22^*2_{2k-2}1_22_2), \lambda_0^-(2_2\theta_k^012_2)\}, & \textrm{ if } k\geq 2 \\ \min\{\lambda^-_0(1_22_{2k-2}1_22^*2_{2k}1_22_2), \lambda^-_0(2_21_22_{2k+1}1_22^*2_{2k-2}1_22_2), 3.0055873125\}, & \textrm{ if } k=1.\end{array}\right.$$
Then, $\lambda_k^{(1)}>m(\gamma_k^1)$ and any $(k,\lambda_k^{(1)})$-admissible word $\theta$ containing the string $\theta_k^0$ from \eqref{e.replication-0} extends as 
$$\theta=...2_{2k}1_22\theta_k^0 12_2...= ...2_{2k}1_22_{2k+1}1_22_{2k+2}1_22_{2k}1_22^*2_{2k}1_22_{2k+2}1_22_{2k}1_22_2...$$
\end{corollary}

\begin{proof} The fact that $\lambda_k^{(1)}>m(\gamma_k^1)$ follows from Lemmas \ref{l.7}, \ref{l.8} and \ref{l.9}.  

By Lemma \ref{l.6}, a $(k,\lambda_k^{(1)})$-admissible word $\theta$ containing $\theta_k^0$ extends as $...\theta_k^012_2...$. By (Remark \ref{r.2} and) Lemma \ref{l.9}, $\theta$ must keep extending as 
$$\theta = ...2_21_22\theta_k^012_2...$$  
Finally, by Lemma \ref{l.7} and \ref{l.8} (together with Remark \ref{r.2}), $\theta$ must keep extending as $\theta=...2_{2k}1_22\theta_k^0 12_2...$. 
\end{proof}
\subsection{Extension from $2_{2k}1_22\theta_k^012_2$ to $2_{2k}1_22\theta_k^012_{2k+1}1_22_2$}

\begin{lemma} \label{l.11} If $1\leq j\leq k$, then 
$\lambda_0^-(2_{2k}1_22\theta_k^012_{2j}1_2) > \lambda_0^-(2_{2k}1_22\theta_k^012_{2k+2})>m(\gamma_k^1)$.
\end{lemma}
\begin{proof} 
By definition,  $m(\gamma^1_k)\leq C_k+D_k$, where $C_k =[2;2_{2k},1_2,2_{2k+2},1_2,2_{2k},1_2,2_{2k+1},1_2,2_2,\overline{2,1}]$
and $D_k =[0;1_2,2_{2k},1_2,2_{2k+2},1_2,2_{2k+1},1_2,2_{2k},1_2,2_{2},\overline{2,1}].$

Note that  $\lambda_0^-(2_{2k}1_22\theta_k^012_{2j}1_2) > \lambda_0^-(2_{2k}1_22\theta_k^012_{2k+2})=A_k+B_k$, where $A_k=[2;2_{2k},1_2,2_{2k+2},1_2,2_{2k},1_2,2_{2k+2},\overline{2,1}]$ and $B_k=[0;1_2,2_{2k},1_2,2_{2k+2},1_2,2_{2k+1},1_2,2_{2k}\overline{1,2}]$. Hence, our work is reduced to prove
that $A_k-C_k>D_k-B_k$.

In order to prove this inequality, we observe that
$$A_k-C_k=\dfrac{[2;\overline{2,1}]-[1;1,2_2,\overline{2,1}]}{q_{8k+9}^2([2;\overline{2,1}]+\beta_{8k+9})([1;1,2_2,\overline{2,1}]+\beta_{8k+9})},$$
and
$$D_k-B_k=\dfrac{[1;1,2_2,\overline{2,1}]-[1;\overline{2,1}]}{\tilde{q}_{8k+11}^2([1;1,2_2,\overline{2,1}]+\tilde{\beta}_{8k+11})([1;\overline{2,1}]+\tilde{\beta}_{8k+11})},$$
where $q_{8k+9}=q(2_{2k}1_22_{2k+2}1_22_{2k}1_22_{2k+1})$ and $\tilde{q}_{8k+11}=q(1_22_{2k}1_22_{2k+2}1_22_{2k+1}1_22_{2k})$.
Thus,
$$\dfrac{A_k-C_k}{D_k-B_k}=\dfrac{[2;\overline{2,1}]-[1;1,2_2,\overline{2,1}]}{[1;1,2_2,\overline{2,1}]-[1;\overline{2,1}]}\cdot Y \cdot \dfrac{\tilde{q}^2_{8k+11}}{q^2_{8k+9}},$$
where
$$Y=\dfrac{([1;1,2_2,\overline{2,1}]+\tilde{\beta}_{8k+11})([1;\overline{2,1}]+\tilde{\beta}_{8k+11})}{([2;\overline{2,1}]+\beta_{8k+9})([1;1,2_2,\overline{2,1}]+\beta_{8k+9})}.$$
We have
$$Y\geq\dfrac{([1;1,2_2,\overline{2,1}]+[0;\overline{2}])([1;\overline{2,1}]+[0;\overline{2}])}{([2;\overline{2,1}]+[0;\overline{2}])([1;1,2_2,\overline{2,1}]+[0;\overline{2}])}>0.64.$$
Let $\alpha=2_{2k}1_22_{2k+2}1_22_{2k}$ and $\tilde{\alpha}=1_2\alpha$, by Euler's rule, $\tilde{q}_{8k+11}\geq q(\tilde{\alpha})q(21_22_{2k})= (p(\alpha)+2q({\alpha}))(p(1_22_{2k})+2q(1_22_{2k}))$ and $q_{8k+9}\leq 2q(\alpha)q(1_22_{2k+1})\leq 2q(\alpha)3q(1_22_{2k})$. Thus, 
\begin{align*}
\dfrac{\tilde{q}_{8k+11}}{q_{8k+9}}  \geq \left(1+\dfrac{1}{2}[0;\alpha]\right)\left(\dfrac{2}{3}+ \dfrac{1}{3}[0;1_22_{2k}]\right)\geq \left(1+\dfrac{1}{2}[0;\overline{2}]\right)\left(\dfrac{2}{3}+ \dfrac{1}{3}[0;1_22_{2}]\right)>1.03 
\end{align*}

Therefore, 
$$\dfrac{A_k-C_k}{D_k-B_k}> 1.925 \cdot 0.64 \cdot (1.03)^2>1.$$
\end{proof}

\begin{corollary}\label{c.4} Consider the parameter 
$$\lambda_k^{(2)}:=\min\{\lambda^-_0(2_{2k}1_22^*2_{2k-2}1_22_2), \lambda_0^-(2_{2k}1_22\theta_k^012_{2k+2})\}.$$ 
Then, $\lambda_k^{(2)}>m(\gamma_k^1)$ and any $(k,\lambda_k^{(2)})$-admissible word $\theta$ containing $2_{2k}1_22\theta_k^012_2$ extends as 
$$\theta = ...2_{2k}1_22\theta_k^012_{2k+1}1_22_2... =...2_{2k}1_22_{2k+1}1_22_{2k+2}1_22_{2k}1_22^*2_{2k}1_22_{2k+2}1_22_{2k}1_22_{2k+1}1_22_2...$$
\end{corollary}

\begin{proof} The fact that $\lambda_k^{(2)}>m(\gamma_k^1)$ follows from Lemmas \ref{l.10} and \ref{l.11}. Moreover, these lemmas (and Remark \ref{r.2}) imply that any $(k,\lambda_k^{(2)})$-admissible word $\theta$ containing $2_{2k}1_22\theta_k^012_2$ extends as $\theta = ...2_{2k}1_22\theta_k^012_{2k+1}1_22_2...$
\end{proof}

\subsection{Extension from $2_{2k}1_22\theta_k^012_{2k+1}1_22_2$ to $2_{2k}1_22\theta_k^012_{2k+1}1_22_{2k+1}$} 
Next, we write $p_{2k+2+j}=p(2_{2k}1_22_j)$ and $q_{2k+2+j}=q(2_{2k}1_22_j)$. Moreover, $\tilde{p}_{2k+4+j}=\tilde{p}(1_22_{2k}1_22_j)$ and $\tilde{q}_{2k+4+j}=\tilde{q}(1_22_{2k}1_22_j)$. We also write 
$$\beta_{2k+2+j}=\dfrac{q_{2k+2+j-1}}{q_{2k+4+j}} \quad \mbox{and} \quad \tilde{\beta}_{2k+4+j}=\dfrac{\tilde{q}_{2k+4+j-1}}{\tilde{q}_{2k+4+j}}.$$
As before, $\tilde{q}_{2k+2+(j+1)}=q_{2k+2+j}$ and then $\tilde{\beta}_{2k+2+(j+1)}=\beta_{2k+2+j}.$
\begin{lemma}\label{l.12} If $m<k$, then $\lambda_0^-(2_21_22_{2k+2}1_22_{2k}1_22^*2_{2k}1_22_{2m+2}1_2)>m(\gamma_k^1)$. 
\end{lemma}

\begin{proof} We write 
	$$\lambda_0^-(2_21_22_{2k+2}1_22_{2k}1_22^*2_{2k}1_22_{2m+2}1_2)=[2;2_{2k},1_2,2_{2m+2},1_2,\overline{2,1}]+[0;1_2,2_{2k},1_2,2_{2k+2},1_2,2_2,\overline{2,1}]$$
	we write just $\lambda_0^-(2_21_22_{2k+2}1_22_{2k}1_22^*2_{2k}1_22_{2m+2}1_2):=A_k+B_k$.
Remember that
	$$m(\gamma^1_k)<[2;2_{2k},1_2,2_{2k+2},1_2,2_2,\overline{1,2}]+[0;1_2,2_{2k},1_2,2_{2k+2},1_2,2_2,\overline{1,2}]:=C_k+D_k.$$
It is suffices take $m=k-1$. We have
	$$A_k-C_k=[0;2_{2k},1_2,2_{2k},1_2,\overline{2,1}]-[0;2_{2k},1_2,2_{2k+2},1_2,2_2,\overline{1,2}]$$
then
	$$A_k-C_k=\dfrac{[2;2,1_2,2_2,\overline{1,2}]-[1;1,{2,1}]}{q^2_{4k+2}([2;2,1_2,2_2,\overline{1,2}]+\beta_{4k+2})([1;1,{2,1}]+\beta_{4k+2})}.$$
Moreover,
	$$D_k-B_k=[0;1_2,2_{2k},1_2,2_{2k+2},1_2,2_2,\overline{1,2}]-[0;1_2,2_{2k},1_2,2_{2k+2},1_2,2_2,\overline{2,1}]$$
then
	$$D_k-B_k=\dfrac{[1;1,2_2,\overline{2,1}]-[1;1,2_2,\overline{1,2}]}{\tilde{q}^2_{4k+8}([1;1,2_2,\overline{2,1}]+\tilde{\beta}_{4k+8})([1;1,2_2,\overline{1,2}]+\tilde{\beta}_{4k+8})}.$$
	Using Lemma \ref{Le1} and since $\tilde{q}_{4k+8}=q_{4k+7}>32q_{4k+2}$ we have that $A_k-C_k>D_k-B_k$, that is, $A_k+B_k>C_k+D_k$.	
\end{proof}

\begin{corollary}\label{c.5} Consider the parameter 
$$\lambda_k^{(3)}:=\min\{\lambda^-_0(2_21_22_{2k+1}1_22^*2_{2k-2}1_22_2), \lambda_0^-(2_21_22_{2k+2}1_22_{2k}1_22^*2_{2k}1_22_{2k}1_2)\}.$$
Then, $\lambda_k^{(3)}>m(\gamma_k^1)$ and any $(k,\lambda_k^{(3)})$-admissible word $\theta$ containing $2_{2k}1_22\theta_k^012_{2k+1}1_22_2$ extends as 
\begin{eqnarray*}\theta &=& ...2_{2k}1_22\theta_k^012_{2k+1}1_22_{2k+1}... \\ 
&=&...2_{2k}1_22_{2k+1}1_22_{2k+2}1_22_{2k}1_22^*2_{2k}1_22_{2k+2}1_22_{2k}1_22_{2k+1}1_22_{2k+1}...
\end{eqnarray*}
\end{corollary}

\begin{proof} The fact that $\lambda_k^{(3)}>m(\gamma_k^1)$ follows from Lemmas \ref{l.8} and \ref{l.12}. Moreover, these lemmas (and Remark \ref{r.2}) imply that any $(k,\lambda_k^{(3)})$-admissible word $\theta$ containing $2_{2k}1_22\theta_k^012_2$ extends as $\theta = ...2_{2k}1_22\theta_k^012_{2k+1}1_22_{2k+1}...$
\end{proof}

\subsection{Extension from $2_{2k}1_22\theta_k^012_{2k+1}1_22_{2k+1}$ to $2_{2k+1}1_22_{2k}1_22\theta_k^012_{2k+1}1_22_{2k+1}$}

\begin{lemma}\label{l.13} If $m=k$ and $j<k$, then $\lambda^-_0(1_22_{2j+2}1_22_{2k}1_22^*2_{2k}1_22_{2m+2}1_22_2)>m(\gamma^1_k).$ 
\end{lemma}

\begin{proof} Let $u=1_22_{2j+2}1_22_{2k}1_22^*2_{2k}1_22_{2k+2}1_22_2$. We can suppose $j=k-1$.  Note that
	$$\lambda^-_0(u)=[2;2_{2k},1_2,2_{2k+2},1_2,2_2,\overline{2,1}]+[0;1_2,2_{2k},1_2,2_{2k},1_2,\overline{2,1}]=A_k+B_k$$
In the same way as before
	$$m(\gamma^1_k)<[2;2_{2k},1_2,2_{2k+2},1_2,2_2,\overline{1,2}]+[0;1_2,2_{2k},1_2,2_{2k+2},1_2,2_2,\overline{1,2}]:=C_k+D_k.$$
Remember that $\dfrac{p_{s+2k+2}}{q_{s+2k+2}}=[0;2_{2k},1_2,2_s]$ and $\dfrac{\tilde{p}_{s+2k+4}}{\tilde{q}_{s+2k+4}}=[0;1_2,2_{2k},1_2,2_s].$ 
Therefore, we have
	$$\dfrac{B_k-D_k}{C_k-A_k}=\dfrac{q^2_{4k+4}}{q^2_{4k+3}}\cdot \dfrac{[2;2,1_2,2_{2},\overline{1,2}]-[1;1,\overline{2,1}]}{[1;1,2_2,\overline{2,1}]-[1;1,2_{2},\overline{1,2}]}\cdot Q$$
where 
	$$Q=\dfrac{([2;2_{2k-2j},1_2,2_{2k},\overline{1,2}]+\beta'_{2k+2j+3})([1;\overline{1,2}]+\beta'_{2k+2j+3})}{([2;2,1_2,\overline{2,1}]+\beta_{4k+3})([1;1,\overline{2,1}]+\beta_{4k+3})}.$$ 
We observe that $\dfrac{q_{4k+4}}{q_{4k+3}}\geq 2$ because $j<k.$ It is not difficult to see that
	$$\dfrac{[2;2,1_2,2_{2},\overline{1,2}]-[1;1,\overline{2,1}]}{[1;1,2_2,\overline{2,1}]-[1;1,2_{2},\overline{1,2}]}>1$$	
and then
	$$\dfrac{B_k -D_k}{C_k-A_k}>4\cdot 1\cdot \dfrac{1}{4}=1.$$
\end{proof} 

\begin{lemma} \label{l.14}
If $u_4=2_{2k+1}1_22_{2k+1}1_22_{2k+2}1_22_{2k}1_22^*2_{2k}1_22_{2k+2}1_22_{2k}1_22_{2k+1}1_22_{2k+1}$, then $\lambda_0^-(u_4)>m(\gamma_k^1)$. 
\end{lemma}
\begin{proof}
By definition, $\lambda^{-}_0(u_4)=A_k+B_k$, where $A_k=[2;2_{2k},1_2,2_{2k+2},1_2,2_{2k},1_2,2_{2k+1},1_2,2_{2k+1},\overline{2,1}]$ and $B_k=[0;1_2,2_{2k},1_2,2_{2k+2},1_2,2_{2k+1},1_2,2_{2k+1},\overline{2,1}]$. 
Moreover, $m(\gamma^1_k)\leq C_k+D_k$, where
\begin{align*}
C_k &=[2;2_{2k},1_2,2_{2k+2},1_2,2_{2k},1_2,2_{2k+1},1_2,2_{2k+2},1_2,\overline{2}] \textrm{ and } \\
D_k &=[0;1_2,2_{2k},1_2,2_{2k+2},1_2,2_{2k+1},1_2,2_{2k},1_2,2_{2},\overline{2,1}].
\end{align*}
We shall show that $A_k+B_k> C_k+D_k$. In order to establish this inequality, we observe that 
$$C_k-A_k=\dfrac{[2;\overline{1,2}]-[1;\overline{2}]}{q_{10k+14}^2([1;\overline{2}]+\beta_{10k+14})([2;\overline{1,2}]+\beta_{10k+14})}$$
and
$$B_k-D_k=\dfrac{[2;\overline{2,1}]-[1;1,2_2,\overline{2,1}]}{\tilde{q}_{8k+11}^2([2;\overline{2,1}]+\tilde{\beta}_{8k+11})([1;1,2_2,\overline{2,1}]+\tilde{\beta}_{8k+11})},$$
where $q_{10k+14}=q(2_{2k}1_22_{2k+2}1_22_{2k}1_22_{2k+1}1_22_{2k+2}1)$ and $\tilde{q}_{8k+11}=q(1_22_{2k}1_22_{2k+2}1_22_{2k+1}1_22_{2k})$.
Thus,
$$\dfrac{C_k-A_k}{B_k-D_k}=\dfrac{[2;\overline{1,2}]-[1;\overline{2}]}{[2;\overline{2,1}]-[1;1,2_2,\overline{2,1}]}\cdot X_4 \cdot \dfrac{\tilde{q}^2_{8k+11}}{q^2_{10k+14}},$$
where
$$X_4=\dfrac{([2;\overline{2,1}]+\tilde{\beta}_{8k+11})([1;1,2_2,\overline{2,1}]+\tilde{\beta}_{8k+11})}{([1;\overline{2}]+\beta_{10k+14})([2;\overline{1,2}]+\beta_{10k+14})}.$$
Let $\alpha=2_{2k}1_22_{2k+2}1_22_{2k}$ and $\tilde{\alpha}=1_2\alpha$, by Euler's rule,
$$\tilde{q}_{8k+11}< 2q(\tilde{\alpha})q(21_22_{2k})\leq 2q(\tilde{\alpha})3 q(1_22_{2k}),$$ and $$q_{10k+14}> q(\alpha)q(1_22_{2k+1})q(1_22_{2k+2}1)\geq q(\alpha)2q(2_{2k+1})2q(12_{2k+2})\geq 4q(\alpha)q(1_22_{2k})q(12_{4}).$$ 
Thus, 
\begin{align*}
\dfrac{\tilde{q}_{8k+11}}{q_{10k+14}}  < \dfrac{3}{2q(12_4)}\left(\dfrac{p(\alpha)+2q(\alpha)}{q(\alpha)}\right)=\dfrac{3}{2\cdot 41}\left([0;\alpha]+2\right)< \dfrac{9}{2 \cdot 41}.
\end{align*}
Therefore, since that $X_4 \leq 4$, by Lemma \ref{Le1}, we obtain
$$\dfrac{C_k-A_k}{B_k-D_k}< 2.003 \cdot 4 \cdot \left(\dfrac{9}{2 \cdot 41}\right)^2<1.$$
\end{proof}

\begin{corollary}\label{c.6} Consider the parameter 
$$\lambda_k^{(4)}:=\min\{\lambda^-_0(2_{2k}1_22^*2_{2k-2}1_22_2), \lambda^-_0(1_22_{2k}1_22_{2k}1_22^*2_{2k}1_22_{2k+2}1_22_2),\lambda_0^-(2_{2k+1}1_22\theta_k^012_{2k+1}1_22_{2k+1})\}.$$
Then, $\lambda_k^{(4)}>m(\gamma_k^1)$ and any $(k,\lambda_k^{(4)})$-admissible word $\theta$ containing $2_{2k}1_22\theta_k^012_{2k+1}1_22_{2k+1}$ extends as 
\begin{eqnarray*}\theta &=& ...2_{2k+1}1_22_{2k}1_22\theta_k^012_{2k+1}1_22_{2k+1}... \\ 
&=&...2_{2k+1}1_22_{2k}1_22_{2k+1}1_22_{2k+2}1_22_{2k}1_22^*2_{2k}1_22_{2k+2}1_22_{2k}1_22_{2k+1}1_22_{2k+1}...
\end{eqnarray*}
\end{corollary}

\begin{proof} The fact that $\lambda_k^{(4)}>m(\gamma_k^1)$ follows from Lemmas \ref{l.10}, \ref{l.13} and \ref{l.14}. Moreover, these lemmas (and Remark \ref{r.2}) imply that any $(k,\lambda_k^{(4)})$-admissible word $\theta$ containing $2_{2k}1_22\theta_k^012_{2k+1}1_22_{2k+1}$ extends as $\theta = ...2_{2k+1}1_22_{2k}1_22\theta_k^012_{2k+1}1_22_{2k+1}...$
\end{proof}

\subsection{Extension from $2_{2k+1}1_22_{2k}1_22\theta_k^012_{2k+1}1_22_{2k+1}$ to $2_{2k+1}1_22_{2k}1_22\theta_k^012_{2k+1}1_22_{2k+2}1_22_{2k}$}

\begin{lemma}\label{l.17} One has $\lambda_0^-(u_5) > \lambda_0^-(u_6)>m(\gamma_k^1)$, where 
$$u_5=2_{2k+1}1_22_{2k}1_22_{2k+1}1_22_{2k+2}1_22_{2k}1_22^*2_{2k}1_22_{2k+2}1_22_{2k}1_22_{2k+1}1_22_{2k+1}1_22_2$$
and 
$$u_6=2_{2k+1}1_22_{2k}1_22_{2k+1}1_22_{2k+2}1_22_{2k}1_22^*2_{2k}1_22_{2k+2}1_22_{2k}1_22_{2k+1}1_22_{2k+3}$$
 \end{lemma}
\begin{proof}
Let $\lambda^{-}_0(u_6)=A_k+B_k$, where 
$$A_k=[2;2_{2k},1_2,2_{2k+2},1_2,2_{2k},1_2,2_{2k+1},1_2,2_{2k+3},\overline{2,1}] \textrm{ and}$$
$$B_k=[0;1_2,2_{2k},1_2,2_{2k+2},1_2,2_{2k+1},1_2,2_{2k},1_2,2_{2k+1}\overline{2,1}].$$

Moreover, by definition, $m(\gamma^1_k)\leq C_k+D_k$, where 
\begin{align*}
C_k &=[2;2_{2k},1_2,2_{2k+2},1_2,2_{2k},1_2,2_{2k+1},1_2,2_{2k+2},1_2,\overline{2}] \textrm{ and } \\
D_k &=[0;1_2,2_{2k},1_2,2_{2k+2},1_2,2_{2k+1},1_2,2_{2k},1_2,2_{2k+2}1_22_2,\overline{2,1}].
\end{align*}

Let us show that $A_k+B_k> C_k+D_k$. For this sake, we observe that 
$$A_k-C_k=\dfrac{[2;\overline{2,1}]-[1;1,\overline{2}]}{{q}_{10k+13}^2([2;\overline{2,1}]+\beta_{10k+13})([1;1,\overline{2}]+\beta_{10k+13})}$$
and
$$D_k-B_k=\dfrac{[2;\overline{1,2}]-[1;2_2,\overline{2,1}]}{\tilde{q}_{10k+16}^2([1;2_2,\overline{2,1}]+\tilde{\beta}_{10k+16})([2;\overline{1,2}]+\tilde{\beta}_{10k+16})},$$
where $\tilde{q}_{10k+16}=q(1_22_{2k}1_22_{2k+2}1_22_{2k+1}1_22_{2k}1_22_{2k+2}1)$ and ${q}_{10k+13}= \\
q(2_{2k}1_22_{2k+2}1_22_{2k}1_22_{2k+1}1_22_{2k+2})$.
Thus,
$$\dfrac{A_k-C_k}{D_k-B_k}=\dfrac{[2;\overline{2,1}]-[1;1,\overline{2}]}{[2;\overline{1,2}]-[1;2_2,\overline{2,1}]}\cdot X_6 \cdot \dfrac{\tilde{q}^2_{10k+16}}{{q}^2_{10k+13}},$$
where
$$X_6=\dfrac{([1;2_2,\overline{2,1}]+\tilde{\beta}_{10k+16})([2;\overline{1,2}]+\tilde{\beta}_{10k+16})}{([2;\overline{2,1}]+\beta_{10k+13})([1;1,\overline{2}]+\beta_{10k+13})}.$$
Note that
$$X_6\geq\dfrac{([1;2_2,\overline{2,1}]+[0,1,2_4,1])([2;\overline{1,2}]+[0,1,2_4,1])}{([2;\overline{2,1}]+[0,2_4,1])([1;1,\overline{2}]+[0,2_4,1])}>1.23.$$
Let $\theta=2_{2k+2}1_22_{2k}1_22_{2k+1}1_22_{2k+2}$, since $q(2_{2k+1}1_22_{2k}1_22_{2k+1}1_22_{2k+2})<(1/2)q(\theta)$ and $q(2_{2k}1)<q(2_{2k}1_2)$, by Euler's rule, we have:
\begin{align*}
q_{10k+13}= q(2_{2k}1_2)q(\theta)+q(2_{2k}1)q(2_{2k+1}1_22_{2k}1_22_{2k+1}1_22_{2k+2})<(3/2)q(2_{2k}1_2)q(\theta).
\end{align*}
Analogously, since $q(1_22_{2k}1)>(1/2)q(1_22_{2k}1_2)$ and  $q(2_{2k+1}1_22_{2k+1}1_22_{2k}1_22_{2k+2}1)>(1/3)q(\theta^t1)$, we obtain:
\begin{align*}
\tilde{q}_{10k+16}&=q(1_22_{2k}1_2\theta^t1)=q(1_22_{2k}1_2)q(\theta^t1)+q(1_22_{2k}1)q(2_{2k+1}1_22_{2k+1}1_22_{2k}1_22_{2k+2}1)\\
&>(7/6)q(1_22_{2k}1_2)q(\theta^t1)>2(7/6)q(1_22_{2k})q(\theta^t).
\end{align*} 
 Thus, 
\begin{align*}
\dfrac{\tilde{q}_{10k+16}}{q_{10k+13}}> \dfrac{7}{3} \cdot \dfrac{2}{3}=\dfrac{14}{9}.
\end{align*}
Therefore,
$$\dfrac{A_k-C_k}{D_k-B_k}> 0.49 \cdot 1.23 \cdot \left(\dfrac{14}{9}\right)^2>1.$$
\end{proof}

\begin{corollary}\label{c.7} Consider the parameter 
$$\lambda_k^{(5)}:=\min\{\lambda_0^-(1_22_{2k+2}1_22^*2_{2k-2}1_22_2),\lambda_0^-(1_22_{2k-2}1_22^*2_{2k+1}1_22_2), \lambda_0^-(u_6) \}.$$
Then, $\lambda_k^{(5)}>m(\gamma_k^1)$ and any $(k,\lambda_k^{(5)})$-admissible word $\theta$ containing $2_{2k+1}1_22_{2k}1_22\theta_k^012_{2k+1}1_22_{2k+1}$ extends as 
\begin{eqnarray*}\theta &=& ...2_{2k+1}1_22_{2k}1_22\theta_k^012_{2k+1}1_22_{2k+2}1_22_{2k}... \\ 
&=&...2_{2k+1}1_22_{2k}1_22_{2k+1}1_22_{2k+2}1_22_{2k}1_22^*2_{2k}1_22_{2k+2}1_22_{2k}1_22_{2k+1}1_22_{2k+2}1_22_{2k}...
\end{eqnarray*}
\end{corollary}

\begin{proof} The fact that $\lambda_k^{(5)}>m(\gamma_k^1)$ follows from Lemmas \ref{l.15}, \ref{l.16} and \ref{l.17}. Moreover, these lemmas (and Remark \ref{r.2}) imply that any $(k,\lambda_k^{(5)})$-admissible word $\theta$ containing $2_{2k+1}1_22_{2k}1_22\theta_k^012_{2k+1}1_22_{2k+1}$ extends as $\theta = ...2_{2k+1}1_22_{2k}1_22\theta_k^012_{2k+1}1_22_{2k+2}1_22_{2k}...$
\end{proof}

\subsection{Replication lemma}

\begin{lemma}\label{l.18} One has $\lambda_0^-(u_7)>\lambda_0^-(u_8)>m(\gamma_k^1)$, where $$u_7=2_21_22_{2k+1}1_22_{2k}1_22_{2k+1}1_22_{2k+2}1_22_{2k}1_22^*2_{2k}1_22_{2k+2}1_22_{2k}1_22_{2k+1}1_22_{2k+2}1_22_{2k}$$ and $$u_8=2_{2k+3}1_22_{2k}1_22_{2k+1}1_22_{2k+2}1_22_{2k}1_22^*2_{2k}1_22_{2k+2}1_22_{2k}1_22_{2k+1}1_22_{2k+2}1_22_{2k}.$$  \end{lemma}
\begin{proof}
Let $\lambda^{-}_0(u_8)=A_k+B_k$, where  
\begin{align*}
A_k&=[2;2_{2k},1_2,2_{2k+2},1_2,2_{2k},1_2,2_{2k+1},1_2,2_{2k+2},1_2,2_{2k},\overline{1,2}] \textrm{ and}\\
B_k&=[0;1_2,2_{2k},1_2,2_{2k+2},1_2,2_{2k+1},1_2,2_{2k},1_2,2_{2k+3},\overline{2,1}].
\end{align*}
Furthermore, by definition,  $m(\gamma^1_k)\leq C_k+D_k$, where 
\begin{align*}
C_k &=[2;2_{2k},1_2,2_{2k+2},1_2,2_{2k},1_2,2_{2k+1},1_2,2_{2k+2},1_2,\overline{2}] \textrm{ and } \\
D_k &=[0;1_2,2_{2k},1_2,2_{2k+2},1_2,2_{2k+1},1_2,2_{2k},1_2,2_{2k+2}1_22_2,\overline{2,1}].
\end{align*} 
Thus, our task is prove that $B_k-D_k>C_k-A_k$. In order to establish this estimative, we observe that
$$C_k-A_k=\dfrac{[2;\overline{2}]-[1;\overline{2,1}]}{q_{12k+15}^2([2;\overline{2}]+\beta_{12k+15})([1;\overline{2,1}]+\beta_{12k+15})}$$
and
$$B_k-D_k=\dfrac{[2;\overline{2,1}]-[1;1,2_2,\overline{2,1}]}{\tilde{q}_{10k+15}^2([2;\overline{2,1}]+\tilde{\beta}_{10k+15})([1;1,2_2,\overline{2,1}]+\tilde{\beta}_{10k+15})},$$
where $q_{12k+15}=q(2_{2k}1_22_{2k+2}1_22_{2k}1_22_{2k+1}1_22_{2k+2}1_22_{2k})$ and $\tilde{q}_{10k+15}=$\\$=q(1_22_{2k}1_22_{2k+2}1_22_{2k+1}1_22_{2k}1_22_{2k+2})$.
Thus,
$$\dfrac{C_k-A_k}{B_k-D_k}=\dfrac{[2;\overline{2,1}]-[1;1,\overline{2}]}{[2;\overline{2,1}]-[1;1,2_2,\overline{2,1}]}\cdot X_8 \cdot \dfrac{\tilde{q}^2_{10k+15}}{{q}^2_{12k+15}},$$
where
$$X_8=\dfrac{([2;\overline{2,1}]+\tilde{\beta}_{10k+15})([1;1,2_2,\overline{2,1}]+\tilde{\beta}_{10k+15})}{([2;\overline{2}]+\beta_{12k+15})([1;\overline{2,1}]+\beta_{12k+15})}.$$
Note that
$$X_8<\dfrac{([2;\overline{2,1}]+[0,2_4,1])([1;1,2,2,\overline{2,1}]+[0,2_4,1])}{([2;\overline{2}]+[0,\overline{2}])([1;\overline{2,1}]+[0,\overline{2}])}<1.72.$$
Let $\omega=2_{2k}1_22_{2k+1}1_22_{2k+2}1_22_{2k}$, by Euler's rule, we have:
\begin{align*}
q_{12k+15}>q(2_{2k}1_2)q(2_{2k+2}1_2)q(\omega)\geq q(2_21_2)q(2_{2k+2}1_2)q(\omega)=12q(2_{2k+2}1_2)q(\omega)
\end{align*} 
and  
\begin{align*}
\tilde{q}_{10k+15} < 2q(1_2\omega^t)q(1_22_{2k+2})< 2 \cdot 3q(\omega)q(1_22_{2k+2}).
\end{align*} 
Thus, 
\begin{align*}
\dfrac{\tilde{q}_{10k+15}}{q_{12k+15}}  < \dfrac{1}{2}.
\end{align*}
Therefore,
$$\dfrac{C_k-A_k}{B_k-D_k} < 1.6 \cdot 1.72 \cdot \left( \dfrac{1}{2}\right)^2<1.$$
\end{proof}

\begin{corollary}\label{c.8} Consider the parameter 
$$\lambda_k^{(6)}:=\min\{\lambda_0^-(1_22_{2k+2}1_22^*2_{2k-2}1_22_2),\lambda_0^-(1_22_{2k-2}1_22^*2_{2k+1}1_22_2), \lambda_0^-(u_8) \}.$$
Then, $\lambda_k^{(6)}>m(\gamma_k^1)$ and any $(k,\lambda_k^{(6)})$-admissible word $\theta$ containing $2_{2k+1}1_22_{2k}1_22\theta_k^012_{2k+1}1_22_{2k+2}1_22_{2k}$ extends as 
\begin{eqnarray*}\theta &=& ...2_{2k}1_22_{2k+2}1_22_{2k}1_22\theta_k^012_{2k+1}1_22_{2k+2}1_22_{2k}... \\ 
&=&...2_{2k}1_22_{2k+2}1_22_{2k}1_22_{2k+1}1_22_{2k+2}1_22_{2k}1_22^*2_{2k}1_22_{2k+2}1_22_{2k}1_22_{2k+1}1_22_{2k+2}1_22_{2k}...
\end{eqnarray*}
\end{corollary}

\begin{proof}
The fact that $\lambda_k^{(6)}>m(\gamma_k^1)$ follows from Lemmas \ref{l.15}, \ref{l.16} and \ref{l.18}. Moreover, these lemmas (and Remark \ref{r.2}) imply that any $(k,\lambda_k^{(6)})$-admissible word $\theta$ containing $2_{2k+1}1_22_{2k}1_22\theta_k^012_{2k+1}1_22_{2k+2}1_22_{2k}$ extends as $\theta = ...2_{2k}1_22_{2k+2}1_22_{2k}1_22\theta_k^012_{2k+1}1_22_{2k+2}1_22_{2k}...$. 
\end{proof} 

The entire discussion of this section can be summarized into the following key lemma establishing the self-replication property of $\gamma_k^1$ for all $k\in\mathbb{N}$: 
\begin{lemma}[Replication Lemma]\label{l.replication} For each $k\in\mathbb{N}$, there exists an explicit constant $\lambda_k>m(\gamma_k^1)$ such that any $(k,\lambda_k)$-admissible word $\theta$ containing $\theta_k^0=2_{2k}1_22_{2k+2}1_22_{2k}1_22^*2_{2k}1_22_{2k+2}1_22_{2k}1$ must extend as 
$$\theta=...2_{2k}1_22_{2k+2}1_22_{2k}1_22_{2k+1}1_22_{2k+2}1_22_{2k}1_22^*2_{2k}1_22_{2k+2}1_22_{2k}1_22_{2k+1}1_22_{2k+2}1_22_{2k}...,$$
and the neighbourhood of the position $-(6k+9)$ is $...2_{2k}1_22_{2k+2}1_22_{2k}1_22^*2_{2k}1_22_{2k+2}1_22_{2k}1...$. In particular, any $(k,\lambda_k)$-admissible word $\theta$ containing $\theta_k^0$ has the form 
$$\overline{2_{2k}1_22_{2k+1}1_22_{2k+2}1_2}2_{2k}1_22^*2_{2k}1_22_{2k+2}1_22_{2k}1_22_{2k+1}1_22_{2k+2}1_22_{2k}...$$
\end{lemma}

\begin{proof} This result for $\lambda_k:=\min\{\lambda_k^{(i)}:i=1,\dots, 6\}$ is an immediate consequence of Corollaries \ref{c.1}, \ref{c.2}, \ref{c.3}, \ref{c.4}, \ref{c.5}, \ref{c.6}, \ref{c.7} and \ref{c.8}.   
\end{proof}

\section{Local uniqueness for $\gamma_1^1$}\label{s.k1}

Note that 
\begin{eqnarray*}
& & m(\theta(\underline{\omega}_1)) = \lambda_0(\overline{2_3 1_2 2_4 1_2 2_2 1_2}2^*22 1_2 2_4 1_2 2_2 1_2 \overline{2_3 1_2 2_4 1_2 2_2 1_2}) \\
& & = 3.00558731248699779818\dots
\end{eqnarray*}
and 
\begin{eqnarray*}
& & m(\gamma_1^1) = \lambda_0(\overline{2_3 1_2 2_4 1_2 2_2 1_2}2^*22 1_2 2_4 1_2 2_2 1_2 2_3 1_2 2_4 1_2 \overline{2}) \\
& & = 3.00558731248699779947\dots
\end{eqnarray*}

By Corollary \ref{c.2}, up to transposition, a $(1,3.009)$-admissible word $\theta$ is 
\begin{itemize}
\item $\theta=\dots 1_42^*21_2\dots$ or 
\item $\theta=\dots 2_21_22^*2_2\dots$
\end{itemize}
 
\begin{lemma}\label{l.4-1} $\lambda_0^+(1_42^*21_3 \dots) < 3.0032$.
\end{lemma}

\begin{lemma}\label{l.4-2} $\lambda_0^+(21_22^*2_3)<3.0017$ and $\lambda_0^+(2_31_22^*2_21_22_2) < 3.00486$. 
\end{lemma}

By Lemmas \ref{l.4-1}, \ref{l.4-2} and Remark \ref{r.2}, it follows that a $(1,3.009)$-admissible word $\theta$ is 
\begin{itemize}
\item $\theta=\dots 1_42^*21_22_2\dots$ or 
\item $\theta=\dots 1_22_21_22^*2_21_22_2\dots$
\end{itemize}

By applying Remark \ref{r.2} once again, we have that 
\begin{itemize} 
\item $\theta=\dots 1_4 2^* 2 1_2 2_21_2\dots$, or
\item $\theta=\dots 1_4 2^* 2 1_2 2_3\dots$, or   
\item $\theta=\dots 1_22_21_22^*2_21_22_2 1_2\dots$, or 
\item $\theta=\dots 1_22_21_22^*2_21_22_3\dots$. 
\end{itemize}
whenever $\theta$ is $(1,3.009)$-admissible. 

\begin{lemma}\label{l.4-3} 
\begin{itemize}
\item[(i)] $\lambda_0^+(2 1_4 2^* 2 1_2 2_2 1_2) < \lambda_0^+(2 1_4 2^*2 1_2 2_3 \dots) < 3.00026$
\item[(ii)] $\lambda_0^-(1_3 2_2 1_2 2^*2_2 1_2 2_2 1_2) > \lambda_0^-(2_2 1_2 2_2 1_2 2^*2_2 1_2 2_2 1_2) > 3.0056$
\item[(iii)] $\lambda_0^-(1_3 2_2 1_2 2^*2_2 1_2 2_3) > 3.0056$ 
\end{itemize}
\end{lemma}

By Lemma \ref{l.4-3}, if $\theta$ is $(1,3.0056)$-admissible, then 
\begin{itemize}
\item $\theta=\dots 1_5 2^* 2 1_2 2_2 1_2\dots$, or
\item $\theta=\dots 1_5 2^* 2 1_2 2_3\dots$, or  
\item $\theta=\dots 2_2 1_2 2_2 1_2 2^* 2_2 1_2 2_3\dots$
\end{itemize}

By Remark \ref{r.2}, it follows that 
\begin{itemize}
\item $\theta=\dots 1_5 2^* 2 1_2 2_2 1_4\dots$, or
\item $\theta=\dots 1_5 2^* 2 1_2 2_2 1_2 2_2\dots$, or
\item $\theta=\dots 1_5 2^* 2 1_2 2_3 1_2 2_2\dots$, or  
\item $\theta=\dots 1_5 2^* 2 1_2 2_4\dots$, or
\item $\theta=\dots 2_2 1_2 2_2 1_2 2^* 2_2 1_2 2_3 1_2 2_2\dots$
\item $\theta=\dots 2_2 1_2 2_2 1_2 2^* 2_2 1_2 2_4\dots$
\end{itemize}
whenever $\theta$ is $(1,3.0056)$-admissible.

\begin{lemma}\label{l.4-4}
\begin{itemize}
\item[(i)] $\lambda_0^+(1_6 2^* 2 1_2 2_2 1_4) < \lambda_0^+(1_6 2^* 2 1_2 2_2 1_2 2_2) < \lambda_0^+(1_6 2^* 2 1_2 2_4) < \lambda_0^+(1_6 2^* 2 1_2 2_3 1_2 2_2) < 3.00513$
\item[(ii)] $\lambda_0^-( 2 1_5 2^* 2 1_2 2_3 1_2 2_2) > \lambda_0^-(21_5 2^* 2 1_2 2_4) > \lambda_0^-(21_5 2^* 2 1_2 2_2 1_2 2_2) > \lambda_0^-(21_5 2^* 2 1_2 2_2 1_4) > 3.0063$ 
\item[(iii)] $\lambda_0^+(2_3 1_2 2_2 1_2 2^* 2_2 1_2 2_3 1_2 2_2) < 3.005584$ 
\item[(iv)] $\lambda_0^-(1 2_2 1_2 2_2 1_2 2^* 2_2 1_2 2_4) > 3.005589$ 
\end{itemize}
\end{lemma}

By Lemma \ref{l.4-4}, if $\theta$ is $(1,3.005589)$-admissible, then 
\begin{itemize}
\item $\theta=\dots 1_2 2_2 1_2 2_2 1_2 2^* 2_2 1_2 2_3 1_2 2_2\dots$, or 
\item $\theta=\dots 2_3 1_2 2_2 1_2 2^* 2_2 1_2 2_4\dots$.
\end{itemize}

\begin{lemma}\label{l.4-5} $\lambda_0^+(2_3 1_2 2_2 1_2 2^* 2_2 1_2 2_5) < 3.0055868$
\end{lemma}

By Lemma \ref{l.4-5} and Remark \ref{r.2}, if $\theta$ is $(1,3.005589)$-admissible, then 
\begin{itemize}
\item $\theta=\dots 1_2 2_2 1_2 2_2 1_2 2^* 2_2 1_2 2_3 1_2 2_2 1_2\dots$, or
\item $\theta=\dots 1_2 2_2 1_2 2_2 1_2 2^* 2_2 1_2 2_3 1_2 2_3\dots$, or 
\item $\theta=\dots 2_3 1_2 2_2 1_2 2^* 2_2 1_2 2_4 1_2 2_2\dots$  
\end{itemize}

\begin{lemma}\label{l.4-6}
\begin{itemize}
\item[(i)] $\lambda_0^+(2 1_2 2_2 1_2 2_2 1_2 2^* 2_2 1_2 2_3 1_2 2_2 1_2) < \lambda_0^+(2 1_2 2_2 1_2 2_2 1_2 2^* 2_2 1_2 2_3 1_2 2_3) < \lambda_0^+(1_4 2_2 1_2 2_2 1_2 2^* 2_2 1_2 2_3 1_2 2_3) < 3.00558725$ 
\item[(ii)] $\lambda_0^+(1 2_3 1_2 2_2 1_2 2^* 2_2 1_2 2_4 1_2 2_2) < 3.0055867$
\end{itemize}
\end{lemma}

By Lemma \ref{l.4-6} and Remark \ref{r.2}, if $\theta$ is $(1,3.005589)$-admissible, then 
\begin{itemize}
\item $\theta=\dots 1_4 2_2 1_2 2_2 1_2 2^* 2_2 1_2 2_3 1_2 2_2 1_2\dots$, or   
\item $\theta=\dots 2_4 1_2 2_2 1_2 2^* 2_2 1_2 2_4 1_2 2_2\dots$ 
\end{itemize}

By applying Remark \ref{r.2} once more, we get that 
\begin{itemize}
\item $\theta=\dots 1_4 2_2 1_2 2_2 1_2 2^* 2_2 1_2 2_3 1_2 2_2 1_4\dots$, or   
\item $\theta=\dots 1_4 2_2 1_2 2_2 1_2 2^* 2_2 1_2 2_3 1_2 2_2 1_2 2_2\dots$, or
\item $\theta=\dots 2_4 1_2 2_2 1_2 2^* 2_2 1_2 2_4 1_2 2_2 1_2\dots$, or  
\item $\theta=\dots 2_4 1_2 2_2 1_2 2^* 2_2 1_2 2_4 1_2 2_3\dots$ 
\end{itemize} 
whenever $\theta$ is $(1,3.005589)$-admissible. 

\begin{lemma}\label{l.4-7} One has 
\begin{itemize}
\item[(i)] $\lambda_0^+(1_4 2_2 1_2 2_2 1_2 2^* 2_2 1_2 2_3 1_2 2_2 1_4) < \lambda_0^+(1_4 2_2 1_2 2_2 1_2 2^* 2_2 1_2 2_3 1_2 2_2 1_2 2_2) < 3.0055872244$
\item[(ii)] $\lambda_0^+(2_5 1_2 2_2 1_2 2^* 2_2 1_2 2_4 1_2 2_3) < \lambda_0^+(2_2 1_2 2_4 1_2 2_2 1_2 2^* 2_2 1_2 2_4 1_2 2_3) < 3.0055873108$ 
\item[(iii)] $\lambda_0^+(2_5 1_2 2_2 1_2 2^* 2_2 1_2 2_4 1_2 2_2 1_2) < 3.005587211$
\end{itemize}
\end{lemma}

By Lemma \ref{l.4-7}(i), if $\theta$ is $(1,3.005589)$-admissible, then 
\begin{itemize} 
\item $\theta=\dots 2_4 1_2 2_2 1_2 2^* 2_2 1_2 2_4 1_2 2_2 1_2\dots$, or  
\item $\theta=\dots 2_4 1_2 2_2 1_2 2^* 2_2 1_2 2_4 1_2 2_3\dots$ 
\end{itemize}

By Remark \ref{r.2}, it follows that 
\begin{itemize} 
\item $\theta=\dots 2_2 1_2 2_4 1_2 2_2 1_2 2^* 2_2 1_2 2_4 1_2 2_2 1_2\dots$, or  
\item $\theta=\dots 2_5 1_2 2_2 1_2 2^* 2_2 1_2 2_4 1_2 2_2 1_2\dots$, or
\item $\theta=\dots 2_2 1_2 2_4 1_2 2_2 1_2 2^* 2_2 1_2 2_4 1_2 2_3\dots$, or 
\item $\theta=\dots 2_5 1_2 2_2 1_2 2^* 2_2 1_2 2_4 1_2 2_3\dots$
\end{itemize}
whenever $\theta$ is $(1,3.005589)$-admissible. 

Hence, Lemma \ref{l.4-7} implies the desired local uniqueness result for $\gamma_1^1$: 
\begin{lemma}[Local uniqueness of $\gamma_1^1$]\label{l.local-m1} A $(1,3.005589)$-admissible word $\theta$ has the form  
$$\theta=\dots 2_2 1_2 2_4 1_2 2_2 1_2 2^* 2_2 1_2 2_4 1_2 2_2 1_2\dots$$
In particular, it contains the string $\theta_1^0=2_{2}1_22_{4}1_22_{2}1_22^*2_{2}1_22_{4}1_22_{2}1$.
\end{lemma}

\section{Local uniqueness for $\gamma_2^1$}\label{s.k2}

Observe that 
\begin{eqnarray*}
& & m(\theta(\underline{\omega}_2)) = \lambda_0(\overline{2_5 1_2 2_6 1_2 2_4 1_2}2^*2_4 1_2 2_6 1_2 2_4 1_2 \overline{2_5 1_2 2_6 1_2 2_4 1_2}) \\
& & = 3.00016423121818941392559426822\dots
\end{eqnarray*} 
and 
\begin{eqnarray*}
& & m(\gamma_2^1) = \lambda_0(\overline{2_5 1_2 2_6 1_2 2_4 1_2}2^*2_4 1_2 2_6 1_2 2_4 1_2 2_5 1_2 2_6 1_2 \overline{2}) \\
& & = 3.00016423121818941392559426906\dots
\end{eqnarray*}

By Corollary \ref{c.2}, up to transposition, a $(2,3.009)$-admissible word $x$ is 
\begin{itemize}
\item $x=\dots 1_42^*21_2\dots$ or 
\item $x=\dots 2_21_22^*2_2\dots$
\end{itemize}

\begin{lemma}\label{l.5-1}  $\lambda_0^-(2_2 1_2 2^*2_2 1) > 3.0043$. 
\end{lemma}

By Lemma \ref{l.5-1}, if $x$ is $(2,3.009)$-admissible, then 
\begin{itemize}
\item $x=\dots 1_4 2^* 2 1_3\dots$ or 
\item $x=\dots 1_4 2^* 2 1_2 2_2\dots$ or
\item $x=\dots 1 2_2 1_2 2^* 2_3\dots$ or 
\item $x=\dots 2_3 1_2 2^* 2_3\dots$ 
\end{itemize}

\begin{lemma}\label{l.5-2} 
\begin{itemize}
\item[(i)] $\lambda_0^-(1 2_2 1_2 2^*2_4) > 3.00073$. 
\item[(ii)] $\lambda_0^+(2_3 1_2 2^*2_3 1) < 3$. 
\end{itemize}
\end{lemma}

By Lemma \ref{l.5-2}, Lemma \ref{l.4} and Lemma \ref{l.5-1}, if $x$ is $(2, 3.00073)$-admissible, then 
\begin{itemize}
\item $x=\dots 1_4 2^* 2 1_4\dots$ or 
\item $x=\dots 1_4 2^* 2 1_2 2_2 1_2\dots$ or
\item $x=\dots 1_4 2^* 2 1_2 2_4\dots$ or
\item $x=\dots 1_2 2_2 1_2 2^* 2_3 1_2 2_2\dots$ or   
\item $x=\dots 2_4 1_2 2^* 2_4\dots$ 
\end{itemize} 

\begin{lemma}\label{l.5-3} 
\begin{itemize}
\item[(i)] $\lambda_0^+(2 1_2 2_2 1_2 2^*2_3 1_2 2_2) < 3.00000758$. 
\item[(ii)] $\lambda_0^+(1_2 2_2 1_4 2^* 2 1_2 2_2 1_2) < \lambda_0^+(1_2 2_2 1_4 2^* 2 1_2 2_4) < 3.0001551$.
\item[(iii)] $\lambda_0^-(1_5 2^* 2 1_2 2_4) > \lambda_0(1_5 2^* 2 1_2 2_2 1_2) > 3.003$. 
\item[(iv)] $\lambda_0^+(2 1_4 2^* 2 1_4) < 3$. 
\end{itemize} 
\end{lemma} 

By Lemma \ref{l.5-3} and Remark \ref{r.2}, if $x$ is $(2,3.00073)$-admissible, then 
\begin{itemize}
\item $x=\dots 1_5 2^* 2 1_4\dots$ or 
\item $x=\dots 1_4 2_2 1_2 2^* 2_3 1_2 2_2\dots$ or   
\item $x=\dots 2_2 1_2 2_4 1_2 2^* 2_4\dots$ or 
\item $x=\dots 2_5 1_2 2^* 2_4\dots$ 
\end{itemize}

\begin{lemma}\label{l.5-4}
\begin{itemize}
\item[(i)] $\lambda_0^+(2_4 1_2 2^* 2_5) < 3.00005$. 
\item[(ii)] $\lambda_0^+(2_5 1_2 2^* 2_4 1_2 2) < 3.0001426$. 
\item[(iii)] $\lambda_0^+(1_4 2_2 1_2 2^* 2_3 1_2 2_2 1) < \lambda_0^+(1_4 2_2 1_2 2^* 2_3 1_2 2_3) < 3.0001544$. 
\item[(iv)] $\lambda_0^-(2 1_5 2^* 2 1_4)>3.0014$. 
\end{itemize}
\end{lemma}

By Lemma \ref{l.5-4}, if $x$ is $(2,3.00073)$-admissible, then 
\begin{itemize}
\item $x=\dots 1_6 2^* 2 1_5\dots$ or 
\item $x=\dots 1_6 2^* 2 1_4 2_2 1_2\dots$ or    
\item $x=\dots 2_2 1_2 2_4 1_2 2^* 2_4 1_2 2_2\dots$  
\end{itemize}

\begin{lemma}\label{l.5-5} 
\begin{itemize}
\item[(i)] $\lambda_0^+(1_6 2^* 2 1_5 2 \dots) < 3.000083$. 
\item[(ii)] $\lambda_0^-(2_2 1_2 2_4 1_2 2^* 2_4 1_2 2_2 1 \dots) > 3.0001647$. 
\end{itemize} 
\end{lemma}

By Lemma \ref{l.5-5} and Lemma \ref{l.5-1}, 
if $x$ is $(2,3.00073)$-admissible, then 
\begin{itemize}
\item $x=\dots 1_6 2^* 2 1_6\dots$ or  
\item $x=\dots 1_6 2^* 2 1_4 2_2 1_4\dots$ or  
\item $x=\dots 1_6 2^* 2 1_4 2_2 1_2 2_2\dots$ or   
\item $x=\dots 2_2 1_2 2_4 1_2 2^* 2_4 1_2 2_4 \dots$  
\end{itemize}

\begin{lemma}\label{l.5-6}
\begin{itemize}
\item[(i)] $\lambda_0^-(1 2_2 1_2 2_4 1_2 2^* 2_4 1_2 2_4) > 3.00016432$.
\item[(ii)] $\lambda_0^-(1_7 2^*2 1_4 2_2 1_2) > 3.000545$. 
\item[(iii)] $\lambda_0^+(21_6 2^* 2 1_6) < \lambda_0^+(2 1_6 2^* 2 1_4 2_2 1_4) < \lambda_0^+(2 1_6 2^* 2 1_4 2_2 1_2 2_2) < 3.000014$. 
\end{itemize}
\end{lemma}

By Lemma \ref{l.5-6} and Lemma \ref{l.5-1}, if $x$ is $(2,3.00016432)$-admissible, then 
\begin{itemize}
\item $x=\dots 1_7 2^* 2 1_6\dots$ or     
\item $x=\dots 2_4 1_2 2_4 1_2 2^* 2_4 1_2 2_4 \dots$  
\end{itemize}

\begin{lemma}\label{l.5-7} $\lambda_0^-(2_4 1_2 2_4 1_2 2^* 2_4 1_2 2_4 1) > 3.000164247$.  
\end{lemma}

By Lemma \ref{l.5-7}, if $x$ is $(2, 3.000164247)$-admissible, then 
\begin{itemize}
\item $x=\dots 1_7 2^* 2 1_7\dots$ or     
\item $x=\dots 1_7 2^* 2 1_6 2_2 1_2\dots$ or 
\item $x=\dots 2_4 1_2 2_4 1_2 2^* 2_4 1_2 2_5 \dots$  
\end{itemize}

\begin{lemma}\label{l.5-8}
\begin{itemize}
\item[(i)] $\lambda_0^+(1_8 2^* 2 1_7) < \lambda_0^+(1_8 2^* 2 1_6 2_2) < 3.0001516$ 
\item[(ii)] $\lambda_0^-(2 1_7 2^* 2 1_6\dots) > 3.0002048$.
\end{itemize}
\end{lemma}

By Lemma \ref{l.5-8}, if $x$ is $(2, 3.000164247)$-admissible, then 
\begin{itemize}
\item $x=\dots 2_2 1_2 2_4 1_2 2_4 1_2 2^* 2_4 1_2 2_5 \dots$ or 
\item $x=\dots 2_5 1_2 2_4 1_2 2^* 2_4 1_2 2_5 \dots$  
\end{itemize}

\begin{lemma}\label{l.5-9} 
\begin{itemize}
\item[(i)] $\lambda_0^-(2_2 1_2 2_4 1_2 2_4 1_2 2^* 2_4 1_2 2_6) > 3.000164233$. 
\item[(ii)] $\lambda_0^+(2_5 1_2 2_4 1_2 2^* 2_4 1_2 2_5 1) < \lambda_0^+(2_2 1_2 2_4 1_2 2_4 1_2 2^* 2_4 1_2 2_5 1) < 3.00016423076$. 
\end{itemize}
\end{lemma} 

By Lemma \ref{l.5-9}, if $x$ is $(2, 3.000164233)$-admissible, then 
\begin{itemize}   
\item $x=\dots 2_5 1_2 2_4 1_2 2^* 2_4 1_2 2_6 \dots$  
\end{itemize} 

\begin{lemma}\label{l.5-10} 
\begin{itemize}
\item[(i)] $\lambda_0^+(1 2_5 1_2 2_4 1_2 2^* 2_4 1_2 2_6) < 3.00016423103$. 
\item[(ii)] $\lambda_0^+(2_6 1_2 2_4 1_2 2^* 2_4 1_2 2_7) <3.00016423078$.
\item[(iii)] $\lambda_0^+(2_7 1_2 2_4 1_2 2^* 2_4 1_2 2_6 1_2 2) < 3.00016423114$. 
\end{itemize}
\end{lemma}

By Lemma \ref{l.5-10} and Remark \ref{r.2}, if $x$ is $(2, 3.000164233)$-admissible, then 
\begin{itemize}   
\item $x=\dots 2_2 1_2 2_6 1_2 2_4 1_2 2^* 2_4 1_2 2_6 1_2 2_2\dots$    
\end{itemize}

By Lemma \ref{l.5-1}, it follows that if $x$ is $(2, 3.000164233)$-admissible, then 
\begin{itemize}   
\item $x=\dots 2_4 1_2 2_6 1_2 2_4 1_2 2^* 2_4 1_2 2_6 1_2 2_4\dots$    
\end{itemize}

\begin{lemma}\label{l.5-11} $\lambda_0^+(2_4 1_2 2_6 1_2 2_4 1_2 2^* 2_4 1_2 2_6 1_2 2_5) < 3.000164231218146$. 
\end{lemma}

By Lemma \ref{l.5-11} (and Remark \ref{r.2}), if $x$ is $(2, 3.000164233)$-admissible, then 
\begin{itemize}   
\item $x=\dots 2_4 1_2 2_6 1_2 2_4 1_2 2^* 2_4 1_2 2_6 1_2 2_4 1_2 2_2\dots$    
\end{itemize}

Hence, Lemma \ref{l.5-11} implies the desired local uniqueness result for $\gamma_2^1$: 
\begin{lemma}[Local uniqueness of $\gamma_2^1$]\label{l.local-m2} A $(2, 3.000164233)$-admissible word $\theta$ has the form  
$$\theta=\dots 2_4 1_2 2_6 1_2 2_4 1_2 2^* 2_4 1_2 2_6 1_2 2_4 1_2 2_2\dots$$
In particular, it contains the string $\theta_2^0 = 2_{4}1_22_{6}1_22_{4}1_22^*2_{4}1_22_{6}1_22_{4}1$.
\end{lemma}

\section{Local uniqueness for $\gamma_3^1$}\label{s.k3}

We begin noticing that
$$m(\theta(\underline{\omega}_3))=3.0000048343047763824279744223474498423...$$
and 
$$m(\gamma^1_3)=3.0000048343047763824279744223474498428...$$

By Corollary \ref{c.2}, up to transposition, a $(3,3.009)$-admissible word has the form \begin{itemize}
\item $\theta=\dots 1_42^*21_2\dots$ or 
\item $\theta=\dots 2_21_22^*2_2\dots$
\end{itemize}

Let us now show that $1_42^*21_2$ can not extend into a $(3, 3.0000075)$-admissible word. By Lemma \ref{l.4}, $\theta$ extends as either $\theta=...1_42^*21_4\dots$ or $\theta=...1_42^*21_22_2...$. 

\begin{description}
\item[$\bullet$] If $\theta=...1_42^*21_4\dots$, it must extend on the left as either $\theta=...1_52^*21_4$ or $...1_22_21_42^*21_4...$ because $2_31_3$ and $212$ are prohibited (cf. Lemma \ref{l.4}). 

\item[$\bullet$] If $\theta=...1_42^*21_22_2...$, it must extend on the left as either $\theta=...1_52^*21_22_2...$ or $1_22_21_42^*21_22_2...$
\end{description}

For subsequent use, recall from Lemma \ref{l.4} and the case $k=2$ (i.e., Section \ref{s.k2}) the following $2$-prohibited strings: 

\begin{enumerate}
\item[(1)] $v_1=21_32^*21_2$ and $v_2=21_32^*21_3$

\item[(2)] $v_3=1_52^*21_22_21_2$ and $v_4=1_52^*21_22_4$

\item[(3)] $v_5=21_52^*21_4$ and $v_6=1_72^*21_42_21_2$

\item[(4)] $v_7=21_72^*21_6$ 
\end{enumerate}

\begin{lemma}
$\lambda^+_0(1_22_21_42^*21_4)<2.997.$
\end{lemma}
This Lemma implies either $\theta=...1_52^*21_4...$ or $\theta=...1_52^*21_22_2...$ or $\theta=...1_22_21_42^*21_22_2...$.

\begin{description}
\item[$\bullet$] If $\theta=...1_52^*21_4$ on the right hand side we have either $\theta=...1_52^*21_5...$ or $...1_52^*21_42_21_2...$ because $2_31_3$ and $212$ are prohibited. 

\item[$\bullet$] If $\theta=...1_52^*21_22_2...$ on the right hand side we must  have $...1_52^*21_22_31_22_2...$ because $\theta=...1_52^*21_22_21_2...$ and $v_3$ and $v_4$ are prohibited.

\item[$\bullet$] If $\theta=...1_22_21_42^*21_22_2...$ on the right hand side we have either $\theta=...1_22_21_42^*21_22_21_2...$ or $...1_22_21_42^*21_22_3...$

\end{description}

\begin{lemma}
$\lambda^-_0(1_52^*21_22_31_22_2)>\lambda^-_0(1_22_21_42^*21_22_3)>3.0001$, $\lambda^-_0(1_22_21_52^*21_42_21_2)>3.002$
\end{lemma}
This lemma implies either $\theta=.....1_52^*21_5$ or $...1_52^*21_42_21_2...$ or $...1_22_21_42^*21_22_21_2...$. Since $v_5$ is $2$-prohibited if $\theta=...1_52^*21_5...$,
then $\theta=...1_62^*21_6...$. If $\theta=...1_52^*21_42_21_2...$ then $\theta=...1_62^*21_42_21_2...$ because by above Lemma $1_22_21_52^*21_42_21_2$ is prohibited.
If $\theta=...1_22_21_42^*21_22_21_2...$ then either $\theta=...1_42_21_42^*21_22_21_2...$, because $21_32_21_2$ is prohibited, or $\theta=...2_21_22_21_42^*21_22_21_2...$

\begin{lemma}
$\lambda^+_0(1_22_21_62^*21_6)<3$
\end{lemma}
From this we have, the left continuations are 
\begin{description}
\item[$\bullet$] $\theta=...1_72^*21_6$ or $\theta=...1_22_21_62^*21_6...$ but this string is $k$-avoided, $k\ge 1$.

\item[$\bullet$] $\theta=...1_72^*21_42_21_2...=...v_6...$ is $2$-prohibited or $\theta=...1_22_21_62^*21_42_21_2...$

\item[$\bullet$] $\theta=...1_52_21_42^*21_22_21_2...$ or $\theta=...1_22_21_42_21_42^*21_22_21_2...$

\item[$\bullet$] $\theta=...1_22_21_22_21_42^*21_22_21_2...$ or $\theta=...2_31_22_21_42^*21_22_21_2...$
\end{description}
The right continuations of the above possible admissible words are

\begin{description}
\item[$\bullet$] $\theta=...1_72^*21_7$ or $\theta=...1_72^*21_62_21_2...$

The first one implies $\theta=...1_82^*21_8...$ because $v_7=21_72^*21_6$ is $2$-prohibited. In the same way the second one implies $\theta=...1_82^*21_62_21_2...$.

\item[$\bullet$] $\theta=...1_22_21_62^*21_42_21_3...$ or $\theta=...1_22_21_62^*21_42_21_22_2...$

Since $v_1$ and $v_2$ (i.e., $21_32_21_2$) are prohibited,  in the former we must have $\theta=...1_22_21_62^*21_42_21_4...$

\item[$\bullet$] $\theta=...1_52_21_42^*21_22_21_3...$ or $\theta=...1_52_21_42^*21_22_21_22_2$ 

\item[$\bullet$] $\theta=...1_22_21_42_21_42^*21_22_21_3...$ or $\theta=...1_22_21_42_21_42^*21_22_21_22_2...$

\item[$\bullet$] $\theta=...1_22_21_22_21_42^*21_22_21_4...$ or $\theta=...1_22_21_22_21_42^*21_22_21_22_2...$

\item[$\bullet$] $\theta=...2_31_22_21_42^*21_22_21_4...$ or $\theta=...2_31_22_21_42^*21_22_21_22_2...$
\end{description}

\begin{lemma}
$\lambda^+_0(1_22_21_62^*21_42_21_4)<3.00000211$, $\lambda^+_0(1_22_21_62^*21_42_21_22_2)<3.00000469$
\end{lemma}

\begin{lemma}We have
$\lambda^+_0(1_52_21_42^*21_22_21_4)<3$, $\lambda^+_0(1_52_21_42^*21_22_21_22_2)<3$. 
\end{lemma}

\begin{lemma}
$\lambda^+_0(1_22_21_42_21_42^*21_22_21_4)<3.0000009352$ and $\lambda^-_0(1_22_21_42_21_42^*21_22_21_22_2)>3.00001$
\end{lemma}

\begin{lemma}
$\lambda^+_0(1_22_21_22_21_42^*21_22_21_4)<3$ and $\lambda^+_0(1_22_21_22_21_42^*21_22_21_22_2)<3.000001133$.
\end{lemma}

\begin{lemma} \label{in.k=4}
$\lambda^+_0(2_31_22_21_42^*21_22_21_4)<3$ and $\lambda^+_0(2_31_22_21_42^*21_22_21_22_2)<3.00000019457$.
\end{lemma}
It follows from these Lemmas that either $\theta=...1_82^*21_8...$ or $\theta=...1_82^*21_62_21_2...$. Their left hand side continuations are

\begin{description}

\item[$\bullet$] $\theta=...1_92^*21_8...$ or $\theta=...1_22_21_82^*21_8...$

\item[$\bullet$] $\theta=...1_92^*21_62_21_2...$ or $\theta=...1_22_21_82^*21_62_21_2...$.

\end{description}

\begin{lemma} \label{l5.9}
$\lambda^+_0(1_22_21_82^*21_8)<3$, $\lambda^-_0(1_92^*21_62_21_2)>3.00007$ and $\lambda^+_0(1_22_21_82^*21_62_21_2)<3.00000080093$
\end{lemma}
From the above Lemma we must have $\theta=...1_92^*21_8...$. This word has two right continuations, $\theta=...1_92^*21_9...$ and $\theta=...1_92^*21_82_21_2...$.
in the first case, the two left hand side continuation are $\theta=...1_{10}2^*21_9...$ and $...1_22_21_92^*21_9...$. In the second case we have, $\theta=...1_{10}2^*21_82_21_2...$ or $\theta=...1_22_21_92^*21_82_21_2...$

\begin{lemma} \label{l5.10}
$\lambda^-_0(1_22_21_92^*21_9)>3.00003$ and $\lambda^-_0(1_22_21_92^*21_82_21_2)>3.00005$.
\end{lemma}
By above Lemma, either $\theta=...1_{10}2^*21_9...$ or $\theta=...1_{10}2^*21_82_21_2...$. In the first case this implies $\theta=...1_{10}2^*21_{10}...$, because $\theta=...1_{10}2^*21_92_21_2...$ contains $1_{9}2^*21_92_21_2$ and by above Lemma this string is $3$-prohibited. In the second case $\theta=...1_{10}2^*21_82_21_4...$ or $\theta=...1_{10}2^*21_82_21_22_2...$. The continuations on the left hand side are

\begin{description}
\item[$\bullet$] $\theta=...1_{11}2^*21_{10}...$ or $\theta=...1_22_21_{10}2^*21_{10}...$

\item[$\bullet$] $\theta=...1_{11}2^*21_82_21_4...$ or $\theta=...1_22_21_{10}2^*21_82_21_4...$

\item[$\bullet$] $\theta=...1_{11}2^*21_82_21_22_2...$ or $\theta=...1_22_21_{10}2^*21_82_21_22_2...$
\end{description}

\begin{lemma} \label{l5.11}We have that 
\begin{enumerate}
\item[(i)] $\lambda^+_0(1_22_21_{10}2^*21_{10})<3$, 

\item[(ii)] $\lambda^+_0(1_22_21_{10}2^*21_82_21_4)<3.000000044$

\item[(iii)] $\lambda^+_0(1_22_21_{10}2^*21_82_21_22_2)<3.000000099$

\item[(iv)] $\lambda^-_0(1_{11}2^*21_82_21_4)>3.00001$,

\item[(v)] $\lambda^-_0(1_{11}2^*21_82_21_22_2)>3.00001$
\end{enumerate}

\end{lemma}
From above Lemma we must have $\theta=...1_{11}2^*21_{10}...$. From the right side, we must continue as $\theta=...1_{11}2^*21_{11}$ or $\theta=...1_{11}2^*21_{10}2_21_2...$. Each word have the following continuations on the left hand side

	\begin{description}
	\item[$\bullet$] $\theta=...1_{12}2^*21_{11}...$ or $\theta=...1_22_21_{11}2^*21_{11}...$

	\item[$\bullet$] $\theta=...1_{12}2^*21_{10}2_21_2...$ or $\theta=...1_22_21_{11}2^*21_{10}2_21_2...$
	\end{description}

\begin{lemma} \label{l5.12}
The string $1_{12}2^*21_9$ is $3$-avoided. Moreover, $\lambda^-_0(1_22_21_{11}2^*21_{10}2_21_2)>3.0000075$.
\end{lemma}
\begin{proof}
In fact, we note that
$\lambda^+_0(u)=[2,2,1_9,\overline{1,2}]+[0,1_{12},\overline{1,2}]<3.000003786<m(\theta(\underline{\omega}_3)).$
\end{proof}
The above Lemma implies that $\theta=...1_22_21_{11}2^*21_{11}...$. The above Lemma still implies that the continuation on the right hand side is $\theta=...1_22_21_{11}2^*21_{11}2_21_2...$
\begin{lemma}
$\lambda^+_0(1_22_21_{11}2^*21_{11}2_21_2)<3.00000473.$
\end{lemma}
From the above Lemma, no word of the form $\theta=\dots1_42^*21_2\dots$ can be $(3,3.0000075)$- admissible. Therefore: 
\begin{corollary} Any $(3,3.0000075)$-admissible word $\theta$ has the form $\theta=...2_21_22^*2_2...$.
\end{corollary} 

By Lemma \ref{l.5-1} we need to continue as $\theta=...2_21_22^*2_3...$ which continue as 
	\begin{description}
	\item[$\bullet$] $\theta=...1_22_21_22^*2_3...$ or $2_31_22^*2_{3}...$.
	\end{description}
By Lemmas \ref{l.5-2} and \ref{l.4}, $\theta=...1_22_21_22^*2_3...$ must to continue as $\theta=...1_22_21_22^*2_31_22_2...$ and $\theta=...2_31_22^*2_3...$ must to continue as $\theta=...2_31_22^*2_4...$.
\begin{lemma}\label{l5.13}
$\lambda^-_0(1_32_21_22^*2_{3}1_22_2)>3.0001$.
\end{lemma}
By Remark \ref{r.2}, Lemmas \ref{l.5-1} and \ref{l5.13}, any $(3,3.0000075)$-admissible word must be continued as 
	\begin{description}
    \item[$\bullet$]$\theta=...2_21_22_21_22^*2_{3}1_22_2...$ or		
    \item[$\bullet$] $\theta=...2_41_22^*2_{4}...$.
    \end{description}

\begin{lemma}\label{l5.14}
\begin{itemize}
\item[(i)]$\lambda^-_0(2_41_22^*2_{4}1_22_2)>3.0001$.
\item[(ii)]$\lambda^+_0(2_21_22_21_22^*2_{3}1_22_21_2)<3.000003$.
\end{itemize}
\end{lemma}
By Lemmas \ref{l.4}, \ref{l5.14} and \ref{l.5-1} any $(3,3.0000075)$-admissible word must be continued as 
	\begin{description}
	\item[$\bullet$] $\theta=...1_22_21_22_21_22^*2_{3}1_22_4...$ or $\theta=...2_31_22_21_22^*2_{3}1_22_4...$ or
	\item[$\bullet$] $\theta=...2_21_22_41_22^*2_{5}...$ or $\theta=...2_51_22^*2_5...$
	\end{description}
\begin{lemma}\label{l5.15}
\begin{itemize}
\item[(i)] $\lambda^+_0(2_21_22_41_22^*2_51_22_2)<3.00000023$.
\item[(ii)]$\lambda^+_0(2_51_22^*2_51_22_2)<3$
\item[(iii)] $\lambda^-_0(1_22_41_22^*2_6)>3.00002$
\item[(iv)]$\lambda^+_0(2_31_22_21_22^*2_31_22_4)<\lambda^+_0(1_22_21_22_21_22^*2_31_22_4)<3.0000047$
\end{itemize}
\end{lemma}
By Lemmas \ref{l.4} and \ref{l5.15} any $(3,3.0000075)$-admissible word must be continued as
	\begin{description}
	\item[$\bullet$] $\theta=...2_21_22_51_22^*2_6...$ or $\theta=...2_61_22^*2_6...$
	\end{description}
\begin{lemma}\label{l5.16}
$\lambda^+_0(2_21_22_51_22^*2_6)<3.0000032$.
\end{lemma}
By Lemmas \ref{l.4} and \ref{l5.16} $\theta=...2_61_22^*2_61_22_2...$ or $\theta=...2_61_22^*2_7...$.
\begin{lemma}\label{l5.17} We have
$$\lambda^+_0(2_71_22^*2_7)<\lambda^+_0(2_21_22_61_22^*2_7)<\lambda^+_0(2_71_22^*2_{6}1_22_2)<3.000004196.$$
\end{lemma}
By Lemmas \ref{l.4} and \ref{l5.17}, any $(3,3.000075)$-admissible word must has center 
	$$\theta=...2_21_22_61_22^*2_61_22_2...$$
\begin{lemma}\label{l5.18}
$\lambda^-_0(2_21_22_61_22^*2_61_22_21_2)>3.00000485$.
\end{lemma}	
By Lemmas \ref{l.4} and \ref{l5.18}, any $(3,3.00000485)$ word $\theta$ must to continue as 
\begin{description}
\item[$\bullet$] $\theta=...1_22_21_22_61_22^*2_61_22_3...$ or $\theta=...2_31_22_61_22^*2_61_22_3...$
\end{description}
\begin{lemma}\label{l5.19}
$\lambda^-_0(1_22_21_22^*2_5)>3.0007$
\end{lemma}
Lemma \ref{l5.19} implies that $\theta=...1_22_21_22_61_22^*2_61_22_3...$ is not $(3,3.00000485)$-admissible word. Then, we must to continue as
	\begin{description}
		\item[$\bullet$] $\theta=...2_31_22_61_22^*2_61_22_31_22_2...$ or $\theta=...2_31_22_61_22^*2_61_22_4...$
	\end{description}
\begin{lemma}\label{l5.20}
$\lambda^+_0(2_ 31_ 22_ 61_ 22^*2_ 61_ 22_ 31_ 22_ 2)<3.000004832$.
\end{lemma}
By Lemmas \ref{l.4} and \ref{l5.20} any $(3,3.00000485)$-admissible word must to continue as
	\begin{description}
		\item[$\bullet$] $\theta=...2_21_22_31_22_61_22^*2_61_22_4...$ or $\theta=...2_41_22_61_22^*2_61_22_4...$
	\end{description}	
For each possible $(3.300000485)$-admissible words above we have the right continuations
	\begin{description}
	\item[$\bullet$] $\theta=...2_21_22_31_22_61_22^*2_61_22_41_22_2...$ or $\theta=...2_21_22_31_22_61_22^*2_61_22_5...$
	\item[$\bullet$] $\theta=...2_41_22_61_22^*2_61_22_41_2...$ or $\theta=...2_41_22_61_22^*2_61_22_5...$
	\end{description}
\begin{lemma}\label{l5.21}
$\lambda^+_0(2_21_22_31_22_61_22^*2_61_22_5)<3.000004834$ and $\lambda^-_0(1_22_21_22^*2_6)>3.00002.$
\end{lemma}
By Lemma \ref{l.4} and \ref{l5.21} we must to continue as
\begin{description}
\item[$\bullet$] $\theta=...2_21_22_41_22_61_22^*2_61_22_5...$ or $\theta=...2_51_22_61_22^*2_61_22_5...$
\end{description}
The right continuation are
	\begin{description}
	\item[$\bullet$] $\theta=...2_21_22_41_22_61_22^*2_61_22_51_22_2...$ or $\theta=...2_21_22_41_22_61_22^*2_61_22_6...$
	\item[$\bullet$] $\theta=...2_51_22_61_22^*2_61_22_51_22_2...$ or $\theta=...2_51_22_61_22^*2_61_22_6...$
	\end{description}
\begin{lemma}\label{l5.22}
$\lambda^+_0(2_51_22_61_22^*2_61_22_51_22_2)<3.0000048343$ and $\lambda^-_0(2_21_22_41_22_61_22^*2_61_22_6)>3.00000483439$.
\end{lemma}
By Lemmas \ref{l.4} and \ref{l5.22} any $(3,3.00000483439)$-admissible word must to continue as
	\begin{description}
	\item[$\bullet$] $\theta=...1_22_21_22_41_22_61_22^*2_61_22_51_22_2...$ or $\theta=...2_31_22_41_22_61_22^*2_61_22_51_22_2...$
	\item[$\bullet$] $\theta=...2_21_22_51_22_61_22^*2_61_22_6...$ or $\theta=...2_61_22_61_22^*2_61_22_6...$
	\end{description}
\begin{lemma}\label{l5.23}	
	$\lambda^+_0(2_31_22_41_22_61_22^*2_61_22_51_22_2)<3.00000483430437$
\end{lemma}
	By Lemmas \ref{l.4} and \ref{l5.23} we must to continue as
	\begin{description}
	\item[$\bullet$] $\theta=...1_22_21_22_41_22_61_22^*2_61_22_51_22_21_2...$ or $\theta=...1_22_21_22_41_22_61_22^*2_61_22_51_22_3...$
	
	\item[$\bullet$] $\theta=...2_21_22_51_22_61_22^*2_61_22_61_22_2...$ or $\theta=...2_21_22_51_22_61_22^*2_61_22_7...$
	
	\item[$\bullet$] $\theta=...2_61_22_61_22^*2_61_22_61_22_2...$ or $\theta=...2_61_22_61_22^*2_61_22_7...$
	\end{description}
	
\begin{lemma}\label{l5.24}
$\lambda^+_0(2_21_22_51_22_61_22^*2_61_22_7)<\lambda^+_0(2_21_22_51_22_61_22^*2_61_22_61_22_2)<3.0000048343044.$ Moreover, $\lambda^-_0(2_61_22_61_22^*2_61_22_61_22_2)>3.00000483431$.
\end{lemma}
By Lemmas \ref{l.4}, \ref{l.5-2} and \ref{l5.24} any $(3,3.00000483431)$-admissible word must to continue as
	\begin{description}
	\item[$\bullet$] $\theta=...1_32_21_22_41_22_61_22^*2_61_22_51_22_3...$ or $\theta=...2_21_22_21_22_41_22_61_22^*2_61_22_51_22_3...$
	
	\item[$\bullet$] $\theta=...2_21_22_61_22_61_22^*2_61_22_7...$ or $\theta=...2_71_22_61_22^*2_61_22_7...$
	\end{description}
\begin{lemma}
$\lambda^+_0(2_21_22_21_22_41_22_61_22^*2_61_22_51_22_3)<3.0000048340772$.
\end{lemma}\label{l5.25}
Since $1_32_21_22_41_2$ is $3$-prohibited, by above Lemma we must have
	\begin{description}
	\item[$\bullet$] $\theta=...2_21_22_61_22_61_22^*2_61_22_71_22_2...$ or $\theta=...2_21_22_61_22_61_22^*2_61_22_8...$
	\item[$\bullet$] $\theta=...2_71_22_61_22^*2_61_22_71_22_2...$ or $\theta=...2_71_22_61_22^*2_61_22_8...$
	\end{description}
\begin{lemma}\label{l5.26} 
We have
	$$\lambda^+_0(2_71_22_61_22^*2_61_22_71_22_2)<\lambda^+_0(2_21_22_61_22_61_22^*2_61_22_71_22_2)<3.0000048343043.$$
Moreover, $\lambda^-_0(2_21_22_61_22_61_22^*2_61_22_8)>3.000004834306$.
\end{lemma}
By Lemmas \ref{l.4} and \ref{l5.26} any $(3,3.000004834306)$-admissible word must to continue as
	$$\theta=...2_21_22_71_22_61_22^*2_61_22_8... \quad \mbox{or} \quad \theta=...2_81_22_61_22^*2_61_22_8...$$
\begin{lemma}\label{l5.27}
$\lambda^+_0(2_ 81_ 22_ 61_ 22^*2_ 61_ 22_ 9)<\lambda^+_0(2_21_22_71_22_61_22^*2_61_22_8)<3.0000048343046.$
\end{lemma}
By Lemmas \ref{l.4} and \ref{l5.27} any $(3,3.000004834306)$-admissible word must to continue as
	$$\theta=...2_21_22_81_22_61_22^*2_61_22_81_22_2... \quad \mbox{or} \quad \theta=...2_91_22_61_22^*2_61_22_81_22_2...$$
\begin{lemma}\label{l5.28}
$\lambda^+_0(2_91_22_61_22^*2_61_22_81_22_2)<3.00000483430471$
\end{lemma}
The above Lemma implies that any $(3,3.000004834306)$-admissible word must has center
	$$\theta=...2_21_22_81_22_61_22^*2_61_22_81_22_2...$$
By Lemma \ref{l5.19} the string $1_22_21_22_6$ and its transposition are $3$-prohibited, then any $(3,3.000004834306)$-admissible word must to continue as
	$$\theta=...2_31_22_81_22_61_22^*2_61_22_81_22_3...$$
\begin{lemma}\label{l5.29}
$\lambda^+_0(2_31_22_81_22_61_22^*2_61_22_81_22_31_22_2)<3.0000048343047761$
\end{lemma}	
By Lemmas \ref{l.4} and \ref{l5.29} any $(3,3.000004834306)$-admissible word must to continue as
	$$\theta=...2_21_22_31_22_81_22_61_22^*2_61_22_81_22_4... \quad \mbox{or} \quad \theta=...2_41_22_81_22_61_22^*2_61_22_81_22_4...$$
\begin{lemma}\label{l5.30} Let $\lambda^{(3)}=3.00000483430477639$. We have
\begin{eqnarray*}\lambda^+_0(2_21_22_31_22_81_22_61_22^*2_61_22_81_22_5)<\lambda^+_0(2_21_22_31_22_81_22_61_22^*2_61_22_81_22_41_22_2)<\lambda^{(3)}.
\end{eqnarray*}
Moreover, $\lambda^-_0(2_41_22_81_22_61_22^*2_61_22_81_22_41_22_2)>3.0000048343047764>\lambda^{(3)}$
\end{lemma}
The above Lemma implies that any $(3,\lambda^{(3)})$-admissible word must has center
$$\theta=...2_21_22_41_22_81_22_61_22^*2_61_22_81_22_5... \quad \mbox{or} \quad \theta=...2_51_22_81_22_61_22^*2_61_22_81_22_5...$$
\begin{lemma}\label{l5.31}
$\lambda^-_0(2_21_22_41_22_81_22_61_22^*2_61_22_81_22_6)>3.00000483430477639=\lambda^{(3)}.$ Moreover,
$\lambda^+_0(2_51_22_81_22_61_22^*2_61_22_81_22_51_22_2)<\lambda^+_0(2_21_22_41_22_81_22_61_22^*2_61_22_81_22_51_22_2)<3.000004834304776381.$
\end{lemma}
By Lemmas \ref{l.4} and \ref{l5.31} any $(3,\lambda^{(3)})$-admissible word must has center
$$\theta=...2_21_22_51_22_81_22_61_22^*2_61_22_81_22_6... \quad \mbox{or} \quad \theta=...2_61_22_81_22_61_22^*2_61_22_81_22_6...$$

\begin{lemma}\label{l5.32}
$\lambda^+_0(2_61_22_81_22_61_22^*2_61_22_81_22_7)<\lambda^+_0(2_21_22_51_22_81_22_61_22^*2_61_22_81_22_6)<3.0000048343047763821$.
\end{lemma}
It follows by Lemmas \ref{l.4} and \ref{l5.32} that

 \begin{lemma}[Local uniqueness of $\gamma_3^1$]\label{l.local-m3} A $(3, \lambda^{(3)})$-admissible word $\theta$ has the form  
$$\theta=\dots 2_{6}1_22_{8}1_22_{6}1_22^*2_{6}1_22_{8}1_22_{6} 1_2 2_2\dots$$
In particular, it contains the string $\theta_3^0=2_{6}1_22_{8}1_22_{6}1_22^*2_{6}1_22_{8}1_22_{6}1$.
\end{lemma}
 
\section{Local uniqueness for $\gamma_4^1$}\label{s.k4}

Note that:	

\begin{eqnarray*}
& & m(\theta(\underline{\omega}_4)) = \lambda_0(\overline{2_91_22_{10}1_22_81_2}2^*2_81_22_{10}1_2\overline{2_81_22_91_22_{10}1_2}) \\
& & = 3.00000014230846289515772187541301530809498052633\dots
\end{eqnarray*} 
and 
\begin{eqnarray*}
& & m(\gamma_4^1) = \lambda_0(\overline{2_91_22_{10}1_22_81_2}2^*2_81_22_{10}1_22_81_22_91_22_{10}1_2 \overline{2}) \\
& & = 3.00000014230846289515772187541301530809498052669\dots.
\end{eqnarray*}

By Corollary \ref{c.2}, up to transposition, a $(4,3.009)$-admissible word $\theta$ is 
\begin{itemize}
\item[(a)] $\theta=\dots 1_42^*21_2\dots$ or 
\item[(b)] $\theta=\dots 2_21_22^*2_2\dots.$
\end{itemize}

First, we start studying the possible continuations of $\theta$ with central combinatorics in the branch $(a)$. By previous sections, after the Lemma \ref{in.k=4}, if $\theta$ in the branch $(a)$ is $(4,3.0001)$-admissible, then
\begin{description}
\item[$\bullet$]  $\theta=...1_82^*21_8...$ or $\theta=...1_82^*21_62_21_2...$;

\item[$\bullet$] $\theta=...1_22_21_62^*21_42_21_4...$ or $\theta=...1_22_21_62^*21_42_21_22_2...$;

\item[$\bullet$] $\theta=...1_22_21_42_21_42^*21_22_21_4...$, because $21_32_21_2$ is prohibited; 

\item[$\bullet$] $\theta=...2_31_22_21_42^*21_22_21_22_2...$ or $\theta=...1_22_21_22_21_42^*21_22_21_22_2...$.
\end{description}
\begin{lemma}\label{l7.1}
\begin{enumerate}
\item[(i)] $\lambda^+_0(2_21_22_21_62^*21_42_21_4)<3$ 
\item[(ii)] $\lambda^-_0(1_42_21_62^*21_42_21_22_2)>3.0000023$
\item[(iii)]$\lambda^+_0(1_22_31_22_21_42^*21_22_21_22_2)<3.00000008$ 
\end{enumerate}
\end{lemma}

By Remark \ref{r.2}, Lemmas \ref{l5.9} and \ref{l7.1}, if $\theta$ in the branch $(a)$ is $(4,3.0000023)$-admissible, then 
\begin{description}
\item[$\bullet$]  $\theta=...1_92^*21_8...$,
\item[$\bullet$] $\theta=...1_22_21_82^*21_62_21_2...$,

\item[$\bullet$] $\theta=...1_42_21_62^*21_42_21_4...$ ,
\item[$\bullet$] $\theta=...2_21_22_21_62^*21_42_21_22_2...$,

\item[$\bullet$] $\theta=...1_42_21_42_21_42^*21_22_21_4...$ or $\theta=...2_21_22_21_42_21_42^*21_22_21_4...$, 

\item[$\bullet$] $\theta=...2_41_22_21_42^*21_22_21_22_2...$,
\item[$\bullet$] $\theta=...1_42_21_22_21_42^*21_22_21_22_2...$ or $\theta=...2_21_22_21_22_21_42^*21_22_21_22_2...$,
\end{description}
where $3.00000008<m(\theta)=\lambda_0(\theta)<3.0000023$.

By Remark \ref{r.2}, the possible continuation of these words on the right hand side are
\begin{description}
\item[$\bullet$]  $\theta=...1_92^*21_9...$ or $\theta=...1_92^*21_82_21_2...$,
\item[$\bullet$] $\theta=...1_22_21_82^*21_62_21_4...$ or $\theta=...1_22_21_82^*21_62_21_22_2...$,

\item[$\bullet$] $\theta=...1_42_21_62^*21_42_21_6...$ or $\theta=...1_42_21_62^*21_42_21_42_2...$,  because $v_5$ is $2$-prohibited, 
\item[$\bullet$] $\theta=...2_21_22_21_62^*21_42_21_22_21_2...$ or $\theta=...2_21_22_21_62^*21_42_21_22_3...$,

\item[$\bullet$] $\theta=...1_42_21_42_21_42^*21_22_21_42_21_2...$, because $v_3$ is $2$-prohibited,
\item[$\bullet$] $\theta=...2_21_22_21_42_21_42^*21_22_21_42_21_2...$, , because $v_3$ is $2$-prohibited, 

\item[$\bullet$] $\theta=...2_41_22_21_42^*21_22_21_22_21_2...$ or $\theta=...2_41_22_21_42^*21_22_21_22_3...$,
\item[$\bullet$] $\theta=...1_42_21_22_21_42^*21_22_21_22_21_2...$ or $\theta=...1_42_21_22_21_42^*21_22_21_22_3...$,
\item[$\bullet$] $\theta=...2_21_22_21_22_21_42^*21_22_21_22_21_2...$ or $\theta=...2_21_22_21_22_21_42^*21_22_21_22_3...$.
\end{description}
\begin{lemma} 
\begin{enumerate} \label{l7.2}
\item[(i)]$\lambda^+_0(2_21_22_21_62^*21_42_21_22_21_2)<3.000000066$.
\item[(ii)]$\lambda^+_0(2_21_22_21_42_21_42^*21_22_21_42_21_2)<\lambda^+_0(1_42_21_42_21_42^*21_22_21_42_21_2)<3.000000019$.
\item[(iii)]$\lambda^+_0(2_41_22_21_42^*21_22_21_22_21_2)<3$. 
\item[(iv)]$\lambda^-_0(1_42_21_22_21_42^*21_22_21_22_3)>\lambda^-_0(2_21_22_21_22_21_42^*21_22_21_22_3)>3.00000051$.
\item[(v)]$\lambda^+_0(2_21_22_21_22_21_42^*21_22_21_22_21_2)<\lambda^+_0(1_42_21_22_21_42^*21_22_21_22_21_2)<3.000000129$.
\end{enumerate}
\end{lemma}

By Lemma \ref{l7.2} and Remark \ref{r.2}, if $\theta$ in the branch $(a)$ is $(4,3.00000051)$-admissible, then
\begin{description}
\item[$\bullet$]  $\theta=...1_92^*21_9...$ or $\theta=...1_92^*21_82_21_2...$,
\item[$\bullet$] $\theta=...1_22_21_82^*21_62_21_4...$ or $\theta=...1_22_21_82^*21_62_21_22_2...$,

\item[$\bullet$] $\theta=...1_42_21_62^*21_42_21_6...$ or $\theta=...1_42_21_62^*21_42_21_42_2...$,  because $v_5$ is prohibited, 
\item[$\bullet$] $\theta=...2_21_22_21_62^*21_42_21_22_3...$,

\item[$\bullet$] $\theta=...2_41_22_21_42^*21_22_21_22_3...$.
\end{description}
\begin{lemma}
\begin{enumerate}\label{l7.3}
\item[(i)] $\lambda^+_0(2_21_22_21_82^*21_62_21_4)<\lambda^+_0(1_42_21_82^*21_62_21_4)<3.000000118$. 
\item[(ii)] $\lambda^-_0(1_42_21_82^*21_62_21_22_2)>3.00000035$.
\item[(iii)] $\lambda^+_0(2_21_22_21_82^*21_62_21_22_2)<3.000000025$. 
\item[(iv)] $\lambda^+_0(2_21_42_21_62^*21_42_21_6)<\lambda^+_0(1_62_21_62^*21_42_21_6)<3.000000118$. 
\item[(v)] $\lambda^-_0(1_62_21_62^*21_42_21_42_2)>3.00000035$.
\item[(vi)] $\lambda^+_0(2_21_42_21_62^*21_42_21_42_2)<3.000000025$. 
\item[(vii)] $\lambda^+_0(2_31_22_21_62^*21_42_21_22_3)<\lambda^+_0(1_22_21_22_21_62^*21_42_21_22_3)<3.000000126$.
\end{enumerate}
\end{lemma}

By Lemmas \ref{l5.10} and \ref{l7.2}, Remark \ref{r.2} and since that $v_5$ is prohibited, if $\theta$ in the branch $(a)$ is $(4,3.00000035)$-admissible, then
\begin{description}
\item[$\bullet$]  $\theta=...1_{10}2^*21_9...$,
\item[$\bullet$]  $\theta=...1_{10}2^*21_82_21_2...$,

\item[$\bullet$] $\theta=...1_22_41_22_21_42^*21_22_21_22_3...$ or $\theta=...2_51_22_21_42^*21_22_21_22_3...$.
\end{description}
\begin{lemma} \label{l7.4}
\begin{enumerate}
\item[(i)] $\lambda^+_0(1_22_41_22_21_42^*21_22_21_22_4)<\lambda^+_0(1_22_41_22_21_42^*21_22_21_22_31_2)<3.000000139$. 
\item[(ii)] $\lambda^+_0(2_51_22_21_42^*21_22_21_22_4)<\lambda^+_0(2_51_22_21_42^*21_22_21_22_31_2)<3.00000012$.
\end{enumerate}
\end{lemma}

By Lemmas \ref{l5.10} and \ref{l7.4} (and Remark \ref{r.2}),  if $\theta$ in the branch $(a)$ is $(4,3.00000035)$-admissible, then
\begin{description}
\item[$\bullet$]  $\theta=...1_{10}2^*21_{10}...$,
\item[$\bullet$]  $\theta=...1_{10}2^*21_82_21_4...$, or $\theta=...1_{10}2^*21_82_21_22_2...,$
\end{description}

By Lemma \ref{l5.11}, if $\theta$ in the branch $(a)$ is $(4,3.00000035)$-admissible, then $\theta=...1_{11}2^*21_{10}...$, and by Remark \ref{r.2}, we must extend as $\theta=...1_{11}2^*21_{11}...$ or $\theta=...1_{11}2^*21_{10}2_21_2...$. 

\begin{lemma} \label{l7.5}
$\lambda^-_0(1_22_21_{11}2^*21_{11})>3.0000044$.
\end{lemma}

By Lemmas \ref{l5.12} and  \ref{l7.5}, if $\theta$ in the branch $(a)$ is $(4,3.00000035)$-admissible, then
\begin{description}
\item[$\bullet$] $\theta=...1_{12}2^*21_{11}...$ or $\theta=...1_{12}2^*21_{10}2_21_2...$.
\end{description}

By Remark \ref{r.2}, the possible continuation of these words on the right hand side are
\begin{description}
\item[$\bullet$] $\theta=...1_{12}2^*21_{12}...$ or $\theta=...1_{12}2^*21_{11}2_21_2...$,
\item[$\bullet$] $\theta=...1_{12}2^*21_{10}2_21_4...$ or $\theta=...1_{12}2^*21_{10}2_21_22_2...$.
\end{description}
\begin{lemma} \label{l7.6}
\begin{enumerate}
\item[(i)] $\lambda^+_0(1_22_21_{12}2^*21_{11})<3$ .
\item[(ii)] $\lambda^+_0(1_22_21_{12}2^*21_{10}2_21_2)<3.0000000171$.
\item[(iii)] $\lambda^-_0(1_{13}2^*21_{10}2_21_2)>3.00000169$.
\end{enumerate}
\end{lemma}
By  Lemma \ref{l7.6},  if $\theta$ in the branch $(a)$ is $(4,3.00000035)$-admissible, then
\begin{description}
\item[$\bullet$] $\theta=...1_{13}2^*21_{11}2_21_2...$ or $\theta=...1_{13}2^*21_{12}...$.
\end{description}

\begin{lemma} \label{l7.7}
\begin{enumerate}
\item[(i)] $\lambda^+_0(1_{14}2^*21_{11}2_21_2)<3$. 
\item[(ii)] $\lambda^+_0(1_22_21_{13}2^*21_{11}2_21_2)<3.0000000066$.
\item[(iii)] $\lambda^-_0(1_22_21_{13}2^*21_{12})>3.00000064$.
\end{enumerate}
\end{lemma}
By  Lemma \ref{l7.7}, if $\theta$ in the branch $(a)$ is $(4,3.00000035)$-admissible, then $\theta=...1_{14}2^*21_{12}...$. And by Remark \ref{r.2}, this word must extend as $\theta=...1_{14}2^*21_{13}...$ or $\theta=...1_{14}2^*21_{12}2_21_2...$. 
\begin{lemma} \label{l7.8}
\begin{enumerate}
\item[(i)] $\lambda^+_0(1_22_21_{14}2^*21_{13})<3$. 
\item[(ii)] $\lambda^-_0(1_{15}2^*21_{12}2_21_2)>3.00000024$.
\item[(iii)] $\lambda^+_0(1_22_21_{14}2^*21_{12}2_21_2)<3.0000000025$.
\end{enumerate}
\end{lemma}
By  Lemma \ref{l7.8}, if $\theta$ in the branch $(a)$ is $(4,3.00000024)$-admissible, then 
$$\theta=...1_{15}2^*21_{13}....$$ 
\begin{lemma} \label{l7.9}
 $\lambda^+_0(1_{15}2^*21_{13}2_21_2)<3.000000037$ 
\end{lemma}
By  Lemma \ref{l7.9},if $\theta$ in the branch $(a)$ is $(4,3.00000024)$-admissible, then $$\theta=...1_{15}2^*21_{14}....$$
\begin{lemma} \label{l7.10}
 $\lambda^+_0(1_{16}2^*21_{14})<3.000000081$ 
\end{lemma}
By  Lemma \ref{l7.10}, if $\theta$ in the branch $(a)$ is $(4,3.00000024)$-admissible, then $$\theta=...1_22_21_{15}2^*21_{14}....$$
\begin{lemma} \label{l7.11}
\begin{enumerate}
\item[(i)] $\lambda^+_0(1_22_21_{15}2^*21_{15})<3.000000127$ 
\item[(ii)] $\lambda^-_0(1_22_21_{15}2^*21_{14}2_21_2)>3.000000161$
\end{enumerate}
\end{lemma}
By Lemma \ref{l7.11}, there is no $\theta$ in the branch $(a)$ which is  $(4,3.000000161)$-admissible.

Second, we study the possible continuations of $x$ with central combinatorics in the branch $(b)$ from Corollary \ref{c.2}. By previous sections, after the Lemma \ref{l5.13}, if $\theta$ in the branch $(b)$ is $(4,3.0001)$-admissible, then
	\begin{description}
    \item[$\bullet$]$\theta=...2_21_22_21_22^*2_{3}1_22_2...$ or		
    \item[$\bullet$] $\theta=...2_41_22^*2_{4}...$.
    \end{description}
By Lemma \ref{l5.14}(i) and Remark \ref{r.2}, if $\theta$ in the branch $(b)$ is $(4,3.0001)$-admissible, then
\begin{enumerate}
\item[$\bullet$] $\theta=...2_21_22_21_22^*2_{3}1_22_21_2...$ or $\theta=...2_21_22_21_22^*2_{3}1_22_4...$,
\item[$\bullet$]$\theta=...2_41_22^*2_{5}...$.
\end{enumerate}
\begin{lemma}\label{l7.12}
\begin{enumerate}
\item[(i)]$\lambda_0^+(2_41_22_21_22^*2_31_22_21_2)<3$.
\item[(ii)]$\lambda_0^-(1_22_21_22_21_22^*2_31_22_4)>3.000003$.
\end{enumerate}
\end{lemma}

By Lemma \ref{l7.12} and Remark \ref{r.2}, if $\theta$ in the branch $(b)$ is $(4,3.000003)$-admissible, then
\begin{enumerate}
\item[$\bullet$] $\theta=...1_22_21_22_21_22^*2_{3}1_22_21_2...$ or $\theta=...2_41_22_21_22^*2_{3}1_22_4...$,
\item[$\bullet$]$\theta=...2_21_22_41_22^*2_{5}...$ or $\theta=...2_51_22^*2_{5}...$.
\end{enumerate}

\begin{lemma}\label{l7.13}
$\lambda_0^+(2_41_22_21_22^*2_31_22_41_2)<3.000000088$.
\end{lemma}

By Lemmas \ref{l7.13}, \ref{l5.15}$(ii)$-$(iii)$ and Remark \ref{r.2}, if $\theta$ in the branch $(b)$ is $(4,3.000003)$-admissible, then
\begin{enumerate}
\item[$\bullet$] $\theta=...1_22_21_22_21_22^*2_{3}1_22_21_4...$ or $\theta=...1_22_21_22_21_22^*2_{3}1_22_21_22_2...$, \item[$\bullet$] $\theta=...2_41_22_21_22^*2_{3}1_22_5...$,
\item[$\bullet$]$\theta=...2_21_22_41_22^*2_{5}1_22_2...$ or $\theta=...2_51_22^*2_{6}...$.
\end{enumerate}
\begin{lemma}\label{l7.14}
\begin{itemize}
\item[(i)]$\lambda_0^+(2_21_22_21_22_21_22^*2_{3}1_22_21_4)<3$.
\item[(ii)]$\lambda_0^-(1_42_21_22_21_22^*2_{3}1_22_21_22_2)>3.00000049$.
\item[(iii)]$\lambda_0^+(2_21_22_21_22_21_22^*2_{3}1_22_21_22_2)<3.000000034$.
\item[(iv)]$\lambda_0^+(2_21_22_41_22_21_22^*2_{3}1_22_5)<3.000000142$.
\item[(v)]$\lambda_0^+(2_31_22_41_22^*2_{5}1_22_2)<3.00000004$.
\end{itemize}
\end{lemma}

By Lemmas \ref{l7.14} and \ref{l.5-2} (and Remark \ref{r.2}), if $\theta$ in the branch $(b)$ is $(4,3.00000049)$-admissible, then
\begin{enumerate}
\item[$\bullet$] $\theta=...1_42_21_22_21_22^*2_{3}1_22_21_4...$,
\item[$\bullet$]$\theta=...2_31_22_41_22^*2_{5}1_22_2...$,
\item[$\bullet$] $\theta=...2_21_22_51_22^*2_{6}...$ or $\theta=...2_61_22^*2_{6}...$.
\end{enumerate}

\begin{lemma}\label{l7.15}
\begin{itemize}
\item[(i)]$\lambda_0^+(1_42_21_22_21_22^*2_{3}1_22_21_5)<3.000000063$.
\item[(ii)]$\lambda_0^+(1_42_21_22_21_22^*2_{3}1_22_21_42_21_2)<3.000000138$.
\item[(iii)]$\lambda_0^+(1_22_21_22_41_22^*2_{5}1_22_4)<3.000000137$.
\item[(iv)]$\lambda_0^+(2_21_22_51_22^*2_{7})<\lambda_0^+(2_21_22_51_22^*2_{6}1_22_2)<3.00000004$.
\item[(v)]$\lambda_0^-(2_61_22^*2_{6}1_22_2)>3.000003$.
\end{itemize}
\end{lemma}

By Lemmas \ref{l7.15}, \ref{l.5-1} and \ref{l.5-2}(and Remark \ref{r.2}), if $\theta$ in the branch $(b)$ is $(4,3.00000049)$-admissible, then $\theta=...2_61_22^*2_7...$. And by Remark \ref{r.2} this word must extend as $\theta=...2_21_22_61_22^*2_{7}...$ or $\theta=...2_71_22^*2_{7}...$.
\begin{lemma} \label{l7.16}
\begin{enumerate}
\item[(i)]$\lambda_0^+(2_21_22_61_22^*2_{7}1_22_2)<3.00000007$.
\item[(ii)]$\lambda_0^-(2_21_22_61_22^*2_{8})>3.0000006$.
\item[(iii)]$\lambda_0^+(2_71_22^*2_{7}1_22_2)<3.$
\end{enumerate}
\end{lemma}

By Lemmas \ref{l7.16} and Remark \ref{r.2}, if $\theta$ in the branch $(b)$ is $(4,3.00000049)$-admissible, then
$$\theta=...2_71_22^*2_{8}....$$
\begin{lemma} \label{l7.17}
$\lambda_0^+(2_21_22_71_22^*2_{8})<3.000000094$.
\end{lemma}

By Lemmas \ref{l7.17} and Remark \ref{r.2}, if $\theta$ in the branch $(b)$ is $(4,3.00000049)$-admissible, then
$$\theta=...2_81_22^*2_{8}....$$
\begin{lemma} \label{l7.18}
$\lambda_0^+(2_81_22^*2_{9})<3.00000005$.
\end{lemma}

By Lemmas \ref{l7.18} and Remark \ref{r.2}, if $\theta$ in the branch $(b)$ is $(4,3.00000049)$-admissible, then
$$\theta=...2_81_22^*2_{8}1_22_2....$$
\begin{lemma} \label{l7.19}
$\lambda_0^+(2_91_22^*2_{8}1_22_2)<3.00000013$.
\end{lemma}

By Lemmas \ref{l7.19} and Remark \ref{r.2}, if $\theta$ in the branch $(b)$ is $(4,3.00000049)$-admissible, then
$$\theta=...2_21_22_81_22^*2_{8}1_22_2....$$

Thus, in summary this discussion over the two branches $(a)$ and $(b)$, from Corollary \ref{c.2}, give to us that if $\theta$ is $(4,3.000000161)$-admissible, then
$$\theta=...2_21_22_81_22^*2_{8}1_22_2....$$

Finally, we follow the script in the Subsection \ref{s4}, which is condensed in Lemma \ref{gr}. Let 
$$\tilde{\lambda}_4:=\lambda_0^-(2_21_22_81_22_81_22^*2_81_22_{10})>3.000000142308464>m(\gamma_k^1)$$
be as in this lemma. Thus, we get the desired local uniqueness result for $\gamma_4^1$:

\begin{lemma}[Local uniqueness of $\gamma_4^1$]\label{l.local-m4} A $(4, \tilde{\lambda}_4)$-admissible word $\theta$ has the form  
$$\theta=\dots 2_{8}1_22_{10}1_22_{8}1_22^*2_{8}1_22_{10}1_22_{8} 1_2 2_2\dots$$
In particular, it contains the string $\theta_4^0=2_{8}1_22_{10}1_22_{8}1_22^*2_{8}1_22_{10}1_22_{8}1$.
\end{lemma}

\section{Proof of Theorem \ref{t.A}}\label{s.tA}

The fact that $m_k=m(\gamma_k^1)$ is a decreasing sequence converging to $3$ is an immediate consequence of Lemmas \ref{l.2} and \ref{l.3}. 

Next, let us show that $m_j\in M\setminus L$ for each $j\in\{1,2,3,4\}$. For this sake, assume that $m_j\in L$ for some $1\leq j\leq 4$: this would mean that $m_j$ is the limit of the Markov values $m(\theta_n)$ of certain periodic words $\theta_n\in\{1,2\}^{\mathbb{Z}}$. By combining the local uniqueness for $\gamma_j^1$, i.e., Lemma \ref{l.local-m1}, \ref{l.local-m2}, \ref{l.local-m3}, \ref{l.local-m4} resp. when $j=1$, $2$, $3$, $4$ resp., with the replication property in Lemma \ref{l.replication}, we get that $\theta_n = \theta(\underline{\omega}_j)$ for all $n$ sufficiently large. Therefore, $m_j = m(\gamma_j^1) = \lim\limits_{n\to\infty}m(\theta_n) = m(\theta(\underline{\omega}_j))$, a contradiction.  

Finally, the quantities $m_j$, $j\in\{1,2,3,4\}$, belong to distinct connected components of $L$ because for any $k\in\mathbb{N}$ one has that $m(\theta(\underline{\omega}_k))\in L$ and Lemma \ref{l.3} ensures that 
$$m(\theta(\underline{\omega}_k))<m_k<m(\theta(\underline{\omega}_{k-1}))$$

\section{Local almost uniqueness for $\gamma_k^1$}\label{s.almost-uniqueness}

 We know from Corollary \ref{c.2} that any $(k,3.009)$-admissible word $\theta$ has the form $\theta=\dots 1_42^*21_2\dots$ or $\dots 2_21_22^*2_2\dots$ (up to transposition). 
 
 In this section, we will establish the following local almost uniqueness property for $\gamma_k^1$:
 \begin{itemize} 
 \item any $(k,\lambda)$-admissible word $\theta=\dots 1_4 2^*21_2\dots$ can not be extended as $\dots 1_s2^*21_t\dots$ with $t, s\gg k$ because the string $1_{2j}2^*21_{2m}$ is $k$-avoided when $\lfloor (2k-1)\log 3/\log 2\rfloor+3<j\leq m+1$ (cf. Lemma \ref{l.9-1} below);  
 \item there exists an explicit constant $\mu_k>m(\gamma_k^1)$ such that any $(k,\mu_k)$-admissible word $\theta=\dots 2_21_22^*2_2\dots$ has the form 
 \begin{itemize}
 \item $\theta=\dots 2_{2k}1_22_{2k+2}1_22_{2k}1_22^*2_{2k}1_22_{2k+2}1_22_{2k}1\dots$ or 
 \item $\theta=\dots 1_22_{2m}1_22^*2_{2m+1}1_22_2\dots$ with $m<k$ or 
 \item $\theta=\dots 2_21_22_{2m-1}1_22^*2_{2m}1_22_2\dots$ with $1<m<k-1$.
 \end{itemize} 
(cf. Lemma \ref{l.semi-local-main} below). 
 \end{itemize} 
 
 \begin{remark}
In view of the statements above, the local uniqueness property for $\gamma_k^1$ is equivalent to the  existence of $\nu_k>m(\gamma_k^1)$ such that no $(k,\nu_k)$-admissible word has the form 
\begin{itemize}
\item $\dots 1_4 2^*2 1_2\dots$ or 
\item $\dots 1_22_{2m}1_22^*2_{2m+1}1_22_2\dots$ with $m<k$ or 
\item $\dots 2_21_22_{2m-1}1_22^*2_{2m}1_22_2\dots$ with $1<m<k-1$ 
\end{itemize}  
\end{remark} 
 
 \subsection{The string  $1_s2^*21_t$.} Let us try to compare $\lambda^+_0(u)$ and $m(\theta(\underline{\omega}_k))$ with $3$.



Note that $\dfrac{p(1_n)}{q(1_n)}=[0;1_n]=\dfrac{f_n}{f_{n+1}}$ (and then $q(1_n)=f_{n+1}$), where $f_n$ is Fibonacci's sequence. 
\begin{lemma}\label{l.9-1}
Suppose that $\lfloor (2k-1)\log 3/\log 2\rfloor+3<j\leq m+1$. Then, $$\lambda^+_0(1_{2j}2^*21_{2m})<m(\theta(\underline{\omega}_k)).$$
\end{lemma}
\begin{proof}
$\lambda^+_0(u)=[2;1_2,1_{2j-1},\overline{2,1}]+[0;2,1_{2m},\overline{2,1}]$. Let $\alpha=[1;1_{2j-2},\overline{2,1}]$. Then
$\lambda^+_0(u)=[2;1_2,\alpha]+[0;2,1_{2m-2j},\alpha]$. We know that
$$3=[2;1_2,1_{2m-2j},\alpha]+[0;2,1_{2m-2j},\alpha]=[2;1_2,\gamma]+[0;2,\gamma],$$
where $\gamma=1_{2m-2j}\alpha$.

Then, $\lambda^+_0(u)-3=[0;1_2,\alpha]-[0;1_2,\gamma]=[0;1_{2j+1},\overline{2,1}]-[0;1_{2m-1},\overline{2,1}].$
Therefore,
$$\lambda^+_0(u)-3=\dfrac{[2;\overline{1,2}]-[1;1_{2m-2j-3},\overline{2,1}]}{f_{2j+2}^2([2;\overline{1,2}]+\beta_{2j+2})([1;1_{2m-2j-3},\overline{2,1}]+\beta_{2j+2})}.$$
Moreover,
$m(\theta(\omega_k))\ge [2;1_2,2_{2k},\overline{2,1}]+[0;2,2_{2k-1},1_2,\overline{2,1}]$. But
$$3=[2;1_2,2_{2k},\overline{2,1}]+[0;2,2_{2k},\overline{2,1}].$$
Therefore,
$$m(\theta(\omega_k))-3\ge [0;2,2_{2k-1},1_2,\overline{2,1}]-[0;2,2_{2k},\overline{2,1}]=[0;2_{2k},1_2,\overline{2,1}]-[0;2_{2k+1},\overline{2,1}].$$
That is,
$$m(\theta(\omega_k))-3\ge \dfrac{[2,\overline{2,1}]-[1;\overline{1,2}]}{q^2_{2k}([2,\overline{2,1}]+\beta^{(2)}_{2k})([1;\overline{1,2}]+\beta^{(2)}_{2k}))}.$$
Then, we need to find $j$ such that $f_{2j+2}>2q_{2k}$.
Of course that $f_{2n+2}>2f_{2n}$ and $q_{m}<3q_{m-1}$. Then, $f_{2j+2}>2^{j}$
 and $q_{2k}<2\cdot 3^{2k-1}$. If $j>\lfloor (2k-1)\log 3/\log 2\rfloor+3$ we will have $f_{2j+2}>2q_{2k}$. 
 \end{proof}
 
  \subsection{Extensions of $2_21_22^*2_2$}
Consider the word $\theta=\dots 2_{2j}1_22^*2_{2m}\dots$ for some $j,m\in \mathbb{N}$. 

\begin{lemma}\label{l.AL0}
If $s, t>2k$, then $\lambda^+_0(2_{s}1_22^*2_{t})<m(\theta(\underline{\omega}_k))$. 
\end{lemma}
\begin{proof}
Since $[2;2_{t},...]<[2;2_{2k},1,...]$ and $[0;1_2,2_{s},...]<[0;1_2,2_{2k},1,...]$, we have that
$\lambda^+_0(2_{s}1_22^*2_{t})<m(\theta(\underline{\omega}_k))$. 
\end{proof}

This lemma says that any $(k,3.009)$-admissible word of the form $\theta =...2_21_22^*2_2...$ extends as 
\begin{itemize}
\item[(A)$_{a,b}$] $\theta=\dots1_2 2_a 1_2 2^* 2_b1_22_2\dots$ with $2\leq a, b<2k+1$ or 
\item[(B)$_a$] $\theta=\dots1_2 2_a 1_2 2^* 2_{2k+1}\dots$ with $2\leq a < 2k+1$ or 
\item[(C)$_b$] $\theta=\dots2_{2k+1} 1_2 2^* 2_{b} 1_2 2_2\dots$ with $2\leq b < 2k+1$.
\end{itemize}

\subsection{Ruling out (B)$_{a}$ with $a$ odd} This situation never occurs: 

\begin{lemma}\label{l.Baodd} If $j<k$, then 
$\lambda^+_0(1_22_{2j+1}1_22^*2_{2k+1})<m(\theta(\underline{\omega}_k)).$
\end{lemma}
\begin{proof}
$\lambda^+_0(1_22_{2j+1}1_22^*2_{2k+1})=[2;2_{2k+1},...]+[0;1_2,2_{2j+1},1...]<[2;2_{2k},1,...]+[0;1_2,2_{2k},1...].$ 
\end{proof}

\subsection{Ruling out (B)$_{a}$ with $a$ even} This case never occurs. Indeed, by Lemma \ref{l.16}, a word $\theta=\dots1_2 2_{2j} 1_2 2^* 2_{2k+1}\dots$ with $j < k$ is not $(k,\lambda_k^{(1)})$-admissible. Moreover, a word $\theta=\dots1_2 2_{2j} 1_2 2^* 2_{2k+1}\dots$ with $j=k$ is also not $(k,\lambda_k^{(1)})$-admissible: 

\begin{lemma}\label{l.AL3}
If $m<k$, then $\lambda^+_0(1_22_{2k}1_22^*2_{2k+1})<\lambda^+_0(1_22_{2k}1_22^*2_{2m+1})<m(\theta(\underline{\omega}_k))$.
\end{lemma}
\begin{proof}
In fact, as before, $\lambda^+_0(1_22_{2k}1_22^*2_{2m+1})=[2;2_{2m+1},\overline{2,1}]+[0;1_2,2_{2k},1_2,\overline{1,2}]:=A_k+B_k$. Moreover, $m(\theta(\underline{\omega}_k))>[2;2_{2k},\overline{2,1}]+[0;1_2,2_{2k},1_2,\overline{2,1}]:=C_k+D_k.$
This implies that
	$$C_k-A_k=\dfrac{[2;\overline{1,2}]-[2;2_{2k-2m-2},\overline{2,1}]}{q^2_{2m+1}([2;2_{2k-2m-2},\overline{2,1}]+\beta_{2m+1})([2;\overline{1,2}]+\beta_{2m+1})}.$$
	Moreover,
	$$B_k-D_k=\dfrac{[1;\overline{1,2}]-[1;1,\overline{1,2}]}{\tilde{q}^2_{2k+2}([1;\overline{1,2}]+\tilde{\beta}_{2k+2})([1;1,\overline{1,2}]+\tilde{\beta}_{2k+2})}.$$
	Then,
	$$\dfrac{B_k-D_k}{C_k-A_k}=\dfrac{q^2_{2k+1}}{q^2_{2m+1}}\cdot X \cdot Y$$
	where $$X=\dfrac{[2;\overline{1,2}]-[2;2_{2k-2m-2},\overline{2,1}]}{[1;\overline{1,2}]-[1;1,\overline{1,2}]}>1$$
	and
	$$Y=\dfrac{([2;2_{2k-2m-2},\overline{2,1}]+\beta_{2m+1})([2;\overline{1,2}]+\beta_{2m+1})}{([1;\overline{1,2}]+\tilde{\beta}_{2k+2})([1;1,\overline{1,2}]+\tilde{\beta}_{2k+2})}.$$
	Since $m\le k-1$ we have $\dfrac{q_{2k+1}}{q_{2m+1}}\ge 5+2\beta_{2m+1}.$ Then, by Lemma \ref{Le1}
		$$\dfrac{C_k-A_k}{B_k-D_k}=25\cdot 1\cdot \dfrac{1}{4}>1.$$
		Therefore, we have that
		$$\lambda^+_0(1_22_{2k}1_22^*2_{2m+1})=A_k+B_k<C_k+D_k<m(\theta(\underline{\omega}_k)).$$
		Finally, since $[2;2_{2k+2},\overline{1,2}]<[2;2_{2m+2},\overline{1,2}]$ when $m<k$ we have
		$$\lambda^+_0(1_22_{2k}1_22^*2_{2k+1})<\lambda^+_0(1_22_{2k}1_22^*2_{2m+1})$$
\end{proof}

\subsection{Ruling out (C)$_{b}$ with $b$ odd} This situation never occurs: 

\begin{lemma}\label{l.Cbodd} If $m<k$, then 
$\lambda^+_0(2_{2k+1}1_22^*2_{2m+1}1_2)<m(\theta(\underline{\omega}_k)).$
\end{lemma}
\begin{proof}
$\lambda^+_0(2_{2k+1}1_22^*2_{2m+1}1_2)=[2;2_{2m+1},1...]+[0;1_2,2_{2k+1},...]<[2;2_{2k},1,...]+[0;1_2,2_{2k},1...].$ 
\end{proof}

\subsection{Ruling out (C)$_{b}$ with $b$ even} 

This case never occurs. Indeed, by Lemma \ref{l.10}, a word $\theta=\dots2_{2k+1} 1_2 2^* 2_{2m} 1_2 2_2\dots$ with $m < k$ is not $(k,\lambda_k^{(1)})$-admissible. Moreover, a word $\theta=\dots2_{2k+1} 1_2 2^* 2_{2m} 1_2 2_2\dots$ with $m = k$ is also not $(k,\lambda_k^{(1)})$-admissible:

\begin{lemma}\label{l.AL5}
If $j<k-1$, then $\lambda^+_0(2_{2k+1}1_22^*2_{2k}1_2) < \lambda^+_0(1_22_{2j+1}1_22^*2_{2k}1_2)<m(\theta(\underline{\omega}_k))$.
\end{lemma}
\begin{proof}
$\lambda^+_0(1_22_{2j+1}1_22^*2_{2k}1_2)=[0;2_{2k},1_2,\overline{1,2}]+[2,1_2,2_{2j+1},1_2,\overline{2,1}]=[0;2,\beta]+[2;1_2,\alpha]$ where
$$\alpha=[2;2_{2j},1_2,\overline{2,1}] \quad \mbox{and} \quad \beta=[2;2_{2k-2},1_2,\overline{1,2}].$$
If $j<k-1$, then $\beta<\alpha$. By Lemma 3 in Chapter 1 of Cusick--Flahive book \cite{CF}, it follows that $\lambda^+_0(1_22_{2j+1}1_22^*2_{2k}1_2)<3$. 
\end{proof}

\subsection{Ruling out (A)$_{a,b}$ with $a,b$ odd} This situation never occurs: 

\begin{lemma}\label{l.AL7} If $j,m<k$, then 
$\lambda^+_0(1_22_{2j+1}1_22^*2_{2m+1}1_2)<m(\theta({\omega}_k)).$
\end{lemma}
\begin{proof}
$\lambda^+_0(1_22_{2j+1}1_22^*2_{2m+1}1_2)=[2;2_{2m+1},1...]+[0;1_2,2_{2j+1},1...]<[2;2_{2k},1,...]+[0;1_2,2_{2k},1...].$ 
\end{proof}

\subsection{Ruling out (A)$_{a,b}$ with $a, b$ even, $a<2k$}

This case never happens: Lemma \ref{l.7} implies that $\theta=\dots 1_2 2_{2j} 1_2 2^* 2_{2m}1_2 2_2\dots$ is not $(k,\lambda_k^{(1)})$-admissible when $j<k$ and $m\leq k$. 

\subsection{Ruling out (A)$_{2k,b}$ with $b<2k$ even} This situation never occurs. Indeed, by Lemma \ref{l.10}, a word $\theta=\dots2_{2k} 1_2 2^* 2_{2m} 1_2 2_2\dots$ with $m < k$ is not $(k,\lambda_k^{(2)})$-admissible.

\subsection{The case (A)$_{2k,2k}$}

This case corresponds to a word 
$$\theta=\dots1_2 2_{2k} 1_2 2^* 2_{2k} 1_2 2_2\dots$$

\subsection{The case (A)$_{a,b}$ with $a$ odd, $b$ even}

This situation can not occur except possibly when $b=a+1<2k-2$. Indeed, let us establish this fact by analysing the subcases $1<a<2k-1$ and $a=2k-1$. Remember that $121$ is $k$-prohibited.

Note that Lemma \ref{l.AL5} implies that a $(k,\lambda_k^{(1)})$-admissible word 
$$\theta=\dots2_21_2 2_{a} 1_2 2^* 2_{b} 1_2 2_2\dots$$ 
with $a<2k-1$ odd and $b$ even satisfies $b<2k$. 

\begin{lemma}\label{l.AL4}
We have:
\begin{enumerate}
\item[(i)] If $k\geq j+1>m$ then $\lambda^-_0(2_{2j+1}1_22^*2_{2m}1_2)>m(\gamma^1_k).$
 
\item[(ii)] If $j+1<m<k$ then $\lambda^+_0(1_2 2_{2j+1}1_22^*2_{2m}1_2)<m(\theta(\underline{\omega}_k)).$
\end{enumerate}
\end{lemma}
\begin{proof} 
To prove (i) write $\lambda^-_0(2_{2j+1}1_22^*2_{2m}1_2)=[2;2_{2m},1_2,\overline{2,1}]+[0;1_2,2_{2j+1},\overline{1,2}]:=A_k+B_k$ and $m(\gamma^1_k)<[2;2_{2k},1_2,\overline{1,2}]+[0;1_2,2_{2k},\overline{1,2}]:=C_k+D_k.$ Then,
	$$A_k-C_k=\dfrac{[2;2_{2k-2m-1},1_2,\overline{2,1}]-[1;\overline{2,1}]}{q^2_{2m}([2;2_{2k-2m-1},1_2,\overline{2,1}]+\beta_{2m})([1;\overline{2,1}]+\beta_{2m})}$$
	while
	$$D_k-B_k=\dfrac{[2;2_{2k-2j-2},\overline{1,2}]-[1;\overline{2,1}]}{\tilde{q}_{2j+3}([2;2_{2k-2j-2},\overline{1,2}]+\tilde{\beta}_{2j+3})([1;\overline{2,1}]+\tilde{\beta}_{2j+3})}.$$
	Thus,
	$$\dfrac{A_k-C_k}{D_k-B_k}=\dfrac{q^2_{2j+2}}{q^2_{2m}}\cdot X\cdot Y$$
	where
	$$X=\dfrac{[2;2_{2k-2m-1},1_2,\overline{2,1}]-[1;\overline{2,1}]}{[2;2_{2k-2j-2},\overline{1,2}]-[1;\overline{2,1}]}$$
	and
	$$Y=\dfrac{([2;2_{2k-2j-2},\overline{1,2}]+\tilde{\beta}_{2j+3})([1;\overline{2,1}]+\tilde{\beta}_{2j+3})}{([2;2_{2k-2m-1},1_2,\overline{2,1}]+\beta_{2m})([1;\overline{2,1}]+\beta_{2m})}.$$
	Using Lemma \ref{Le1} we have
	$$\dfrac{A_k-C_k}{D_k-B_k}>25 \cdot \dfrac{1}{16}>1,$$
	because $j+1>m$ implies $q_{2j+2}\ge q_{2m+2}=5q_{2m}+q_{2m-1}$, that is, $\dfrac{q_{2j+2}}{q_{2m}}>5+2\beta_{2m}.$
	
	To prove (ii) note that if $j+1<m$ then writing $\alpha=[2;2_{2j},1_2,\overline{2,1}]$ and $\beta=[2;2_{2m-2},1_2,\overline{1,2}]$ we have that $\lambda^+_0(1_22_{2j+1}1_22^*2_{2m}1_2)=[2,1_2,\alpha]+[0;2,\beta].$ But, $\beta<\alpha$ and by Lemma 3 in \cite{CF} we have that 		$$\lambda^+_0(1_22_{2j+1}1_22^*2_{2m}1_2)<3.$$
	
	\end{proof}

Let $\mu_k^{(1)}:=\min\{\lambda_k^{(1)}, \lambda^-_0(2_{2j+1}1_22^*2_{2m}1_2): m<j+1\leq k\}$. By Lemma \ref{l.AL4}, a $(k,\lambda_k^{(1)})$-admissible word 
$$\theta=\dots2_21_2 2_{a} 1_2 2^* 2_{b} 1_2 2_2\dots$$ 
with $a<2k-1$ odd and $b$ even satisfies $b=a+1$. 

The next lemma allows to rule out the case $a=2k-3$: 

\begin{lemma}\label{l.AL6} If $k>2$ then
$\lambda^+_0(2_21_22_{2k-3}1_22^*2_{2k-2}1_22_2)<m(\theta(\underline{\omega}_k)).$
\end{lemma}
\begin{proof}
In this case 
	$$\lambda^+_0(2_21_22_{2k-3}1_22^*2_{2k-2}1_22_2)=[2;2_{2k-2},1_2,2_2,\overline{1,2}]+[0;1_2,2_{2k-3},1_2,2_2,\overline{2,1}]:=A_k+B_k$$
	and
	$$m(\theta(\underline{\omega}_k))>[2;2_{2k},1_2,2_2,\overline{2,1}]+[0;1_2,2_{2k},1_2,2_2,\overline{2,1}]:=C_k+D_k.$$
	This implies that 
	$$A_k-C_k=\dfrac{[2;2,1_2,2_2,\overline{2,1}]-[1;1,2_2,\overline{1,2}]}{q^2_{2k-2}([2;2,1_2,2_2,\overline{2,1}]+\beta_{2k-2})([1;1,2_2,\overline{1,2}]+\beta_{2k-2})}$$
	and
	$$D_k-B_k=\dfrac{[2;2_{2},1_2,2_2,\overline{2,1}]-[1;1,2_2,\overline{2,1}]}{q^2_{2k-2}([2;2_{2},1_2,2_2,\overline{2,1}]+\beta_{2k-2})([1;1,2_2,\overline{2,1}]+\beta_{2k-2}}.$$
Therefore
	$$\dfrac{D_k-B_k}{A_k-C_k}=X\cdot Y$$
where
	$$X=\dfrac{[2;2_{2},1_2,2_2,\overline{2,1}]-[1;1,2_2,\overline{2,1}]}{[2;2,1_2,2_2,\overline{2,1}]-[1;1,2_2,\overline{1,2}]}>1.03$$
and
	$$Y=\dfrac{([2;2,1_2,2_2,\overline{2,1}]+\beta_{2k-2})([1;1,2_2,\overline{1,2}]+\beta_{2k-2})}{([2;2_{2},1_2,2_2,\overline{2,1}]+\beta_{2k-2})([1;1,2_2,\overline{2,1}]+\beta_{2k-2})}>0.986$$
	because the minimum is attained when $\beta_{2k-2}=0.4$.
	Then,
	$$\dfrac{D_k-B_k}{A_k-C_k}>1.03\cdot 0.986>1.01.$$
	That is,
	$C_k+D_k>A_k+B_k$.
\end{proof}

So far, we showed that a $(k,\lambda_k^{(1)})$-admissible word 
$$\theta=\dots2_2 1_2 2_{a} 1_2 2^* 2_{b} 1_2 2_2\dots$$ 
with $a<2k-1$ odd and $b$ even satisfies $b=a+1<2k-2$.  

Closing our discussion of the case (A)$_{a,b}$ with $a$ odd, $b$ even, let us now show that the case $a=2k-1$ can not occur: 

\begin{lemma}\label{l.A2k-1} If $j=k-1$, then $\lambda^+_0(1_22_{2j+1}1_22^*2_{2k}1_2)<m(\theta(\underline{\omega}_k))$. Moreover, if $m<k$ then
$\lambda^-_0(1_22_{2k-1}1_22^*2_{2m}1_2)>m(\gamma^1_k)$.
\end{lemma}

\begin{proof} Note that
	$$\lambda^+_0(1_22_{2k-1}1_22^*2_{2k}1_2)=[2;2_{2k},1_2,\overline{1,2}]+[0;1_2,2_{2k-1},1_2,\overline{2,1}]:=A_k+B_k$$
while
	$$m(\theta(\underline{\omega}_k))>[2;2_{2k},1_2,\overline{1,2}]+[0;1_2,2_{2k},1_2,\overline{2,1}]=C_k+D_k.$$
We have
	$$A_k-C_k=\dfrac{[1;\overline{1,2}]-[1;1,\overline{1,2}]}{q^2_{2k}([1;\overline{1,2}]+\beta_{2k})([1;1,\overline{1,2}]+\beta_{2k})}$$
and
	$$D_k-B_k=\dfrac{[2;1_2,\overline{2,1}]-[1;\overline{1,2}]}{\tilde{q}^2_{2k+1}([2;1_2,\overline{2,1}]+\tilde{\beta}_{2k+1})([1;\overline{1,2}]+\tilde{\beta}_{2k+1})}.$$
Thus,
	$$\dfrac{D_k-B_k}{A_k-C_k}=X\cdot Y$$
where
	$$X=\dfrac{[2;1_2,\overline{2,1}]-[1;\overline{1,2}]}{[1;\overline{1,2}]-[1;1,\overline{1,2}]}>5.46$$
and	
	$$Y=\dfrac{([1;\overline{1,2}]+\beta_{2k})([1;1,\overline{1,2}]+\beta_{2k})}{([2;1_2,\overline{2,1}]+\tilde{\beta}_{2k+1})([1;\overline{1,2}]+\tilde{\beta}_{2k+1})}$$
By Lemma \ref{Le1} we have
	$$\dfrac{D_k-B_k}{A_k-C_k}>5.46\cdot \dfrac{1}{4}>1$$
and this implies that
	$$m(\theta(\underline{\omega}_k))>C_k+D_k>A_k+B_k=\lambda^+_0(1_22_{2k-1}1_22^*2_{2k}1_2).$$
	By Lemma \ref{l.AL4} we have that if $m<k$ then $\lambda^-_0(1_22_{2k-1}1_22^*2_{2m}1_2)>m(\gamma^1_k)$, because $k=j+1>m$.
	\end{proof}

In summary, we showed that 

\begin{corollary}
If $\theta=\dots2_2 1_2 2_{a} 1_2 2^* 2_{b} 1_2 2_2\dots$
 is $(k,\mu_k^{(1)})$-admissible word with $|a|$ odd and $b$ even, then $b=a+1$ and $b<2k-2$.
\end{corollary}

\subsection{The case (A)$_{a,b}$ with $a$ even, $b$ odd}

This case can not occur except possibly when $b= a+1<2k+1$. Indeed, by Lemma \ref{l.AL3}, a $(k,\lambda_1^{(k)})$-admissible word 
$$\theta = \dots 2_2 1_2 2_a 1_2 2^* 2_b 1_2 2_2\dots$$
with $a$ even and $b$ odd satisfies $a<2k$.  

\begin{lemma}\label{c.AL2}
We have:
\begin{enumerate}
\item[(i)] if $j<m<k$, then $\lambda^-_0(1_22_{2j}1_22^*2_{2m+1})>m(\gamma^1_k).$

\item[(ii)]  if $k>j>m$ then $\lambda^+_0(2_{2j}1_22^*2_{2m+1}1_2)<m(\theta(\underline{\omega}_k)).$
\end{enumerate}
\end{lemma}
\begin{proof}
To prove (i) let $\lambda^-_0(1_22_{2j}1_22^*2_{2m+1})=[2;2_{2m+1},\overline{1,2}]+[0;1_2,2_{2j},1_2,\overline{2,1}]:=A_k+B_k$. We know that $m(\gamma^1_k)<[2;2_{2k},\overline{1,2}]+[0;1_2,2_{2k},1_2,\overline{1,2}]:=C_k+D_k.$ Therefore,
	$$C_k-A_k=\dfrac{[2_{2k-2m-2},\overline{1,2}]-[1;\overline{2,1}]}{q^2_{2m+1}([2_{2k-2m-2},\overline{1,2}]+\beta_{2m+1})([1;\overline{2,1}]+\beta_{2m+1})}$$
	and
	$$B_k-D_k=\dfrac{[2;2_{2k-2j-1},1_2,\overline{1,2}]-[1;\overline{1,2}]}{\tilde{q}^2_{2j+2}([2;2_{2k-2j-1},1_2,\overline{1,2}]+\tilde{\beta}_{2j+2})([1;\overline{1,2}]+\tilde{\beta}_{2j+2})}.$$
	Thus,
	$$\dfrac{B_k-D_k}{C_k-A_k}=\dfrac{q^2_{2m+1}}{q^2_{2j+1}}\cdot X\cdot Y$$
	where
	$$X=\dfrac{[2;2_{2k-2j-1},1_2,\overline{1,2}]-[1;\overline{1,2}]}{[2_{2k-2m-2},\overline{1,2}]-[1;\overline{2,1}]}$$
	and
	$$Y=\dfrac{([2_{2k-2m-2},\overline{1,2}]+\beta_{2m+1})([1;\overline{2,1}]+\beta_{2m+1})}{([2;2_{2k-2j-1},1_2,\overline{1,2}]+\tilde{\beta}_{2j+2})([1;\overline{1,2}]+\tilde{\beta}_{2j+2})}.$$
	By Lemma \ref{Le1} it follows that
	$$\dfrac{B_k-D_k}{C_k-A_k}>25\cdot \dfrac{1}{16}>1$$
	because $m\ge j+1$ implies $q_{2m+1}\ge q_{2j+3}=2q_{2j+2}+q_{2j+1}=5q_{2j+1}+2q_{2j}$ and then
	$$\dfrac{q_{2m+1}}{q_{2j+1}}\ge 5+2\beta_{2j+1}>5.$$
If $j>m$ then $\alpha=[2;2_{2j-1},\overline{1,2}]>\beta=[2;2_{2m-1},1_2,\overline{2,1}]$. Therefore, since
	$$\lambda^+_0(2_{2j}1_22^*2_{2m+1}1_2)=[2;1_2,\alpha]+[0;2,\beta]$$ 	by Lemma 3 in \cite{CF} we have
$$\lambda^+_0(2_{2j}1_22^*2_{2m+1}1_2)=[2;1_2,\alpha]+[0;2,\beta]<3.$$
\end{proof}
Let $\mu_k^{(2)}:=\min\{\lambda_k^{(1)}, \lambda^-_0(1_22_{2j}1_22^*2_{2m+1}): j<m<k\}$. A direct consequence of Lemma \ref{c.AL2} is the following result: 

\begin{corollary}\label{c.Aevenodd}
If $\theta=...2_21_22_{a}1_22^*2_{b}1_22_2...$ with $a$ even, $b$ odd, and $2\leq a, b < 2k+1$ is $(k,\mu^{(2)}_k)$-admissible, then $b=a+1<2k+1$. 
\end{corollary}

In particular, we established the following statement: 

\begin{corollary}\label{c.EndofFirstStage}
If $\theta=...2_{2}1_22^*2_{2}...$ is $(k,\mu_k^{(2)})$-admissible then
\begin{enumerate}
\item[(a)] $\theta=...2_21_22_{2k}1_22^*2_{2k}1_22_2...$

\item[(b)] $\theta=...1_22_{2m}1_22^*2_{2m+1}1_22_2...$, with $m<k$

\item[(c)] $\theta=...2_21_22_{2m-1}1_22^*2_{2m}1_22_2...$ with $1<m<k-1$
\end{enumerate}
\end{corollary}
 
\subsection{Going to the Replication (Extensions of $2_21_22_{2k}1_22^*2_{2k}1_22_2$)}\label{s4}

In this subsection, we investigate for every $k\geq 2$ the extensions of a word $\theta$ containing the string 
\begin{equation} \label{t0} 
\theta^1_k:=2_21_22_{2k}1_22^*2_{2k}1_22_2. 
\end{equation}

Let $\tilde{\lambda}^{(1)}_k=\min \{\lambda^-_0(1_22_{2k-2}1_22^*2_{2k}),\lambda^-_0(2_21_22_{2k+1}1_22^*2_{2k-2}1_22_2)\}$. By Lemmas \ref{l.7} and \ref{l.8}, $\tilde{\lambda}^{(1)}_k>m(\gamma^1_k)$ and a $(k,\tilde{\lambda}^{(1)}_k)$-admissible word $\theta$ containing $\theta_k^1$ must extend as $$...\theta_k^12_{2k-2}....$$
\begin{lemma} \label{t1}
Let $\theta_k^1$ be the string in $(\ref{t0})$, then $\lambda_0^-(1_22_{2k-4}\theta^1_k2_{2k-2})>m(\gamma_k^1)$. In particular, $\lambda_0^-(1_22_{2j}1_22_{2k}1_22^*2_{2k}1_22_{2k})>m(\gamma_k^1)$, for each $1\leq j \leq k-1$.
\end{lemma}
\begin{proof}
The inequality $\lambda_0^-(1_22_{2j}1_22_{2k}1_22^*2_{2k}1_22_{2k})>\lambda_0^-(1_22_{2k-4}\theta^1_k2_{2k-2})$ is clear, for each $1\leq j \leq k-1$. Hence, it remains only to prove that $\lambda_0^-(1_22_{2k-4}\theta^1_k2_{2k-2})>m(\gamma_k^1)$. 
For this sake, let $\lambda^{-}_0(u)=A+B$, where $A=[2;2_{2k},1_2,2_{2k},\overline{2,1}]$ and $B=[0;1_2,2_{2k},1_2,2_{2k-2},1_2,\overline{2,1}]$. By definition, $m(\gamma^1_k)\leq C_k+D_k$, where $C_k =[2;2_{2k},1_2,2_{2k+2},1_2,2_2,\overline{1,2}]$  and
$D_k =[0;1_2,2_{2k},1_2,2_{2k+2},1_2,2_{2},\overline{1,2}].$

Thus, our work is reduced to prove that $A+B> C_k+D_k$. Note that
$$C_k-A=\dfrac{[2;1_2,2_2,\overline{1,2}]-[1;\overline{2,1}]}{q_{4k+3}^2([2;1_2,2_2,\overline{1,2}]+\beta_{4k+3})([1;\overline{2,1}]+\beta_{4k+3})},$$
while
$$B-D_k=\dfrac{[2;2_3,1_2,2_2,\overline{1,2}]-[1;1,\overline{2,1}]}{\tilde{q}_{4k+2}^2([1;1,\overline{2,1}]+\tilde{\beta}_{4k+2})([2;2_3,1_2,2_2,\overline{1,2}]+\tilde{\beta}_{4k+2})},$$
where $q_{4k+3}=q(2_{2k}1_22_{2k+1})$ and $\tilde{q}_{4k+2}=q(1_22_{2k}1_22_{2k-2})$.
Thus,
$$\dfrac{C_k-A}{B-D_k}=\dfrac{[2;1_2,2_2,\overline{1,2}]-[1;\overline{2,1}]}{[2;2_3,1_2,2_2,\overline{1,2}]-[1;1,\overline{2,1}]}\cdot X \cdot \dfrac{\tilde{q}^2_{4k+2}}{q^2_{4k+3}},$$
where
$$X=\dfrac{([1;1,\overline{2,1}]+\tilde{\beta}_{4k+2})([2;2_3,1_2,2_2,\overline{1,2}]+\tilde{\beta}_{4k+2})}{([2;1_2,2_2,\overline{1,2}]+\beta_{4k+3})([1;\overline{2,1}]+\beta_{4k+3	})}.$$
Let $\alpha=2_{2k}1_22_{2k-2}$, then $\tilde{q}_{4k+2}=p(\alpha)+2q(\alpha)$ and $q_{4k+2}=q(\alpha2_3)>12q(\alpha)$. Thus,
$$\dfrac{\tilde{q}_{4k+2}}{q_{4k+2}}<\dfrac{p(\alpha)+2q(\alpha)}{12q(\alpha)}<\dfrac{1}{4}.$$
By Lemma \ref{Le1}, $X\leq 4$ and therefore,  
$$\dfrac{C_k-A}{B-D_k}\leq 1.8 \cdot 4 \cdot \left(\dfrac{1}{4}\right)^2<1.$$
\end{proof}
Let  $\tilde{\lambda}^{(2)}_k=\min \{\lambda^-_0(1_22_{2k-4}\theta_k^12_{2k-2}),\lambda^-_0(2_{2k}1_22^*2_{2k-2}1_22_2)\}$.
By Lemmas \ref{t1} and Lemma \ref{l.10}, $\tilde{\lambda}^{(2)}_k>m(\gamma_k^1)$ and a $(k,\tilde{\lambda}^{(2)}_k)$-admissible word $\theta$ containing $\theta_k^12_{2k-2}$ must extend as 
\begin{equation*} \label{95}
...2_{2k-2}\theta_k^12_{2k-2}...=...2_{2k}1_22_{2k}1_22^*2_{2k}1_22_{2k}....
\end{equation*}

\begin{lemma}\label{t2}
If $\tilde{\lambda}_k^{(3)}:=\lambda_0^-(2_{2k-2}\theta_k^12_{2k-2}1_2)=\lambda_0^-(2_{2k}1_22_{2k}1_22^*2_{2k}1_22_{2k}1_2)$, then $\tilde{\lambda}_k^{(3)}>m(\gamma_k^1)$.
\end{lemma}
\begin{proof}
By definition,  $m(\gamma^1_k)\leq C_k+D_k$, where $C_k =[2;2_{2k},1_2,2_{2k+2},\overline{1,2}]$
and $D_k =[0;1_2,2_{2k},1_2,2_{2k+2},\overline{1,2}].$

Note that  $\lambda_0^-(2_{2k-2}\theta_k^12_{2k-2}1_2)=A_k+B_k$, where $A_k=[2;2_{2k},1_2,2_{2k},1_2,\overline{2,1}]$ and $B_k=[0;1_2,2_{2k},1_2,2_{2k},\overline{2,1}]$. Hence, our work is reduced to prove
that $A_k-C_k>D_k-B_k$.

In order to prove this inequality, we observe that
$$A_k-C_k=\dfrac{[2;2,\overline{2,1}]-[1;1,\overline{2,1}]}{q_{4k+2}^2([1;1,\overline{2,1}]+\beta_{4k+2})([2;2,\overline{1,2}]+\beta_{4k+2})},$$
and
$$D_k-B_k=\dfrac{[2;\overline{1,2}]-[1;\overline{2,1}]}{\tilde{q}_{4k+5}^2([2;\overline{1,2}]+\tilde{\beta}_{4k+5})([1;\overline{2,1}]+\tilde{\beta}_{4k+5})},$$
where $q_{4k+2}=q(2_{2k}1_22_{2k})$ and $\tilde{q}_{4k+5}=q(1_22_{2k}1_22_{2k+1})$.
Thus,
$$\dfrac{A_k-C_k}{D_k-B_k}=\dfrac{[2;2,\overline{2,1}]-[1;1,\overline{2,1}]}{[2;\overline{1,2}]-[1;\overline{2,1}]}\cdot Y \cdot \dfrac{\tilde{q}^2_{4k+5}}{q^2_{4k+2}},$$
where
$$Y=\dfrac{([2;\overline{1,2}]+\tilde{\beta}_{4k+5})([1;\overline{2,1}]+\tilde{\beta}_{4k+5})}{([1;1,\overline{2,1}]+\beta_{4k+2})([2;2,\overline{1,2}]+\beta_{4k+2})}.$$
Let $\alpha=2_{2k}1_22_{2k}$, then $\tilde{q}_{4k+5}=q(1_2\alpha2)>2q(1_2\alpha)=2(p(\alpha)+2q(\alpha))$. Thus, 
\begin{align*}
\dfrac{\tilde{q}_{4k+5}}{q_{4k+2}}  > 2 \cdot \dfrac{p(\alpha)+2q(\alpha)}{q(\alpha)}>4.
\end{align*}
By Lemma \ref{Le1}, $Y\geq 1/4$ and therefore, 
$$\dfrac{A_k-C_k}{D_k-B_k}> 0.46 \cdot 0.25 \cdot (4)^2>1.$$

\end{proof}
By Lemma \ref{t2} and Remark \ref{r.2} any $(k,\tilde{\lambda}_k^{(3)})$-admissible word $\theta$ containing $2_{2k-2}\theta_k^12_{2k-2}$ must to extend to right as
$$...2_{2k-2}\theta_k^12_{2k-1}=...2_{2k}1_22_{2k}1_22^*2_{2k}1_22_{2k+1}....$$
\begin{lemma}\label{t3}
If $\tilde{\lambda}_k^{(4)}:=\lambda_0^-(2_21_22_{2k-2}\theta_k^12_{2k})=\lambda_0^-(2_21_22_{2k}1_22_{2k}1_22^*2_{2k}1_22_{2k+2})$,  then $\tilde{\lambda}_k^{(4)}>m(\gamma_k^1)$
\end{lemma}
\begin{proof}
By definition, $\lambda^{-}_0(2_21_22_{2k-2}\theta_k^12_{2k})=A_k+B_k$, where $A_k=[2;2_{2k},1_2,2_{2k+2},\overline{2,1}]$ and $B_k=[0;1_2,2_{2k},1_2,2_{2k},1_2,2_2,\overline{2,1}]$. 
Moreover, $m(\gamma^1_k)\leq C_k+D_k$, where $C_k =[2;2_{2k},1_2,2_{2k+2},1_2,\overline{1,2}]$ and $D_k =[0;1_2,2_{2k},1_2,2_{2k+2},\overline{1,2}]$.

We shall show that $A_k+B_k> C_k+D_k$. In order to establish this inequality, we observe that 
$$C_k-A_k=\dfrac{[2;\overline{1,2}]-[1;1,\overline{1,2}]}{q_{4k+4}^2([1;1,\overline{1,2}]+\beta_{4k+4})([2;\overline{1,2}]+\beta_{4k+4})}$$
and
$$B_k-D_k=\dfrac{[2;2,\overline{1,2}]-[1;1,2_2,\overline{2,1}]}{\tilde{q}_{4k+4}^2([1;1,2_2,\overline{2,1}]+\tilde{\beta}_{4k+4})([2;2,\overline{1,2}]+\tilde{\beta}_{4k+4})},$$
where $q_{4k+4}=q(2_{2k}1_22_{2k+2})$ and $\tilde{q}_{4k+4}=q(1_22_{2k}1_22_{2k})$.
Thus,
$$\dfrac{C_k-A_k}{B_k-D_k}=\dfrac{[2;\overline{1,2}]-[1;1,\overline{1,2}]}{[2;2,\overline{1,2}]-[1;1,2_2,\overline{2,1}]}\cdot X \cdot \dfrac{\tilde{q}^2_{4k+4}}{q^2_{4k+4}},$$
where
$$X=\dfrac{([1;1,2_2,\overline{2,1}]+\tilde{\beta}_{4k+4})([2;2,\overline{1,2}]+\tilde{\beta}_{4k+4})}{([1;1,\overline{1,2}]+\beta_{4k+4})([2;\overline{1,2}]+\beta_{4k+4})}.$$
We have
$$X<\dfrac{([1;1,2_2,\overline{2,1}]+[0;2])([2;2,\overline{1,2}]+[0;2])}{([1;1,\overline{1,2}]+[0;\overline{2}]([2;\overline{1,2}]+[0;\overline{2}])}<1.1.$$
Let $\alpha=2_{2k}1_22_{2k}$, then $q_{4k+4}=q(\alpha2_2)>5q(\alpha)$ and $\tilde{q}_{4k+4}=q(1_2\alpha)=p(\alpha)+2q(\alpha)$. Thus, 
\begin{align*}
\dfrac{\tilde{q}_{4k+4}}{q_{4k+4}} < \dfrac{p(\alpha)+2q(\alpha)}{5q(\alpha)}<\dfrac{3}{5}.
\end{align*}
Therefore,
$$\dfrac{C_k-A_k}{B_k-D_k}< 1.76 \cdot 1.1 \cdot \left(\dfrac{3}{5}\right)^2<1.$$
\end{proof}

\begin{lemma} \label{t4}
Let $\theta_k^1$ be the string in $(\ref{t0})$, then
$$\lambda_0^+(2_{2k-1}\theta_k^12_{2k-1}1_22_2)<\lambda_0^+(2_21_22_{2k-2}\theta_k^12_{2k-1}1_22_2)<m(\theta(\underline{\omega}_k)).$$
\end{lemma}
\begin{proof}
Let $\lambda^{+}_0(2_21_22_{2k-2}\theta_k^12_{2k-1}1_22_2)=A_k+B_k$, where  $A_k=[2;2_{2k},1_2,2_{2k+1},1_2,2_{2},\overline{2,1}]$  and
$B_k=[0;1_2,2_{2k},1_2,2_{2k},1_2,2_{2},\overline{1,2}]$.
Furthermore, by definition,  $m(\theta(\underline{\omega})^k)\geq C_k+D_k$, where $C_k =[2;2_{2k},1_2,2_{2k+2},1_2,\overline{2,1}]$  and $D_k =[0;1_2,2_{2k},1_2,2_{2k+2},1_2,\overline{2,1}]$.

Thus, our task is prove that $C_k+D_k>A_k+B_k$. In order to establish this estimative, we observe that
$$C_k-A_k=\dfrac{[2;1_2,\overline{2,1}]-[1;1,2_2,\overline{2,1}]}{q_{4k+3}^2([2;1_2,\overline{2,1}]+\beta_{4k+3})([1;1,2_2,\overline{2,1}]+\beta_{4k+3})}$$
and
$$B_k-D_k=\dfrac{[2;2,1_2,\overline{2,1}]-[1;1,2_2,\overline{1,2}]}{\tilde{q}_{4k+4}^2([1;1,2_2,\overline{1,2}]+\tilde{\beta}_{4k+4})([2;2,1_2,\overline{2,1}]+\tilde{\beta}_{4k+4})},$$
where $q_{4k+3}=q(2_{2k}1_22_{2k+1})$ and $\tilde{q}_{4k+4}=q(1_22_{2k}1_22_{2k})$.
Thus,
$$\dfrac{C_k-A_k}{B_k-D_k}=\dfrac{[2;1_2,\overline{2,1}]-[1;1,2_2,\overline{2,1}]}{[2;2,1_2,\overline{2,1}]-[1;1,2_2,\overline{1,2}]}\cdot Y \cdot \dfrac{\tilde{q}^2_{4k+4}}{{q}^2_{4k+3}},$$
where
$$Y=\dfrac{([1;1,2_2,\overline{1,2}]+\tilde{\beta}_{4k+4})([2;2,1_2,\overline{2,1}]+\tilde{\beta}_{4k+4})}{([2;1_2,\overline{2,1}]+\beta_{4k+3})([1;1,2_2,\overline{2,1}]+\beta_{4k+3})}.$$
Note that
$$Y\geq \dfrac{([1;1,2_2,\overline{1,2}]+[0;\overline{2}])([2;2,1_2,\overline{2,1}]+[0;\overline{2}])}{([2;1_2,\overline{2,1}]+[0;\overline{2}])([1;1,2_2,\overline{2,1}]+[0;\overline{2}])}>0.93.$$
Let $\alpha=2_{2k}1_22_{2k}$, then $q_{4k+3}=q(\alpha2)=2q(\alpha)+q(2_{2k}1_22_{2k-1})<(2+1/2)q(\alpha)$ and $\tilde{q}_{4k+4}=p(\alpha)+2q(\alpha)$. Thus,
\begin{align*}
\dfrac{\tilde{q}_{4k+4}}{q_{4k+3}}> \dfrac{2}{5} \cdot \dfrac{p(\alpha)+2q(\alpha)}{q(\alpha)}=\dfrac{2}{5} \cdot ([0,\alpha]+2)= \dfrac{2}{5}\cdot [2,\overline{2}]>0.96.
\end{align*}
Therefore,
$$\dfrac{C_k-A_k}{B_k-D_k} > 1.26 \cdot 0.93 \cdot \left(0.96\right)^2>1.$$
\end{proof}

By Lemmas \ref{t3}, \ref{t4} and Remark \ref{r.2} any $(k,\tilde{\lambda}_k^{(4)})$-admissible word $\theta$ containing $2_{2k-2}\theta_k^12_{2k-1}$ must to extend as
$$...2_{2k-1}\theta_k^12_{2k}...=...2_{2k+1}1_22_{2k}1_22^*2_{2k}1_22_{2k+2}....$$
\begin{lemma} \label{t5}
$\lambda_0^+(2_21_22_{2k-1}\theta_k^12_{2k})=\lambda_0^+(2_21_22_{2k+1}1_22_{2k}1_22^*2_{2k}1_22_{2k+2})<m(\theta(\underline{\omega}_k))$.
\end{lemma}
\begin{proof}
Let $\lambda^{+}_0(2_21_22_{2k-1}\theta_k^12_{2k})=A_k+B_k$, where $A_k=[2;2_{2k},1_2,2_{2k+2},\overline{1,2}]$ and $B_k=[0;1_2,2_{2k},1_2,2_{2k+1},1_2,2_{2},\overline{2,1}].$ Moreover, by definition, $m(\theta(\underline{\omega}_k))\geq C_k+D_k$, where 
$C_k =[2;2_{2k},1_2,2_{2k+2},1_2,2_{2},\overline{2,1}]$ and 
$D_k =[0;1_2,2_{2k},1_2,2_{2k+2},1_2,\overline{2,1}]$.

Let us show that $A_k+B_k< C_k+D_k$. For this sake, we observe that 
$$A_k-C_k=\dfrac{[2;\overline{1,2}]-[1;2_2,\overline{2,1}]}{{q}_{4k+5}^2([2;\overline{1,2}]+\beta_{4k+5})([1;2_2,\overline{2,1}]+\beta_{4k+5})}$$
and
$$D_k-B_k=\dfrac{[2;1_2,\overline{2,1}]-[1;1,2_2,\overline{2,1}]}{\tilde{q}_{4k+5}^2([2;1_2,\overline{2,1}]+\tilde{\beta}_{4k+5})([1;1,2_2,\overline{2,1}]+\tilde{\beta}_{4k+5})},$$
where $\tilde{q}_{4k+5}=q(1_22_{2k}1_22_{2k+1})$ and ${q}_{4k+5}= q(2_{2k}1_22_{2k+2}1)$.
Thus,
$$\dfrac{A_k-C_k}{D_k-B_k}=\dfrac{[2;\overline{1,2}]-[1;2_2,\overline{2,1}]}{[2;1_2,\overline{2,1}]-[1;1,2_2,\overline{2,1}]}\cdot X \cdot \dfrac{\tilde{q}^2_{4k+5}}{{q}^2_{4k+5}},$$
where
$$X=\dfrac{([2;1_2,\overline{2,1}]+\tilde{\beta}_{4k+5})([1;1,2_2,\overline{2,1}]+\tilde{\beta}_{4k+5})}{([2;\overline{1,2}]+\beta_{4k+5})([1;2_2,\overline{2,1}]+\beta_{4k+5})}.$$
Note that
$$X <\dfrac{([2;1_2,\overline{2,1}]+[0;\overline{2}])([1;1,2_2,\overline{2,1}]+[0;\overline{2}])}{([2;\overline{1,2}]+[0;1,2])([1;2_2,\overline{2,1}]+[0;1,2])}<0.9.$$
Let $\alpha=2_{2k}1_22_{2k+1}$, since $q(2_{2k}1_22_{2k})>(1/3)q(\alpha)$, we have:
\begin{align*}
q_{4k+5}= q(\alpha21)=q(\alpha2)+q(\alpha)=3q(\alpha)+q(2_{2k}1_22_{2k})>(10/3)q(\alpha).
\end{align*}
Thus, 
$$\dfrac{\tilde{q}_{4k+5}}{q_{4k+5}}< \dfrac{3}{10} \cdot  \dfrac{p(\alpha)+2q(\alpha)}{q(\alpha)}< \dfrac{3}{10}([0;\alpha]+2)<\dfrac{3}{10}\cdot[2;2]=0.75.$$
Therefore,
$$\dfrac{A_k-C_k}{D_k-B_k}< 1.52 \cdot 0.9 \cdot \left(0.75\right)^2<1.$$
\end{proof}

\begin{lemma} \label{t6}
$\lambda_0^+(2_{2k}\theta_k^12_{2k+1})=\lambda_0^+(2_{2k+2}1_22_{2k}1_22^*2_{2k}1_22_{2k+3})<m(\theta(\underline{\omega}_k))$.
\end{lemma}
\begin{proof}
By definition, $\lambda^{+}_0(2_{2k}\theta_k^12_{2k+1})=A_k+B_k$, where $A_k=[2;2_{2k},1_2,2_{2k+3},\overline{2,1}]$ and $B_k=[0;1_2,2_{2k},1_2,2_{2k+2},\overline{1,2}]$. Moreover, $m(\theta(\underline{\omega}_k))\geq C_k+D_k$, where
$C_k =[2;2_{2k},1_2,2_{2k+2},1_2,2_{2},\overline{2,1}]$ and
$D_k =[0;1_2,2_{2k},1_2,2_{2k+2},1_2,2_{2},\overline{2,1}]$.

We shall show that $A_k+B_k< C_k+D_k$. In order to establish this inequality, we observe that 
$$C_k-A_k=\dfrac{[2;\overline{2,1}]-[1;1,2_2,\overline{2,1}]}{q_{4k+4}^2([1;1,2_2,\overline{2,1}]+\beta_{4k+4})([2;\overline{2,1}]+\beta_{4k+4})}$$
and
$$B_k-D_k=\dfrac{[2;\overline{1,2}]-[1;2_2,\overline{2,1}]}{\tilde{q}_{4k+7}^2([2;\overline{2,1}]+\tilde{\beta}_{4k+7})([1;2_2,\overline{2,1}]+\tilde{\beta}_{4k+7})},$$
where $q_{4k+4}=q(2_{2k}1_22_{2k+2})$ and $\tilde{q}_{4k+7}=q(1_22_{2k}1_22_{2k+2}1)$.

Thus,
$$\dfrac{C_k-A_k}{B_k-D_k}=\dfrac{[2;\overline{2,1}]-[1;1,2_2,\overline{2,1}]}{[2;\overline{1,2}]-[1;2_2,\overline{2,1}]}\cdot Y \cdot \dfrac{\tilde{q}^2_{4k+7}}{q^2_{4k+4}},$$
where
$$Y=\dfrac{([2;\overline{2,1}]+\tilde{\beta}_{4k+7})([1;2_2,\overline{2,1}]+\tilde{\beta}_{4k+7})}{([1;1,2_2,\overline{2,1}]+\beta_{4k+4})([2;\overline{2,1}]+\beta_{4k+4})}.$$
Let $\alpha=2_{2k}1_22_{2k+2}$ and $\tilde{\alpha}=1_2\alpha$, since that $q(1_22_{2k}1_22_{2k+1})>(1/3)q(\tilde{\alpha})$, we have
$$\tilde{q}_{4k+7}=q(\tilde{\alpha}1)=q(\tilde{\alpha})+q(1_22_{2k}1_22_{2k+1})>(4/3)q(\tilde{\alpha}).$$ 
Thus, 
\begin{align*}
\dfrac{\tilde{q}_{4k+7}}{q_{4k+4}}  > \dfrac{4}{3}\cdot\dfrac{p(\alpha)+2q(\alpha)}{q(\alpha)}=\dfrac{4}{3}\left([0;\alpha]+2\right)> \dfrac{4}{3}\cdot[2;\overline{2}]>3.2.
\end{align*}
Therefore, since that $Y \geq 0.25$, by Lemma \ref{Le1}, we obtain
$$\dfrac{C_k-A_k}{B_k-D_k}> 0.49 \cdot 0.25 \cdot \left(3.2\right)^2>1.$$
\end{proof}
Let $\tilde{\lambda}_k^{(5)}=\min\{\lambda^-_0(1_22_{2k+2}1_22^*2_{2k-2}1_22_2),\lambda^-_0(1_22_{2k-2}1_22^*2_{2k+1})\}$. We have $\tilde{\lambda}_k^{(5)}>m(\gamma_k^1)$ from Lemmas \ref{l.16} and \ref{l.15}. Moreover, these lemmas, Lemmas \ref{t5} and \ref{t6} (and Remark \ref{r.2}) imply that any $(k,\tilde{\lambda}_k^{(5)})$-admissible $\theta$ containing $2_{2k-1}\theta_k^12_{2k}$ must extend as
$$...2_{2k}\theta_k^12_{2k}1_22_{2k}...=...2_{2k+2}1_22_{2k}1_22^*2_{2k}1_22_{2k+2}1_22_{2k}....$$
 
\begin{lemma} \label{t7}
$\lambda_0^+(2_{2k+1}\theta_k^12_{2k}1_22_{2k})=\lambda_0^+(2_{2k+3}1_22_{2k}1_22^*2_{2k}1_22_{2k+2}1_22_{2k})<m(\theta(\underline{\omega}_k))$.
\end{lemma}
\begin{proof}
By definition,  $m(\theta(\underline{\omega}_k))\geq C_k+D_k$, where $C_k=[2;2_{2k},1_2,2_{2k+2},1_2,2_{2k},1_2,2_2,\overline{2,1}]$ and $D_k =[0;1_2,2_{2k},1_2,2_{2k+2},1_2,2_{2},\overline{2,1}].$ Note that  $\lambda_0^+(2_{2k+1}\theta_k^12_{2k}1_22_{2k})=A_k+B_k$, where $A_k=[2;2_{2k},1_2,2_{2k+2},1_2,2_{2k},\overline{1,2}]$ and $B_k=[0;1_2,2_{2k},1_2,2_{2k+3},\overline{2,1}]$.

Hence, our work is reduced to prove that $A_k+B_k<C_k+D_k$. In order to prove this inequality, we observe that
$$A_k-C_k=\dfrac{[2;\overline{1,2}]-[1;2_2,\overline{2,1}]}{q_{6k+7}^2([2;\overline{1,2}]+\beta_{6k+7})([1;2_2,\overline{2,1}]+\beta_{6k+7})}$$
and
$$D_k-B_k=\dfrac{[2;\overline{2,1}]-[1;1,2_2,\overline{2,1}]}{\tilde{q}_{4k+6}^2([1;1,2_2,\overline{2,1}]+\tilde{\beta}_{4k+6})([2;\overline{2,1}]+\tilde{\beta}_{4k+6})},$$
where $q_{6k+7}=q(2_{2k}1_22_{2k+2}1_22_{2k}1)$ and $\tilde{q}_{4k+6}=q(1_22_{2k}1_22_{2k+2})$.

Thus,
$$\dfrac{A_k-C_k}{D_k-B_k}=\dfrac{[2;\overline{1,2}]-[1;2_2,\overline{2,1}]}{[2;\overline{2,1}]-[1;1,2_2,\overline{2,1}]}\cdot X \cdot \dfrac{\tilde{q}^2_{4k+6}}{q^2_{6k+7}},$$
where
$$X=\dfrac{([1;1,2_2,\overline{2,1}]+\tilde{\beta}_{4k+6})([2;\overline{2,1}]+\tilde{\beta}_{4k+6})}{([2;\overline{1,2}]+\beta_{6k+7})([1;2_2,\overline{2,1}]+\beta_{6k+7})}.$$

Let $\alpha=2_{2k}1_22_{2k+2}$, since that $2k\geq 4$, we have $q_{6k+7}=q(\alpha1_22_{2k}1)>2^4q(\alpha)$. Thus, 
\begin{align*}
\dfrac{\tilde{q}_{4k+6}}{q_{6k+7}}  < \dfrac{p(\alpha)+2q(\alpha)}{2^4q(\alpha)}<\dfrac{3}{16}.
\end{align*}
By Lemma \ref{Le1}, we have $X<4$ and therefore, 
$$\dfrac{A_k-C_k}{D_k-B_k}< 2.1 \cdot 4 \cdot \left(\dfrac{3}{16}\right)^2<1.$$

\end{proof}

\begin{lemma} \label{t8}
$\lambda_0^+(2_{2k}1_22_{2k}\theta_k^12_{2k}1_22_{2k+1})=\lambda_0^+(2_{2k}1_22_{2k+2}1_22_{2k}1_22^*2_{2k}1_22_{2k+2}1_22_{2k+1})<m(\theta(\underline{\omega}_k))$.
\end{lemma}
\begin{proof}
Let $\lambda^{+}_0(2_{2k}1_22_{2k}\theta_k^12_{2k}1_22_{2k+1})=A_k+B_k$, where $A_k=[2;2_{2k},1_2,2_{2k+2},1_2,2_{2k+1},1_2,\overline{2,1}]$ and $B_k=[0;1_2,2_{2k},1_2,2_{2k+2},1_2,2_{2k},\overline{1,2}]$.
Moreover, $m(\theta(\omega_k))\geq C_k+D_k$, where 
\begin{align*}
C_k &=[2;2_{2k},1_2,2_{2k+2},1_2,2_{2k},1_2,2_2,\overline{2,1}] \textrm{ and }\\
D_k &=[0;1_2,2_{2k},1_2,2_{2k+2},1_2,2_{2k+1},1_2,2_{2},\overline{1,2}].
\end{align*}
Let us show that $A_k+B_k< C_k+D_k$. For this sake, we observe that
$$C_k-A_k=\dfrac{[2;1_2,\overline{2,1}]-[1;1,2_2,\overline{2,1}]}{q_{6k+6}^2([1;1,2_2,\overline{2,1}]+\beta_{6k+6})([2;1_2,\overline{2,1}]+\beta_{6k+6})},$$
while
$$B_k-D_k=\dfrac{[2;1_2,2_2,\overline{1,2}]-[1;\overline{2,1}]}{\tilde{q}_{6k+8}^2([1;\overline{2,1}]+\tilde{\beta}_{6k+8})([2;1_2,2_2,\overline{1,2}]+\tilde{\beta}_{6k+8})},$$
where $q_{6k+6}=q(2_{2k}1_22_{2k+2}1_22_{2k})$ and $\tilde{q}_{6k+8}=q(1_22_{2k}1_22_{2k+2}1_22_{2k})$.

Thus,
$$\dfrac{C_k-A_k}{B_k-D_k}=\dfrac{[2;1_2,\overline{2,1}]-[1;1,2_2,\overline{2,1}]}{[2;1_2,2_2,\overline{1,2}]-[1;\overline{2,1}]}\cdot Y \cdot \dfrac{\tilde{q}^2_{6k+8}}{q^2_{6k+6}},$$
where
$$Y=\dfrac{([1;\overline{2,1}]+\tilde{\beta}_{6k+8})([2;1_2,2_2,\overline{1,2}]+\tilde{\beta}_{6k+8})}{([1;1,2_2,\overline{2,1}]+\beta_{6k+6})([2;1_2,\overline{2,1}]+\beta_{6k+6})}.$$
Let $\alpha=2_{2k}1_22_{2k+2}1_22_{2k}$, then
$$\dfrac{\tilde{q}_{6k+8}}{q_{6k+6}}=\dfrac{p(\alpha)+2q(\alpha)}{q(\alpha)}=2+[0;\alpha]>2+[0;\overline{2}]>2.41.$$

By Lemma \ref{Le1}, we have $Y \geq 1/4$ and therefore,
$$\dfrac{C_k-A_k}{B_k-D_k}>0.7 \cdot 0.25 \cdot \left(2,41 \right)^2>1.$$
\end{proof}

Let $\tilde{\lambda}_k^{(5)}>m(\gamma_k^1)$ be as before. By Lemma \ref{t7} and Remark \ref{r.2}, a $(k,\tilde{\lambda}_k^{(5)})$-admissible word $\theta$ containing $2_{2k}\theta_k^12_{2k}1_22_{2k}$ extend as $2_21_22_{2k}\theta_k^12_{2k}1_22_{2k}$. By Lemmas \ref{l.16} and \ref{l.15}, $\theta$ must keeping extending as $2_{2k}1_22_{2k}\theta_k^12_{2k}1_22_{2k}$. Finally, by Lemma \ref{t8}, $\theta$ must keeping extending as
$$2_{2k}1_22_{2k}\theta_k^12_{2k}1_22_{2k}1=...2_{2k}1_22_{2k+2}1_22_{2k}1_22^*2_{2k}1_22_{2k+2}1_22_{2k}1....$$

The full discussion of this subsection can be compiled into the following lemma establishing that a word $\theta$ containing the right string $2_21_22_{2k}1_22^*2_{2k}1_22_2$ must extend until the beginning of replication mechanism:
\begin{lemma}[Going to the Replication]\label{gr}
For every $k \geq 2$, there exists a explicit constant $\tilde{\lambda}_k>m(\gamma_k^1)$ such that any $(k,\tilde{\lambda}_k)$-admissible word $\theta$ containing $\theta_k^1:=2_21_22_{2k}1_22^*2_{2k}1_22_2$ must extend as 
$$...2_{2k}1_22_{2k+2}1_22_{2k}1_22^*2_{2k}1_22_{2k+2}1_22_{2k}1....$$   
\end{lemma}
\begin{proof}
This result for $\tilde{\lambda}_k:=\min\{\tilde{\lambda}_k^{(i)}: i=1,...,5\}$ is a consequence of this subsection.
\end{proof}

By Corollary \ref{c.EndofFirstStage} and Lemma \ref{gr}, we have:
\begin{lemma}\label{l.semi-local-main} If $\theta=...2_{2}1_22^*2_{2}...$ is $(k,\min\{\mu_k^{(2)},\tilde{\lambda}_k\})$-admissible, then
\begin{enumerate}
\item $\dots 2_{2k}1_22_{2k+2}1_22_{2k}1_22^*2_{2k}1_22_{2k+2}1_22_{2k}1\dots$ or 
 
\item $\theta=...1_22_{2m}1_22^*2_{2m+1}1_22_2...$, with $m<k$ or

\item $\theta=...2_21_22_{2m-1}1_22^*2_{2m}1_22_2...$ with $1<m<k-1$.
\end{enumerate}
\end{lemma}

\subsection{Conclusion: local almost uniqueness} As it was announced in the beginning of this section, Lemmas \ref{l.9-1} and \ref{l.semi-local-main} give us the following local almost uniqueness property for $\gamma_k^1$:
\begin{theorem} There exists an explicit constant $\mu_k>m(\gamma_k^1)$ such that any $(k,\mu_k)$-admissible word has the form 
\begin{itemize}
\item $\theta=\dots 1_4 2^*21_2\dots$ or 
\item $\theta=\dots 2_{2k}1_22_{2k+2}1_22_{2k}1_22^*2_{2k}1_22_{2k+2}1_22_{2k}1\dots$ or 
 \item $\theta=\dots 1_22_{2m}1_22^*2_{2m+1}1_22_2\dots$ with $m<k$ or 
 \item $\theta=\dots 2_21_22_{2m-1}1_22^*2_{2m}1_22_2\dots$ with $1<m<k-1$.
\end{itemize}
Moreover, any $(k,\mu_k)$-admissible word $\theta=\dots 1_4 2^*21_2\dots$ can not be extended as  $1_{2j}2^*21_{2m}$ with $\lfloor (2k-1)\log 3/\log 2\rfloor+3<j\leq m+1$. 
 \end{theorem}









\bibliographystyle{amsplain}

\end{document}